\documentclass[11pt]{article}
\usepackage[width=16cm,top = 3cm,bottom=3cm]{geometry}
\usepackage{amsthm,amsfonts,amssymb,amsmath,amscd, physics,mathtools,thmtools, MnSymbol,stmaryrd, comment} 
\usepackage{enumitem}
\usepackage{xcolor}
\pagecolor{white}
\usepackage{float}
\usepackage{adjustbox}
\usepackage{lipsum}
\usepackage{wrapfig}
\usepackage{pgfplots}
\pgfplotsset{compat=1.17}
\usepackage{tikz}
\usetikzlibrary{decorations.pathreplacing,decorations.markings}
\usetikzlibrary{arrows.meta}
\usepackage{tikz-cd}
\tikzcdset{scale cd/.style={every label/.append style={scale=#1},
    cells={nodes={scale=#1}}}}
\usepackage[all,cmtip,pdf]{xy}
\usetikzlibrary{positioning}
\usetikzlibrary{decorations.pathmorphing,calligraphy}
\tikzset{
  on each segment/.style={
    decorate,
    decoration={
      show path construction,
      moveto code={},
      lineto code={
        \path [#1]
        (\tikzinputsegmentfirst) -- (\tikzinputsegmentlast);
      },
      curveto code={
        \path [#1] (\tikzinputsegmentfirst)
        .. controls
        (\tikzinputsegmentsupporta) and (\tikzinputsegmentsupportb)
        ..
        (\tikzinputsegmentlast);
      },
      closepath code={
        \path [#1]
        (\tikzinputsegmentfirst) -- (\tikzinputsegmentlast);
      },
    },
  },
  mid arrow/.style={postaction={decorate,decoration={
        markings,
        mark=at position .5 with {\arrow[#1]{Stealth[scale width=1.5]}}
      }}},
}
\usetikzlibrary{calc,intersections}
\usepackage{color}
\usepackage{hyperref}
\input{insbox.tex}
\definecolor{tocolor}{rgb}{.1,.1,.1}
\definecolor{urlcolor}{rgb}{.2,.2,.6}
\definecolor{linkcolor}{rgb}{.1,.1,.5}
\definecolor{citecolor}{rgb}{.4,.2,.1}
\hypersetup{colorlinks=true, urlcolor=urlcolor, linkcolor=linkcolor, citecolor=citecolor}

\usepackage[maxbibnames=5,minbibnames=4,backend=bibtex,isbn=false]{biblatex}
\long\def\authornote#1{%
	\leavevmode\unskip\raisebox{-3.5pt}{\rlap{$\scriptstyle\diamond$}}%
	\marginpar{\raggedright\hbadness=10000
		\def\baselinestretch{0.8}\tiny
		\it #1\par}}
\newcommand{\marco}[1]{\authornote{MG: #1}}
\newcommand{\yucong}[1]{\authornote{YJ: #1}}
\newcommand{\daniel}[1]{\authornote{DA: #1}}

\newtheorem{theorem}{Theorem}[section]
\newtheorem{lemma}[theorem]{Lemma}
\newtheorem{proposition}[theorem]{Proposition}
\newtheorem{corollary}[theorem]{Corollary}
\newtheorem{definition}[theorem]{Definition}
\theoremstyle{definition}
\newtheorem{remark}[theorem]{Remark}
\newtheorem{example}[theorem]{Example}
\AtBeginEnvironment{example}{%
  \pushQED{\qed}%
}
\AtEndEnvironment{example}{\popQED\endexample}

\newcommand{\R}{\mathbb{R}}
\newcommand{\C}{\mathbb{C}}
\newcommand{\Z}{\mathbb{Z}}

\renewcommand{\a}{\alpha}
\renewcommand{\b}{\beta}
\newcommand{\sss}{\mathtt{s}}
\newcommand{\ttt}{\mathtt{t}}
\newcommand{\mm}{\mathtt{m}}
\newcommand{\ii}{\mathtt{i}}
\newcommand{\uu}{\mathtt{u}}
\newcommand{\g}{\mathfrak{g}}

\renewcommand{\d}{\partial}
\newcommand{\dbar}{\bar{\partial}}

\renewcommand{\ker}{\mathrm{ker}}

\renewcommand{\L}{\mathcal{L}}

\newcommand{\Gr}{\mathrm{Gr}}
\newcommand{\E}{\mathcal{E}}
\newcommand{\J}{\mathbb{J}}
\newcommand{\G}{\mathbb{G}}
\newcommand{\A}{\mathcal{A}}
\newcommand{\B}{\mathcal{B}}
\newcommand{\F}{\mathcal{F}}
\renewcommand{\H}{\mathcal{H}}
\newcommand{\DT}{\mathbb{T}}
\renewcommand{\l}{\ell}

\newcommand{\id}{\mathrm{id}}

\renewcommand{\Re}{\mathrm{Re}}
\renewcommand{\Im}{\mathrm{Im}}
\newcommand{\IP}[1]{\langle #1 \rangle}
\newcommand{\cl}{\mathrm{cl}}
\newcommand{\delbar}{\bar\partial}
\newcommand{\del}{\partial}
\renewcommand{\hom}{\mathrm{Hom}}

\title{Symplectic double groupoids and the generalized Kähler potential} 
\author{Daniel Álvarez\thanks{Department of Mathematics, University of Toronto; \url{dalv@math.toronto.edu}} \quad Marco Gualtieri\thanks{Department of Mathematics, University of Toronto; \url{mgualt@math.toronto.edu}} \quad Yucong Jiang\thanks{{Department of Mathematics, University of Illinois Urbana-Champaign; \url{yucongj2@illinois.edu}}}}
\date{} 

\begin{document}
	
	\maketitle
 \vspace{2em}
	\begin{abstract}
            A description of the fundamental degrees of freedom underlying generalized K\"ahler geometry, which separates its holomorphic moduli from its compatible Riemannian metric in a similar way to the K\"ahler case, has been sought since its discovery in 1984.  In this paper, we describe a full solution to this problem for arbitrary generalized K\"ahler manifolds.  We discover that the holomorphic structure underlying a generalized K\"ahler manifold is a holomorphic symplectic Morita double bimodule between double symplectic groupoids, and that each compatible Riemannian metric is given by a Lagrangian submanifold forming a bisection of the real symplectic core of this double bimodule. In other words, a generalized K\"ahler manifold has an associated holomorphic symplectic manifold of quadruple dimension and equipped with an anti-holomorphic involution; the generalized K\"ahler metric is then determined by the choice of a Lagrangian submanifold of the fixed point locus of this involution.  This resolves affirmatively a long-standing conjecture by physicists concerning the existence of a generalized K\"ahler potential. 

            We demonstrate the theory by constructing explicitly the above Morita double bimodule and Lagrangian bisection for the well-known generalized K\"ahler structures on compact even-dimensional semisimple Lie groups, which have until now escaped such analysis. We construct the required holomorphic symplectic manifolds by expressing them as moduli spaces of flat connections on surfaces with Lagrangian boundary conditions, through a quasi-Hamiltonian reduction.  
            
\end{abstract}

        \setcounter{tocdepth}{1}
	\tableofcontents
\vspace{2em}
\pagebreak 
\section*{Introduction} In 1984, physicists Gates, Hull and Ro\v{c}ek discovered what was later called \emph{generalized K\"ahler geometry}, a structure on a Riemannian manifold which, when used as the target space of a 2-dimensional sigma model, endows it with $N=(2,2)$ supersymmetry~\cite{MR776369}. This extended earlier work of Zumino, who in 1979 observed that a K\"ahler structure has this same effect, precipitating an explosion of interest in K\"ahler manifolds by physicists, leading to a series of breakthroughs which continues to this day, mostly under the umbrella of Mirror Symmetry. 

A generalized K\"ahler structure on a Riemannian manifold $(M,g)$ is a pair of complex structures $(I_+,I_-)$, each orthogonal with respect to the metric $g$, whose corresponding Hermitian forms $\omega_\pm = gI_\pm$ are pluriclosed, i.e. $dd^c_+ \omega_+ = dd^c_-\omega_-= 0$, and which satisfy 
\begin{equation}\label{comppm}
d^c_+\omega_+ + d^c_-\omega_- = 0.
\end{equation}
Solutions of this system may be found by deforming K\"ahler structures, but there are also solutions in which the closed 3-form $H=d^c_+\omega_+$ has nontrivial de Rham class, contradicting the $\partial\bar\partial-$lemma and implying that the complex manifolds $X_\pm=(M,I_\pm)$ admit no K\"ahler structures.  For this reason, generalized K\"ahler geometry is of particular interest in the study of nonalgebraic complex manifolds such as the Calabi-Eckmann structure on $S^3\times S^3$.

What was not understood at the time, but became clear with the advent of Hitchin and Weinstein's theories of generalized complex and Dirac geometry, is that a generalized K\"ahler structure has the character of an infinitesimal object, analogous to that of a Lie algebra in comparison to its corresponding Lie group.    The main problem we solve in this paper is to determine the structure of the \emph{integration}, or group-like object, underlying generalized K\"ahler geometry.  We also provide an explicit moduli-theoretic construction of it in the important case of the generalized K\"ahler structures on even-dimensional compact semisimple Lie groups.  As a corollary, we resolve one of the main conjectures in generalized K\"ahler geometry, first stated in the original 1984 paper, about the existence of a generalized K\"ahler potential function.  

\subsection{Summary of the main Theorem}
Our solution is formulated in terms of symplectic double groupoids, introduced by Weinstein and developed by Lu and Mackenzie.  We show that each of the complex manifolds $X_\pm$ described above is the base of a holomorphic symplectic double Lie groupoid $D_\pm$, and that these double groupoids, together with their complex conjugates, are related by a loop of 1-morphisms (multiplicative Morita equivalences) as shown in Fig.~\ref{mor2}.
In addition, there is a 2-morphism $\mathcal{C}$ (Morita 2-equivalence) filling the square.  
\begin{figure}[H]
  \centering
\begin{tikzpicture}[scale=.6]
\pgfdeclarelayer{background}
\pgfdeclarelayer{foreground}
\pgfsetlayers{background,main,foreground}
 \coordinate (a) at (-1,-1);
 \coordinate (b) at (1,-1);
 \coordinate (c) at (1,1);
 \coordinate (d) at (-1,1);
 \coordinate (e) at (0,0);
\coordinate (f) at (0,1);
\coordinate (g) at (0,-1);
\coordinate (h) at (1,0);
\coordinate (i) at (-1,0);
 \begin{pgfonlayer}{background}
      \draw[fill=black!10,postaction={on each segment={mid arrow=black}}] (b)--(c)--(d);
      \draw[fill=black!10,postaction={on each segment={mid arrow=black}}] (b)--(a)--(d);
            \end{pgfonlayer}
 \node [below left] at (a)
                {$\overline{D}_+$};
\node [circle, fill=black, inner sep=1pt] at (a){};
\node [below right] at (b)
                {$D_-$};
\node [circle, fill=black, inner sep=1pt] at (b){};                
\node [above right] at (c)
                {$D_+$};
\node [circle, fill=black, inner sep=1pt] at (c){};
\node [above left] at (d)
                {$\overline{D}_-$};
\node [circle, fill=black, inner sep=1pt] at (d){};
\node  at (e)
                {$\mathcal{C}$};
\node [above] at (f)
                {$\mathcal{W}$};
\node[below] at (g)
                {$\overline{\mathcal{W}}$};
\node[right] at (h)
                {$\mathcal{Z}$};
\node[left] at (i)
                {$\overline{\mathcal{Z}}$};
\end{tikzpicture} \hspace{2mm}
\caption{A holomorphic symplectic Morita 2-equivalence}\label{mor2}
\end{figure}
The structures $(D_\pm,\mathcal{Z},\mathcal{W},\mathcal{C})$ described so far are all holomorphic symplectic manifolds of quadruple the dimension of $X_\pm$, and  have as many groupoid structures as their codimension in the diagram ($D_\pm$ each have two groupoid structures, $\mathcal{Z}$ and $\mathcal{W}$ have one, and $\mathcal{C}$ has none), and act on each other if they are related by face inclusion in the diagram (e.g., $D_\pm$ act on $\mathcal{Z}$ and $\mathcal{W}$, both of which act on $\mathcal{C}$).   They capture information analogous to the complex structure and K\"ahler class of a K\"ahler manifold.  Furthermore,  in analogy with the fact that K\"ahler classes are real, the above diagram is equipped with a real structure, which restricts to an anti-holomorphic symplectic involution $\tau:\mathcal{C}\to \overline{\mathcal{C}}$.  
The isomorphism class of $(\mathcal{C},\tau)$ is what we call the \emph{generalized K\"ahler class}, a nonlinear generalization of the usual K\"ahler class in $H^{1,1}(M,\R)$.  

Finally, a generalized K\"ahler metric within this class is given by the choice of a trivialization of $(\mathcal{C},\tau)$,  which consists of a Lagrangian bisection $\mathcal{L}$ in the \emph{real symplectic core} $\mathcal{S}\subset\mathcal{C}$, defined as the fixed point locus of $\tau$. 
In fact, the real symplectic core 
defines a symplectic Morita equivalence between the real parts $(X_\pm^\R,\sigma_\pm^\R)$ of the holomorphic Poisson structures $\sigma_\pm$ on $X_\pm$ described by Hitchin~\cite{MR2217300}, as shown in Fig.~\ref{sympme}.
Such a Lagrangian bisection $\mathcal{L}$ induces a diffeomorphism between the underlying real manifolds of $X_\pm$, and identifies the real parts of the Hitchin Poisson structures, i.e. $\sigma^\R_+ = \sigma^\R_- = -\tfrac{1}{4}[I_+,I_-]g^{-1}$.

\begin{figure}
    \centering
\[\begin{tikzcd}        
& (\mathcal{S},\omega_\mathcal{S})\arrow{dl}\arrow{dr} & \\
(X_+^\R,\sigma_+^\R)&    \mathcal{L} \arrow[hook]{u} & (X_-^\R,\sigma_-^\R)
\end{tikzcd}\]
    \caption{Lagrangian bisections $\mathcal{L}$ of $\mathcal{S}$ correspond to generalized K\"ahler metrics}
    \label{sympme}
\end{figure}
Our main result, which establishes the above in precise detail, is \autoref{thm: reconstruction}, in which we show how a Morita 2-equivalence $\mathcal{C}$ of holomorphic symplectic double groupoids, with real structure $\tau$ and equipped with a Lagrangian bisection $\mathcal{L}$ of its symplectic core, determines a generalized K\"ahler structure.   We also show a partial converse to this result, \autoref{thm: the (1,1) morphism of a GK}, in which we promote a generalized K\"ahler structure to a Morita 2-equivalence with the required structure, provided that its underlying Manin triple consists of integrable Lie algebroids.

\subsection{Consequences of the main Theorem}
One of the key consequences of the double groupoid approach described above is that it allows us, in Section~\ref{sec:gk potential}, to resolve affirmatively a long-standing conjecture by physicists that generalized K\"ahler metrics are locally determined by a single real-valued scalar function of $\dim M$ variables, just as in the K\"ahler case.  Since we show that the metric is determined by a real Lagrangian submanifold $\mathcal{L}$ of a symplectic manifold $\mathcal{S}$ of dimension $2\dim M$, one may describe $\mathcal{L}$ locally using a generating function of $\dim M$ variables, as hoped.    

In the special case of a usual K\"ahler structure (i.e. $I_+=I_-=I$ is the complex structure and $\omega_+=\omega_- = \omega$ is the K\"ahler form), the situation simplifies dramatically: the double groupoids $D_\pm$ coincide with the pair groupoid of the holomorphic cotangent bundle of the underlying complex manifold, and the remaining morphisms $\mathcal{Z},\mathcal{W},\mathcal{C}$ are constructed from the $\omega$-twisted cotangent bundle $T^*_\omega X$, whose isomorphism class as a symplectic affine bundle captures the K\"ahler class.  The K\"ahler metrics within this class are then seen as smooth sections $\mathcal{L}$ of $T^*_\omega X$ which are Lagrangian and symplectic, respectively, for the real and imaginary parts of the holomorphic symplectic form.  This observation, that K\"ahler metrics are Lagrangian submanifolds, was first made by Donaldson~\cite{MR1959581};  we show how our picture specializes to his in Section~\ref{subsec:examples}.  

Another special case which informed our work is that of generalized K\"ahler structures of \emph{symplectic type}, in which one of the two real Poisson structures $\pi_A = \tfrac{1}{2}(I_+- I_-)g^{-1}, \pi_B=\tfrac{1}{2}(I_++I_-)g^{-1}$ discovered by Lyakhovich and Zabzine~\cite{MR1948542} is symplectic. In this case, the use of double groupoids is unnecessary, and one obtains a simplified picture of generalized K\"ahler metrics as real Lagrangian bisections of a Morita equivalence between usual symplectic groupoids~\cite{MR4466669}.  The main difference between this and the K\"ahler case is that the symplectic groupoids may be nonabelian, in contrast with the addition law on the cotangent bundle in Donaldson's picture.  This has recently been used~\cite{MR4397176} to construct noncommutative graded algebras out of generalized K\"ahler structures of symplectic type, in contrast to the usual commutative coordinate rings in K\"ahler geometry.  The algebraic ramifications of the more general case treated in this paper remain to be seen. 

\subsection{Generalized K\"ahler structures on Lie groups}
Some of the most interesting known examples of generalized K\"ahler metrics which are not deformations of K\"ahler structures can be found on even-dimensional compact semisimple Lie groups, where the Riemannian metric is bi-invariant, but the complex structures $(I_+,I_-)$ are right and left-invariant respectively; the full generalized K\"ahler structure is only invariant under the action of the maximal torus.  The Hermitian forms are not closed, and $d^c_+\omega_+ = -d^c_-\omega_- = H$ coincides with the Cartan 3-form, whose nontrivial cohomology class contradicts the $\partial\bar\partial$-Lemma as mentioned above.  These generalized K\"ahler structures, relevant to the Wess--Zumino--Witten 2-dimensional sigma model,  are not of symplectic type and therefore cannot be treated by the  special cases mentioned above.  In Theorem~\ref{pro:doumorequ}, we explicitly construct the Morita 2-equivalences integrating this family of examples, expressing the symplectic manifolds $(D_\pm,\mathcal{Z},\mathcal{W},\mathcal{C})$ as moduli spaces of flat connections on  2-dimensional surfaces with Lagrangian boundary conditions.  Our method is inspired by Boalch's construction~\cite{MR2352135} of the Lu-Weinstein symplectic double groupoid associated to a Lie bialgebra, and uses \v{S}evera's recent generalization~\cite{MR3330247,MR3001637} of the work of Alekseev, Malkin and Meinrenken~\cite{MR1638045}, which itself is based upon the result of Atiyah and Bott on the symplectic structure on the moduli space of flat connections on a closed oriented 2-manifold~\cite{atibot}.  The existence of groupoid and double groupoid structures on these moduli spaces is interpreted in terms of cutting and gluing of surfaces.  We then use this presentation in \autoref{ex:hopf surf} to obtain explicit formulae for the generalized K\"ahler potential function on the Hopf surface.

\subsection{Outline of the paper} 
In Section \ref{section: The holomorphic Manin triples of GK structures}, we review the required tools from Dirac geometry and generalized complex geometry, and recast generalized Kähler structures purely in terms of pairs of transverse holomorphic Dirac structures known as Manin triples.
This new characterization unravels an intimate connection between generalized K\"ahler geometry and double groupoids at the infinitesimal level, setting the stage for the integration problem.

In Section \ref{section: Symplectic double groupoids}, we review the way in which symplectic double groupoids are related to Manin triples, and refine the existing theory of double Lie groupoids to include both vertical and horizontal generalized 1-morphisms as well as a new notion of Morita 2-equivalence between such morphisms, which have the shape of squares as in Fig.~\ref{mor2}. 

In Section \ref{section: integration}, we explain the passage from a generalized K\"ahler structure to its integrating object.  \autoref{thm: the (1,1) morphism of a GK} shows that the holomorphic symplectic double groupoids $D_\pm$ and their conjugates are related by a holomorphic symplectic Morita 2-equivalence equipped with a real structure. Furthermore, \autoref{thm: key Lagrangian in the central brane} shows that a single Lagrangian bisection as in Fig.~\ref{sympme} determines a trivialization of the underlying real Morita 2-equivalence in Fig.~\ref{mor2}. 

In Section \ref{sec:gk potential}, we prove the converse, \autoref{thm: reconstruction}, showing that the Morita 2-equivalence and Lagrangian bisection suffice to recover the generalized K\"ahler structure. As an application, we show how the group of Hamiltonian flows acts on the space of Lagrangian bisections to obtain a method of deforming generalized K\"ahler metrics. Our deformation method recovers the existing constructions in special cases \cite{MR1702248, MR2287917, MR2681704, MR2217300}. 

In Section \ref{sec: lie groups}, we provide an application of our main results to the known examples of generalized K\"ahler structures on compact even-dimensional semisimple Lie groups.  We construct the integrating Morita 2-equivalence explicitly as a moduli space of flat connections on an annulus with Lagrangian boundary conditions, and exhibit a canonical Lagrangian bisection of its symplectic core.  

\vspace{1ex}
\noindent \textit{\textbf{Acknowledgements}}

We thank Henrique Bursztyn, Ezra Getzler, Nigel Hitchin, Martin Ro\v{c}ek, Pavol \v{S}evera, Maxim Zabzine,  and Alan Weinstein for teaching us many of the tools needed in this work.  This project was supported by NSERC Discovery grant ``Higher structures in generalized geometry and mathematical physics'' RGPIN-2018-04349. 

\section{Generalized Kähler geometry and holomorphic Manin triples}\label{section: The holomorphic Manin triples of GK structures}


Our starting point is the formulation of generalized K\"ahler geometry in terms of generalized complex geometry~\cite{marcophd,MR2811595, MR2013140}.  To fix notation, let $M$ be a smooth manifold and $H$ a closed 3-form on it.  Recall that a generalized complex structure $\J$ on $(M,H)$ is a complex structure on the vector bundle $\DT M= TM\oplus T^*M$ whose $+i$ eigenbundle is maximal isotropic for the natural split-signature bilinear form 
\[
\IP{X+\a,Y+\b}=\iota_X\b + \iota_Y\a,
\] 
and involutive with respect to the $H$--twisted Courant bracket  
\begin{equation}\label{stdcrt}
[X+\a, Y+\b]_H = [X,Y] + \L_X \b - \iota_Y d\a +\iota_X\iota_YH.
\end{equation}
A compatible pair of such structures then defines a generalized K\"ahler structure, as follows.  

\begin{definition}
	A {\em generalized Kähler structure} on $(M,H)$ is a pair of commuting generalized complex structures $\J_A, \J_B$ such that the bilinear form 
 \[ \IP{ \J_A \cdot,\J_B \cdot} \] 
defines a positive definite bundle metric on $\DT M$ known as a \emph{generalized Riemannian metric}. 
\end{definition}

A key result of \cite{MR3232003} is that the bi-Hermitian geometry discovered by Gates, Hull and Ro\v{c}ek, namely a pair $(I_+,I_-)$ of complex structures on $M$, each compatible with the Riemannian metric $g$ and whose Hermitian forms $\omega_\pm  = gI_\pm$ satisfy 
\begin{equation}\label{dcpm}
d^c_+\omega_+ = -d^c_-\omega_- = H,    
\end{equation}
is equivalent to a generalized K\"ahler structure in the sense above via the following formula. 
\begin{align}\label{Gualtieri's formula}
\J_A = \frac{1}{2}  \begin{pmatrix}
	I_{+}+ I_{-} & (I_+- I_-)g^{-1} \\
	g(I_+ - I_-)& -(I_+^*+ I_-^*) \end{pmatrix}
 \qquad
\J_B= \frac{1}{2}  \begin{pmatrix}
	I_{+}- I_{-} & (I_+ + I_-)g^{-1} \\
	g(I_+ + I_-)& -(I_+^*- I_-^*)\end{pmatrix}
\end{align}
In the remainder of this section, we use the above to provide a further reformulation of generalized K\"ahler geometry in terms of Dirac structures, which makes apparent its infinitesimal Lie-theoretic character.

\subsection{Holomorphic Courant algebroids and matched pairs}\label{subsec: hol courant} 

In this and the next subsection, we show that each of the complex manifolds $X_\pm = (M,I_\pm)$ involved in a generalized K\"ahler structure is endowed with a holomorphic exact Courant algebroid $\E_\pm$ which splits as a sum of holomorphic Dirac structures: $\E_\pm = \A_\pm \oplus\B_\pm$.  Such a structure is known as a (holomorphic) Manin triple~\cite{MR1472888}. 

We begin by reviewing a result from~\cite{MR3232003}: holomorphic exact Courant algebroids over the complex manifold $X$ are classified by the first cohomology with coefficients in the sheaf of holomorphic closed 2-forms
\[
H^1(X,\Omega^{2,\cl}_X) = \mathbb{H}^1(X,\Omega^2_X\to \Omega^3_X\to \Omega^4_X),
\]
in the following way.  Given a closed 3-form in the Dolbeault resolution of the above complex,
\[
\H = \H^{2,1}+\H^{3,0}\in \Omega^{2,1}(X)\oplus\Omega^{3,0}(X),
\]
we define a holomorphic structure on $\DT_{1,0} X = T_{1,0}X\oplus T^*_{1,0}X$ via the Dolbeault operator
\begin{equation}\label{dolb1}
\overline{D} = \begin{pmatrix}
    \delbar & 0 \\
    \H^{2,1} & \delbar
\end{pmatrix},
\end{equation}
and define a bracket on the smooth sections of $\DT_{1,0}X$ via
    \begin{equation}\label{crntmp}
[X+\a,Y+\b]_{1,0} = [X,Y] + (\del \iota_X + \iota_X \del)\b - \iota_Y\del\a + \iota_X\iota_Y\H^{3,0},
\end{equation}
which descends to a Courant bracket on the $\overline{D}$-holomorphic sections.  Two such Courant algebroids defined by $\H$ and $\H'$ over the same complex manifold are isomorphic if their Dolbeault operators and Courant brackets are intertwined by the gauge transformation of $\DT_{1,0}X$ given by
\begin{equation}\label{holgauge}
e^\B=\begin{pmatrix}
    1 & 0\\
    \B & 1
\end{pmatrix}, \qquad \B\in\Omega^{2,0}(X),
\end{equation}
which occurs if and only if $\H'-\H = d\B$. 

There is a convenient way of repackaging the above structure into a single smooth complex Courant algebroid known as a \emph{matched pair}~\cite{MR3264784}: on $\DT_\C X = \DT_{1,0}X \oplus \DT_{0,1}X$, where $\DT_{0,1}X = T_{0,1}X\oplus T^*_{0,1}X$, we consider the Courant algebroid structure with anchor given by the canonical projection to $T_\C X$, the canonical symmetric pairing, and the standard Courant bracket~\eqref{stdcrt} twisted by $\H=\H^{2,1}+\H^{3,0}$.
The Dolbeault operator~\eqref{dolb1} and bracket~\eqref{crntmp} can then be recovered by projection of the bracket onto $\DT_{1,0}X$:
\begin{equation}\label{projbrak}
\overline{D}_X(Y+\beta) = \pi_{1,0}([X, Y+\beta]),\qquad [X+\a,Y+\b]_{1,0}=\pi_{1,0}([X+\a,Y + \b]).
\end{equation}
The notion of a matched pair extends to holomorphic Dirac structures: if $\A\subset\DT_{1,0}X$ is a holomorphic Dirac structure (namely, a maximal isotropic holomorphic subbundle closed under the bracket), then its matched pair is the smooth Dirac structure $\A\oplus T_{0,1}X\subset \DT_\C X$.  This coincides with the notion of matched pair of Lie algebroids~\cite{MR1460632}.
\begin{figure}[H]
\[
\begin{tikzcd}        
    (\DT_{1,0}X, \overline{D}, [\cdot,\cdot]_{1,0}) \arrow[|->]{r} & (\DT_{1,0}X\oplus \DT_{0,1}X, [\cdot,\cdot]_\H)\\
\A \arrow[hook]{u}& \A \oplus T_{0,1} X\arrow[hook]{u}
\end{tikzcd}
\]
\caption{Matched pair functor from holomorphic to smooth Courant algebroids}
\end{figure}
Importantly for what follows, the above embedding of holomorphic Courant algebroids into smooth Courant algebroids (the matched pair functor) is not full: two non-isomorphic holomorphic Courant algebroids may have gauge-equivalent matched pairs.

\begin{example}\label{pluri}
    Let $X$ be a complex manifold and $\omega\in\Omega^{1,1}_\R$ a pluriclosed real $(1,1)$--form (i.e. $dd^c\omega=0$).  The real closed 3-form 
    \[
    H = d^c\omega = i(\delbar-\del)\omega
    \]
    is the real part of the $(2,1)$-form $\H = -2i\del\omega$, which is itself closed, defining a holomorphic Courant algebroid $\E = (\DT_{1,0}X, \overline{D}, [\cdot,\cdot]_{1,0})$ via~\eqref{dolb1} and~\eqref{crntmp}. 


    In the same way,  the closed form $\overline\H = 2i\delbar\omega$ defines a holomorphic Courant algebroid $\overline{\E}$ on the conjugate complex manifold $\overline{X}$.  Observe that while $\E$ and $\overline{\E}$ may not be isomorphic (indeed, $X$ and $\overline{X}$ need not be isomorphic as complex manifolds), their matched pairs are gauge equivalent: the identity
    \[
    \H - \overline\H= -2id\omega 
    \]
    implies that the matched pairs of $\E$ and $\overline{\E}$ are equivalent as smooth Courant algebroids, via the imaginary gauge transformation defined by $2i\omega$.  
    
    Since the matched pair of $\E$ is a smooth complex Courant algebroid, we may extract its real and imaginary parts $\E_R = (\DT M, H)$ and $\E_I = (\DT M, -d\omega)$ respectively, each of which is a smooth real Courant algebroid.  In particular, the pluriclosed form $\omega$ may be interpreted as a trivialization of the imaginary part of $\E$.
    \end{example}

Example~\ref{pluri} has the following converse, which describes the close relationship between pluriclosed real $(1,1)$-forms and holomorphic Courant algebroids. 

	\begin{proposition}\label{Real structures of CAlg}
        Let $\E$ be a holomorphic exact Courant algebroid. An involutive isotropic splitting of its imaginary part $\E_I$ induces an isomorphism between its real part $\E_R$ and $(\DT M, H)$, where $H=d^c\omega$, for a unique pluriclosed real $(1,1)$-form $\omega$.  
        \end{proposition}
	\begin{proof}
		Let $\H = \H^{3,0} + \H^{2,1}$ be a closed form representing $\E$, and let $B$ be the real 2-form representing the involutive isotropic splitting of $\E_I$, so that  $dB = \Im (\H)$, or in components,
		\begin{equation}\label{imcompo}
		\H^{(3,0)} = 2i \d B^{(2,0)}, \qquad \H^{(2,1)} = 2i(\dbar B^{(2,0)} + \d B^{(1,1)}). 
		\end{equation} 
        Note that a gauge transformation of the form~\eqref{holgauge} takes $(\H,B)$ to $(\H+ d\B, B + \Im(\B))$,  for $\B$ of type $(2,0)$, so that the $(1,1)$ component 
        $\omega=-B^{1,1}$ is gauge-invariant.  
        As a result of~\eqref{imcompo}, 
		\[
		\Re(\H) =d(-2 \Im (B^{(2,0)})) + d^c \omega,
		\]
		showing that $2\Im(B^{(2,0)}) = - \tfrac{1}{2}(BI+I^*B)$ defines an isomorphism between $\E_R$ and $(\DT M, d^c\omega)$, as required. 
	\end{proof}

On a generalized K\"ahler manifold, we may apply Example~\ref{pluri} to each of the complex manifolds $X_\pm$:  from Equation~\ref{dcpm}, we see that $\H_+ = -2i\del_+\omega_+$ and $\H_- = 2i\del_-\omega_-$ define holomorphic Courant algebroids $\E_\pm$ on $X_\pm$, each equipped with a trivialization of its imaginary part, together with an isomorphism between the real parts of $\E_+$ and $\E_-$.  This implies that the matched pairs of $\E_\pm$ and their complex conjugates are related by gauge transformations according to the following diagram of closed 3-forms:
\begin{figure}[H]
    \centering
    \begin{tikzcd}
\overline{\mathcal{H}}_- &                     & \mathcal{H}_+ \\
   & H \arrow[lu, "-i\omega_-" description] \arrow[ru, "-i\omega_+" description] \arrow[ld, "i\omega_+" description] \arrow[rd, "i\omega_-" description] &               \\
\overline{\mathcal{H}}_+ &              & \mathcal{H}_-
\end{tikzcd}
   \hspace{5em}
   \begin{tikzcd}
\overline{\E}_- &                     & \E_+ \ar[ll,"e^{i(\omega_+-\omega_-)}"', "\cong"]\\
   & &               \\
\overline{\E}_+\ar[uu, "\cong"', "e^{-i(\omega_++\omega_-)}"] &              & \E_-\ar[ll,"e^{i(\omega_+ - \omega_-)}", "\cong"']\ar[uu, "e^{-i(\omega_++\omega_-)}"', "\cong"]
\end{tikzcd}
    \caption{Gauge equivalences between matched pairs of $\E_\pm$}
    \label{gaucyc}
\end{figure}

\begin{proposition}\label{gaugcyc}
    Let $\E_\pm$ be the holomorphic Courant algebroids determined by $\H_\pm =\mp 2i\del_\pm\omega_\pm$ on a generalized K\"ahler manifold.  The pluriclosed forms $\mp \omega_\pm$ define trivializations of the imaginary parts of $\E_\pm$ and identify the real parts of $\E_\pm$ with $(\DT M,H)$, where $H=\pm d^c_\pm \omega_\pm$.   
    
    As a result, the matched pairs of $\E_\pm$ and their complex conjugates are equivalent as smooth Courant algebroids by the pure imaginary gauge transformations in Fig.~\ref{gaucyc}.  
\end{proposition}

\subsection{Manin triples from generalized K\"ahler structures}

A Manin triple $(\E, \A, \B)$ consists of a Courant algebroid $\E$ and Dirac structures $\A,\B$ in $\E$ such that $\E = \A \oplus \B$. In this section, we show that each of the holomorphic Courant algebroids $\E_\pm$ determined by a generalized K\"ahler structure is equipped with a pair of transverse holomorphic Dirac structures, forming a holomorphic Manin triple $(\E_\pm,\A_\pm,\B_\pm)$ over the complex manifold $X_\pm$. Furthermore, we find that these Manin triples are related by gauge transformations determined by the generalized K\"ahler structure.  We conclude with a reformulation of generalized K\"ahler geometry in terms of gauge equivalences between holomorphic Manin triples. 

Since a generalized K\"ahler structure consists of the pair of commuting complex structures~\eqref{Gualtieri's formula}, we obtain a decomposition of $\DT_\C M$ into four simultaneous eigenspaces of $(\J_A,\J_B)$:
\begin{equation}
\DT_\C M = \l_+\oplus \l_- \oplus\overline{\l_+}\oplus\overline{\l_-},
\end{equation}
where $L_A = \l_+\oplus \l_-$ and $L_B=\l_+\oplus\overline{\l_-}$ are the  $+i$ eigenbundles of $\J_A$ and $\J_B$, respectively. 

\begin{lemma}\label{lilell}
    The simultaneous eigenbundles of $(\J_A,\J_B)$ may be expressed in terms of the bihermitian data as follows:
    \[
    \l_\pm = e^{\mp i\omega_\pm}(T_{1,0}X_\pm) = \{ Y \mp i\omega_\pm(Y)\ :\ Y\in T_{1,0} X_\pm\}.
    \]
\end{lemma}
\begin{proof}
    From \eqref{Gualtieri's formula}, we see that $Y+\b\in \l_+$ if and only if
    \[
    \tfrac{1}{2}((I_+\pm I_-)Y + (I_+\mp I_-)g^{-1}\b) = i(Y+\b).
    \]
    The sum and difference of these equations yields the result.  Proceed similarly for $\l_-$. 
\end{proof}

As a result, the gauge transform $e^{-i\omega_+}$ which identifies $(\DT_\C M, H)$ with the matched pair of $\E_+$ takes $\overline{\l_+}$ to $T_{0,1}X_+$.  We now show that the Dirac structures $\overline{L}_A$ and $\overline{L}_B$, each of which contain $\overline{\l_+}$ as a summand, are sent to transverse holomorphic Dirac structures $\A_+, \B_+$ in $\E_+$, hence defining a holomorphic Manin triple $(\E_+,\A_+,\B_+)$.  Similarly, the gauge transform $e^{i\omega_-}$ identifying $(\DT_\C M, H)$ with the matched pair of $\E_-$ takes $\overline{\l_-}$ to $T_{0,1}X_-$, and the Dirac structures $\overline{L}_A$ and ${L}_B$, each of which contain $\overline{\l_-}$ as a summand, are sent to transverse holomorphic Dirac structures $\A_-, \B_-$ in $\E_-$, hence defining a holomorphic Manin triple $(\E_-,\A_-,\B_-)$. In~\cite{MR3232003} this is shown using Courant reduction, but we need an alternative, more direct proof in order to obtain our Theorem~\ref{Inf descprtion of GK}.

\begin{theorem}\label{manintriples}
    On a generalized K\"ahler manifold, the Dirac structures $L_A, L_B$ defined by the $+i$ eigenbundles of $\J_A,\J_B$ determine holomorphic Manin triples $\E_\pm=\A_\pm\oplus\B_\pm$ on $X_\pm$ as follows:
\begin{align*}
    \A_+ &=\DT_{1,0} X_+\cap ( e^{-i\omega_+}(\overline{L}_A)) & \A_-=&\DT_{1,0} X_-\cap (e^{i\omega_-}(\overline{L}_A)) \\
    \B_+ &=\DT_{1,0} X_+\cap ( e^{-i\omega_+}(\overline{L}_B)) & \B_-=&\DT_{1,0} X_-\cap (e^{i\omega_-}(L_B)) 
\end{align*}
\end{theorem}
\begin{proof}
Since $L_A, L_B$ are involutive for the Courant bracket defined by $H$, and since the above gauge transformations coincide with those from Proposition~\ref{gaugcyc},  it follows that $\A_\pm, \B_\pm$ are isotropic, involutive and holomorphic for the bracket and Dolbeault operators~\eqref{projbrak}.  What remains to show is that $\A_\pm, \B_\pm$ are maximal isotropic and transverse. 

From Lemma~\ref{lilell} and the decomposition $\overline{L}_A = \overline{\l_+}\oplus\overline{\l_-}$, we see that 
\begin{equation}\label{gaugla}
e^{-i\omega_+}(\overline{L}_A) = T_{0,1}^+\oplus e^{-i(\omega_-+\omega_+)}(T_{0,1}^-).
\end{equation}
We then use the following identity, for $X\in T_{0,1}^-$: 
\begin{equation}\label{omegasproj}
i(\omega_- + \omega_+)X = ig(I_- + I_+)X = -g(1+iI_+)X
=-2g P_{0,1}^+X,
\end{equation}
where $P_{0,1}^+=\tfrac{1}{2}(1 + iI_+)$ is the projection to the $(0,1)$ component for the complex structure $I_+$, to conclude that $\omega_-+\omega_+$ maps $T_{0,1}^-$ into $(T_{1,0}^+)^*$.  Combining the above, we obtain the following explicit description of $\A_+$:
\begin{equation}\label{aplus}
\A_+=\DT_{1,0} X_+\cap ( e^{-i\omega_+}(\overline{L}_A)) =\{P_{1,0}^+X - 2g P_{0,1}^+ X\ :\ X\in T_{0,1}^-\}.
\end{equation}
Here $P_{1,0}^+=\tfrac{1}{2}(1 - iI_+)$ is the projection to the $(1,0)$ component in the complex structure $I_+$.  Since $g$ is nondegenerate, we see that $P_{1,0}^+ - 2g P_{0,1}^+$ is an isomorphism from $T_\C M$ to $\DT_{1,0} X_+$, implying that $\A_+$ has the same dimension as $T_{0,1}^-$ and is hence maximal isotropic, as required. In the same way, we have an explicit description of $\B_+$: 
\begin{equation}\label{bplus}
\B_+ = \{P^+_{1,0}X - 2g P^+_{0,1} X\ : \ X\in T^{-}_{1,0}\},
\end{equation}
so that the transversality $T_\C M = T_{0,1}^-\oplus T_{1,0}^-$ implies that $\DT_{1,0}X_+ = \A_+\oplus \B_+$, as required.  Proceed in the same way to obtain the holomorphic Manin triple on $X_-$.
\end{proof}

A holomorphic Manin triple $(\E, \A, \B)$ induces a holomorphic Poisson structure $\sigma$ on its base manifold $X$, by taking the following difference of Dirac structures:
\begin{equation}\label{poismt}
    \Gr(\sigma) = \A - \B = \{X + \alpha - \beta\ :\ X+\alpha\in \A,\ X+\beta\in\B\}
\end{equation}

We may read the above as a Baer sum of $\A$, a Dirac structure in the Courant algebroid $\E$ determined by $\H$, with $-\B$, a Dirac structure in the transpose Courant algebroid $\E^\top$ determined by $-\H$.  The resulting Dirac structure is the graph of a Poisson bivector in the trivial Courant algebroid.  

In this way, we conclude from Theorem~\ref{manintriples} that the complex manifolds $X_\pm$ inherit holomorphic Poisson structures $\sigma_\pm$.  As shown in~\cite{MR3232003}, these coincide with the holomorphic Poisson structures discovered by Hitchin~\cite{MR2217300}. 

Having shown that a generalized K\"ahler manifold involves two complex structures $X_\pm$, each endowed with a holomorphic Manin triple $(\E_\pm,\A_\pm,\B_\pm)$, the main remaining question is how these two holomorphic structures are related.  There is no direct holomorphic morphism between them, but as we shall see, their matched pairs are gauge equivalent in an interesting way.    

\subsection{Gauge equivalence of Manin triples}

In Theorem~\ref{manintriples}, we constructed holomorphic Manin triples $(\E_\pm,\A_\pm,\B_\pm)$ in which $\A_\pm$ were both obtained from the complex Dirac structure $\overline{L}_A$, whereas $\B_+, \B_-$ derived from complex conjugate Dirac structures $\overline{{L}_B}$ and $L_B$, respectively.  We now show how this implies that the matched pairs of $\A_\pm$, $\B_\pm$, and their conjugates are related by the cycle of gauge equivalences from Fig.~\ref{gaucyc}.
\begin{theorem}
\label{Inf descprtion of GK} 
	The holomorphic Manin triples $(\E_\pm, \A_{\pm}, \B_{\pm})$ underlying a generalized K\"ahler structure are related by the following gauge equivalences of corresponding matched pairs:
    
\begin{equation}\label{manindiag}
\begin{gathered}
 \begin{tikzcd}
(\overline{\E_-}, \overline{\A_{-}}, \overline{\B_{-}}) &                     & ({\E_+}, {\A_{+}}, {\B_{+}}) \ar[ll,"e^{i(\omega_+ - \omega_-)}"']\\
   & &               \\
(\overline{\E_+}, \overline{\A_{+}}, \overline{\B_{+}})\ar[uu,  "e^{-i(\omega_++\omega_-)}"] &              & (\E_-,\A_-,\B_-)\ar[ll,"e^{i(\omega_+ - \omega_-)}"]\ar[uu, "e^{-i(\omega_++\omega_-)}"']
\end{tikzcd}
\end{gathered}
\end{equation}
where the vertical maps only intertwine Dirac structures of type $\A$, and the horizontal maps only intertwine Dirac structures of type $\B$.  
\end{theorem} 
\begin{proof}
To avoid ambiguity, we denote the matched pairs of $\A_\pm, \B_\pm$ respectively by $L_{\A_\pm}, L_{\B_\pm}$, so that the content of the Theorem is the pair of identities (together with their complex conjugates)
\begin{equation}\label{pmrel}
L_{\A_+} = e^{-i(\omega_++\omega_-)} L_{\A_{-}}, \quad \overline{L_{\B_{-}}} =   e^{i(\omega_+-\omega_{-})} L_{\B_+}.
\end{equation}

These identities follow from the relationship between the holomorphic Manin triples and the generalized complex eigenbundles described in
Theorem~\ref{manintriples}, from which we derive the identities
\begin{align}\label{mpgk}
\begin{split}
        L_{\A_{+}} &= e^{-i\omega_{+}} \overline{L_A}, \qquad L_{\A_{-}} = e^{i\omega_{-}} \overline{L_A},\\
	L_{\B_{+}} &= e^{-i\omega_{+}} \overline{L_B}, \qquad L_{\B_{-}} = e^{i\omega_{-}} L_B.  
 \end{split}
\end{align}
 We only prove the first gauge equivalence as the others are obtained similarly. From the definition of the matched pair of $\A_+$, and using  \eqref{aplus}, we obtain
 \begin{align*}
L_{\A_+}=T_{0,1}^+\oplus \A_+  &= T_{0,1}^+\oplus \{P_{1,0}^+X - 2g P_{0,1}^+ X\ :\ X\in T_{0,1}^-\} \\
&=T_{0,1}^+\oplus\{ P_{0,1}^+X+P_{1,0}^+X - 2g P_{0,1}^+ X\ :\ X\in T_{0,1}^-\}\\
&=T_{0,1}^+\oplus e^{-(\omega_-+\omega_+)}(T_{0,1}^-),
 \end{align*}
 which coincides with  $e^{-i\omega_{+}} \overline{L_A}$ by~\eqref{gaugla},
 as claimed by~\eqref{mpgk}. Combining the identities~\eqref{mpgk}, we immediately obtain~\eqref{pmrel}, as required. 
 \end{proof}

The gauge equivalences described in Theorem~\ref{Inf descprtion of GK} between the Manin triples $(\E_\pm, \A_\pm, \B_\pm)$ and their complex conjugates allow us to describe in detail the real and imaginary parts of the holomorphic Dirac structures involved.

\begin{definition}
    Let $L$ be a smooth complex Dirac structure in $\DT_\C M$, involutive for the Courant bracket twisted by $H\in \Omega^{3,\cl}(M,\C)$ and such that its tangent projection $\rho(L)$ is transverse to its conjugate, i.e. $\rho(L)+\rho(\overline{L}) = T_\C M$.  The real part of $L$ is defined to be the real Dirac structure 
    \[
    L_R= \tfrac{1}{2}(L+\overline{L}) \cap \DT M = \{ X+\tfrac{1}{2}(\alpha + \beta)\in \DT M\ :\ X+\alpha\in L,\  X+\beta\in \overline{L}\},
    \]
 involutive with respect to $\Re(H)$, while the imaginary part of $L$ is 
    \[
     L_I=\tfrac{1}{2i}(L-\overline{L})  \cap \DT M =\{ X+\tfrac{1}{2i}(\alpha - \beta)\in\DT M\ :\ X+\alpha\in L,\  X+\beta\in \overline{L}\},
    \]
involutive with respect to $\Im(H)$.  For $\A$ a holomorphic Dirac structure with matched pair $L_{\A}$, its real and imaginary parts are the real Dirac structures
    \[
    \A_{R}=\tfrac{1}{2}(L_{\A}+\overline{L_{\A}})  \cap \DT M \subset \E_R, \qquad \A_I=\tfrac{1}{2i}(L_{\A}-\overline{L_{\A}})  \cap \DT M \subset \E_I.
    \] 
\end{definition} 
 
Despite the inclusion of $\A_R$ and $\A_I$ in possibly non-isomorphic Courant algebroids, they are isomorphic as Lie algebroids~\cite{MR4480214}. From the following equivalent description
\[
\A_{R} = \{2\Re (u)+\Re(\a): u + \a \in \A \}, \quad \A_I = \{2\Re(u) + \Im(\a): u+ \a \in \A\},
\]
the Lie algebroid isomorphism is $2\Re(u)+\Re(\a) \mapsto 2\Re(u)+\Im(\a)$.

\begin{example}\label{ex:real hol poisson}
    Let $\A = \Gr(\sigma)$, where $\sigma = \sigma_R + i \sigma_I$ is a holomorphic Poisson structure. Then $\A_R = \Gr({4\sigma_R})$ and $\A_I = \Gr({4\sigma_I})$.   The equivalence of these as Lie algebroids accounts for the fact that the real and imaginary parts of a holomorphic Poisson structure share the same foliation by symplectic leaves. 
\end{example}

\begin{example}\label{rimgc}
    Let $L$ be the $+i$-eigenbundle of a generalized complex structure $\J$. Then we have the following identities for the real and imaginary parts of $L$: 
    \[
    L_R = \tfrac{1}{2}(L+\overline{L})  \cap \DT M = \J(T^*M),\qquad L_I =\tfrac{1}{2i}(L-\overline{L})  \cap \DT M = \Gr({\pi_{\J}}),
    \]
    where $\pi_\J(\alpha) = \rho(\J\alpha)$ (for $\rho$ the projection to $TM$) is the real Poisson structure underlying $\J$.  
\end{example}
Applying Example~\ref{rimgc} to a generalized K\"ahler manifold, we obtain four real Dirac structures: two integrable with respect to $H$, namely  $\J_A(T^*M)$ and $\J_B(T^*M)$, and two Poisson structures $\pi_A, \pi_{B}$.  Equation~\eqref{Gualtieri's formula} may be used to write these Dirac structures in terms of the bi-Hermitian data; in particular, the real Poisson tensors are:
\begin{equation}\label{poisab}
		\pi_A = \frac{1}{2}(I_{+}-I_{-})g^{-1}, \quad \pi_B = \frac{1}{2}(I_{+}+I_{-})g^{-1}.
\end{equation}
We now explain how these relate to the real Dirac structures determined by the pair of holomorphic Manin triples, using Theorem~\ref{Inf descprtion of GK}. 

\begin{corollary}\label{cor: the underlying re and im}
Let $(\E_\pm, \A_{\pm}, \B_{\pm})$ be the holomorphic Manin triples underlying a generalized K\"ahler manifold.  Then, their real parts are canonically isomorphic to each other, as follows:
\[
\Re(\E_+, \A_+, \B_+) \cong((\DT M, H), \J_A(T^*M), \J_B(T^*M)) \cong \Re(\E_-, \A_-, \B_-),
\]
while their imaginary parts are canonically isomorphic to the following real Manin triples: 
\[
\Im(\E_\pm, \A_\pm, \B_\pm) \cong ((\DT M, 0), \Gr({-\pi_A}), \Gr({\mp \pi_B})). 
\]
\end{corollary} 

\begin{proof}
    By \autoref{rimgc} and \autoref{mpgk}
    \[
	(\A_{+})_{I} = \frac{1}{2i} (L_{\A_{+}} - \overline{L_{\A_{+}}})  \cap \DT M =  \frac{e^{-i\omega_{+} } \bar{L}_1 - e^{i\omega_{+}}L_1} {2i}   \cap \DT M  = e^{-\omega_{+}} \frac{\bar{L}_1 - L_1}{2i}  \cap \DT M = e^{-\omega_{+}} \Gr(-\pi_A). 
	\]
    A similar computation shows $(\B_+)_I = e^{-\omega_+} \Gr(-\pi_B)$. Thus we have the isomorphism of real Manin triples
    \[
   \begin{tikzcd}
       \Im(\E_+,\A_+,\B_+)\arrow[r,"{e^{\omega_+}}","\cong"'] &  ((\DT M, 0), \Gr({-\pi_A}), \Gr({-\pi_B})). 
   \end{tikzcd} 
    \]
    By the same method, we obtain the isomorphism
    \[
\begin{tikzcd}
   \Im(\E_-,\A_-,\B_-)\arrow[r,"{e^{-\omega_-}}","\cong"']& ((\DT M, 0), \Gr({-\pi_A}), \Gr({\pi_B})). 
    \end{tikzcd}\]
    On the other hand, by \autoref{rimgc} and \autoref{mpgk},
    \[
    (\A_+)_R = \frac{1}{2}(L_{\A_+}+\overline{L_{\A_+}})  \cap \DT M = \frac{1}{2}(L_{A}+\overline{L_{A}})  \cap \DT M = \J_A(T^*M), 
    \]
    and similarly,
    \[
    (\B_+)_R = \frac{1}{2}(L_{B}+\overline{L_{B}}) \cap \DT M= \J_B(T^*M).
    \]
    The same argument for $\A_-,\B_-$ provides the required isomorphism for the real parts. 
\end{proof}

\begin{corollary}\label{cor: Holomorphic Hitchin and real Poisson}
	The real parts of the Hitchin Poisson structures $\sigma_{\pm}$ coincide with the underlying Poisson structure of the Manin triple  $((\DT M, H), \J_A(T^*M), \J_B(T^*M)) $ while the imaginary parts coincide with the underlying Poisson structures of $((\DT M,0),\Gr(-\pi_A), \Gr(\mp\pi_B))$.
\end{corollary}

\begin{proof}
Since $\Gr(\sigma_\pm) = \A_\pm - \B_\pm$ (Equation~\eqref{poismt}), the same holds for matched pairs: $L_{\sigma_\pm} = L_{\A_\pm} - L_{\B_\pm}$.  Taking real parts, and using Corollary~\ref{cor: the underlying re and im}, we obtain 
\begin{align*}
\tfrac{1}{2}(L_{\sigma_\pm} +\overline{L_{\sigma_\pm}})\cap \DT M  &=  \tfrac{1}{2}( L_{\A_\pm} - L_{\B_\pm} +  \overline{L_{\A_\pm}} - \overline{L_{\B_\pm}})\cap \DT M\\
&=\J_A(T^*M) - \J_B(T^*M),
\end{align*}
as required.  For the imaginary parts, we have 
\begin{align*}
\tfrac{1}{2i}(L_{\sigma_\pm} -\overline{L_{\sigma_\pm}})\cap \DT M  &=  \tfrac{1}{2i}( L_{\A_\pm} - L_{\B_\pm} -  \overline{L_{\A_\pm}} + \overline{L_{\B_\pm}})\cap \DT M\\
&=e^{-\omega_+}\Gr(-\pi_A) - e^{-\omega_+}\Gr(\mp\pi_B)\\
&=\Gr(-\pi_A) -\Gr(\mp\pi_B),
\end{align*}
as required. 
\end{proof}
Corollary~\ref{cor: Holomorphic Hitchin and real Poisson} provides an easy way to compute the real and imaginary parts of the Hitchin Poisson structures: the Poisson structure underlying the Manin triple $(\DT M_{H}, \J_A(T^*M), \J_B(T^*M) )$ is computed as follows: for $\xi\in T^*$, $Q\xi = \rho(\J_A\alpha)$, where we have the decomposition
    \[
    \xi = \J_A\alpha + \J_B\beta,
    \]
    for unique $\alpha,\beta\in T^*$.  Taking inner product with $\J_B\eta$, $\eta\in T^*$, we obtain 
    \[
    \IP{\xi,\J_B\eta}=\IP{\J_A\alpha,\J_B\eta}=g^{-1}(\alpha,\eta),
    \]
    so that $\alpha = -g\rho\J_B\xi$, and using~\eqref{poisab}, we obtain the (coincident) real parts of the Hitchin Poisson structures   
    \[
    \Re(\sigma_\pm) = Q = -\pi_A g \pi_B = -\tfrac{1}{4}[I_+,I_-]g^{-1}. 
    \]

\subsection{Generalized K\"ahler structures from Manin triples}

In this subsection, we prove a converse of \autoref{Inf descprtion of GK}, obtaining a new characterization of generalized K\"ahler geometry in terms of a pair of holomorphic Manin triples related by gauge equivalences as in~\eqref{manindiag}. 

We begin with a pair of holomorphic exact Manin triples $(\E_\pm, \A_\pm, \B_\pm)$ over the complex manifolds $X_\pm$.  We assume the corresponding matched pairs are isomorphic as follows:
\[
\begin{tikzcd}
\E_-\oplus \DT_{0,1}X_-\arrow[r,"{\phi_1}","{\cong}"'] & {\E_+}\oplus \DT_{0,1}{X_+}
\end{tikzcd},\qquad
\begin{tikzcd}
    \E_+\oplus \DT_{0,1}X_+\arrow[r,"{\phi_2}","{\cong}"'] & \overline{\E_-}\oplus \DT_{0,1}\overline{X_-}
\end{tikzcd}.
\]
That is, $X_\pm$ have the same underlying smooth manifold, i.e. $X_\pm=(M,I_\pm)$, we may identify our Courant algebroids $\E_\pm$ with the normal forms given in Section~\ref{subsec: hol courant} by the closed 3-forms
\begin{equation}\label{hplm}
\H_\pm \in \Omega^{2,1}(M, I_\pm)\oplus\Omega^{3,0}(M, I_\pm),
\end{equation}
and the isomorphisms $\phi_1,\phi_2$ are gauge transformations by 2-forms $\F_1,\F_2\in \Omega^2(M,\C)$ satisfying
\begin{equation}\label{gaughpm}
\H_+-\H_- = d\F_1,\qquad \overline{\H_-}-\H_+ = d\F_2.
\end{equation}
Taking complex conjugates of the above, we obtain the following diagram of gauge equivalences:
\begin{equation}\label{diagcalg}
\begin{aligned}
\begin{tikzcd}
\overline{\H_-} & \H_+\arrow[l,"\F_2"']\\
\overline{\H_+}\arrow[u,"{-\overline{\F_1}}"] & \H_-\arrow[u,"{\F_1}"']\arrow[l,"{-\overline{\F_2}}"]
\end{tikzcd}
\end{aligned}
\end{equation}
Our final assumption concerning the diagram of Courant algebroids is that the vertical and horizontal gauge transformations should coincide, namely $\F_1 = -\overline{\F_1}$ and $\F_2 = -\overline{\F_2}$, or equivalently 
\begin{equation}\label{horvereq}
\F_1 = iF_1,\qquad \F_2 = iF_2,
\end{equation}
for $F_1,F_2$ a pair of real 2-forms. 
With these isomorphisms of Courant algebroids in place, we assume that the Dirac structures $\A_\pm,\B_\pm$ are related  as in~\eqref{manindiag}.  In short, we assume that 
\begin{align}\label{eq:3}
L_{\A_+} = e^{iF_1} L_{\A_{-}}, \quad \overline{L_{\B_{-}}}= e^{iF_2}L_{\B_+} .
\end{align} 
\begin{theorem}\label{thm: Infinitesimal descrption of a GK structure}
    Let $(\E_\pm,\A_\pm,\B_\pm)$ be a pair of holomorphic exact Manin triples related as above, so that we have two complex structures $I_\pm$ on the same smooth manifold $M$, and the corresponding closed 3-forms $\H_\pm$ as in~\eqref{hplm} are related via~\eqref{gaughpm}, for a pair of real 2-forms $F_1, F_2$ satisfying~\eqref{eq:3}.  
    If we project $F_\pm=\tfrac{1}{2}(-F_1\pm F_2)$ to its components of type $(1,1)$ and $(2,0)+(0,2)$ respectively
    \begin{equation}\label{partsoffpm}
    \begin{aligned}
    \omega_\pm &= \tfrac{1}{2}(F_\pm + I_\pm^*F_\pm I_\pm)\\
    \beta_\pm  &= \tfrac{1}{2}(F_\pm - I_\pm^*F_\pm I_\pm),
    \end{aligned}
    \end{equation} 
    then in the decomposition of $F_\pm I_\pm$ into symmetric and skew-symmetric parts
    \[
    -F_\pm I_\pm = -\omega_\pm I_\pm - \beta_\pm I_\pm = g_\pm + b_\pm,
    \] 
    we obtain the coincidences
    \begin{equation}\label{sameprts}
    g_+ = g_- = g,\qquad b_+ = -b_- = b.
    \end{equation}
    such that Equation~\eqref{dcpm} holds for $H = \Re(\H_+) - db = \Re(\H_-)-db$, i.e.
    \[
    \pm d^c_\pm\omega_\pm =  H.
    \]
    That is, we obtain a generalized K\"ahler structure, providing that $g$ is positive-definite. 
\end{theorem}


\begin{proof}

	From the conditions
		\[
		d(iF_1) = \H_+ - \H_-, \qquad d(iF_2) = \overline{\H_-}-\H_+,
		\]
		we see that  $\Re(\H_+) = \Re(\H_-) = H_0$ and $\Im (\H_\pm) = \mp dF_\pm$.
		By \autoref{Real structures of CAlg}, these trivializations of $(\E_\pm)_I$ generate pluriclosed forms, in particular
 \[
\pm d^c_{\pm} \omega_{\pm} = H_0 \mp db_\pm,
 \] 
for $\omega_\pm$ and $b_{\pm}$ as given in~\eqref{partsoffpm} and \eqref{sameprts}.   It remains to show that \eqref{eq:3} implies the coincidences~\eqref{sameprts}, so that we can set $H = H_0-db_+ = H_0+db_-$ and obtain the result.

Since the imaginary parts of $\E_{\pm}$ are trivial, and $F_\pm$ define gauge transformations as follows:
			\[
   \begin{tikzcd}	(\DT M_{\C}, \H_\pm)\arrow[r,"{e^{\pm iF_{\pm}}}","{\cong}"'] &(\DT_\C M, H_0),
   \end{tikzcd}
			\]
   we may transfer the complex Dirac structures $L_{\A_{\pm}}$ and $L_{\B_{\pm}}$ to the right hand side, 
			\[
			L_A=e^{iF_{+}}L_{\A_{+}} = e^{-iF_{-}}L_{\A_{-}}, \qquad L_B = e^{iF_{+}}L_{\B_{+}} = e^{iF_{-}}\overline{L_{\B_{-}}},
			\]
   where we use~\eqref{eq:3} in the above coincidences.  

   The complex Dirac structures $L_A,L_B$ and $L_A,\overline{L_B}$ have intersection given by
\[
   \begin{aligned}
       \l_+ &= L_A\cap L_B = e^{iF_+}(L_{\A_+} \cap L_{\B_+}) = e^{iF_+}  T_{0,1}X_+\\
       \l_- &= L_A\cap\overline{L_B}=e^{-iF_-}(L_{\A_-}\cap L_{\overline{\B_-}}) = e^{-iF_-} T_{0,1} X_-.
   \end{aligned}
   \]
   We can therefore express $\l_\pm$ as follows:
   \[\begin{aligned}
   \l_\pm &= \{X \pm i F_\pm X\ :\ X\in T_{0,1} X_\pm\}\\
        &=\{X  \mp F_\pm I_\pm X\ :\ X\in T_{0,1} X_\pm\}\\
        &=\{ X \pm (g_\pm + b_\pm) X\ :\ X\in T_{0,1} X_\pm\}.
     \end{aligned}
   \]
   If we now consider the subspaces $V_\pm\otimes\C = \l_\pm \oplus \overline{\l_\pm}$; 
   we have 
   \[
   V_\pm = \{X\pm (g_\pm+ b_\pm) X: X \in TM \}.
   \]
   But since $L_A,L_B$ are maximal isotropic, it follows that $V_+, V_-$ must be orthogonal, which can only hold if $g_+=g_-$ and $b_+=-b_-$, as required.
\end{proof}

\section{Symplectic double groupoids and higher Morita equivalences}\label{section: Symplectic double groupoids}

In the previous section, we expressed a generalized K\"ahler structure as a pair of Manin triples related by gauge transformations.  Manin triples are comprised of Dirac structures, which are Lie algebroids, and so in this sense generalized K\"ahler geometry incorporates infinitesimal structures analogous to Lie algebras.  In Section~\ref{section: integration}, we integrate these to global Lie-theoretic structures: each Manin triple integrates to a  double symplectic Lie groupoid, while the gauge transformations between Manin triples integrate to Morita equivalences between the double groupoids.  In this section, we prepare for the integration by reviewing the known relationship between Manin triples and symplectic double groupoids, and by introducing new types of generalized 1-morphisms as well as Morita 2-equivalences for symplectic double groupoids. The definitions and results in this section hold in both the smooth and holomorphic categories; we omit the corresponding adjectives unless needed. 

Our notation for Lie groupoids is as follows. We denote a Lie groupoid by $G \rightrightarrows M$, where $G$ is the manifold of arrows and $M$ is the manifold of objects. The structure maps of a Lie groupoid are its source, target, inversion, unit inclusion and multiplication, respectively denoted by $\sss,\ttt,\ii,\uu,\mm$; for simplicity we also use the notation $\mm(g,h)=gh$ and $\ii(g)=g^{-1}$. A Lie groupoid $G \rightrightarrows M$ determines an equivalence relation on $M$: for $p\in M$, its corresponding equivalence class is its {\em orbit} which is the immersed submanifold $G\cdot p=\{q\in M\ :\  \exists g\in G, \ttt(g)=q,\sss(g)=p\}$.   We denote the simplicial differential by $\delta^*$; for an $n$-form $\a$ on $M$, $\delta^*\a := \ttt^*\a-\sss^*\a$.

\begin{remark}While every Lie algebroid integrates to a local Lie groupoid \cite{MR4069892}, a global integration need not exist, see \cite{MR1973056}.  We work under the assumption that all our Lie algebroids have corresponding global Lie groupoids, since this is the case for all examples we shall consider in~\S\ref{sec: lie groups}. 
\end{remark}

 \subsection{Quasi-symplectic groupoids} \label{qsg}
 

A Dirac structure in the exact Courant algebroid $(\DT M, H)$ may be viewed as a Lie algebroid $A$ together with a bundle map 
\begin{equation}\label{im2f}
\mu:A\to T^*M
\end{equation}
which, when combined with the anchor $\rho:A\to TM$, defines a bracket-preserving embedding 
\[
(\rho,\mu):A\to \DT M.
\]
In~\cite{MR2068969}, it was proven that if $G\rightrightarrows M$ is a source simply-connected Lie groupoid integrating $A$, then we can obtain the additional map~\eqref{im2f} from a 2-form $\omega\in\Omega^2(G)$ which is uniquely determined by the Dirac structure, via the formula  
\begin{equation}\label{eq:IM 2 form}
    \mu_{\omega} (a)= \uu^*(\iota_a(\omega)).
\end{equation}
In fact, $(G,\omega)$ is known as a quasi-symplectic groupoid, whose definition we now recall.  
\begin{definition}[\cite{MR2068969,MR2153080}]\label{def:quasi symplectic} Let $M$ be a manifold equipped with a closed 3-form $H$ (of type $(2,1)+(3,0)$ in the holomorphic case).  A Lie groupoid $G \rightrightarrows M$ with $\dim G = 2 \dim M$ is {\em quasi-symplectic} if it is equipped with a 2-form $\omega$ on $G$ (of type (2,0), in the holomorphic case) such that: 
\begin{enumerate}
	\item The 2-form $\omega$ is multiplicative, i.e. on $G{_\sss\times}_\ttt G $ we have $\text{pr}_1^*\omega+\text{pr}_2^*\omega=\mm^*\omega$.
 \item The 2-form $\omega$ is relatively closed, i.e. $d\omega = \ttt^*H - \sss^*H$.
\item The intersection $\ker (\omega) \cap \ker(d \sss) \cap \ker(d\ttt)$ vanishes (in the holomorphic case, it has zero intersection with $T_{1,0}G$). 
\end{enumerate} 
\end{definition}

The relationship between quasi-symplectic groupoids and Dirac structures described above also holds in the holomorphic category, as studied in~\cite{MR4480214}. In this case, we have that
$L=\{(\rho(v_{1,0}),\mu_\omega(v_{1,0})):\ v\in A\} \subset (\DT_{1,0}M,H)$ is a holomorphic Dirac structure whose real and imaginary parts are given as follows (from~\cite[\S 6]{MR4480214}). 
\begin{align}
    L_R= (\rho,\mu_{\Re(\omega)})(A), \qquad L_I= (\rho,\mu_{\Im(\omega)})(A); \label{eq:re and im part dirac via q symplectic}
\end{align} 
The clearest characterization of the relation between a quasi-symplectic groupoid $(G,\omega)$ and its corresponding Dirac structure $A$ is that the anchor map 
\begin{equation}\label{sfdm}
\begin{tikzcd}
G \arrow[r, "{(\mathtt{t},\mathtt{s})}"] & M\times M,    &&
\text{Gr}(\omega) \arrow[r, maps to]     & A\times (-A)
\end{tikzcd}
\end{equation}
is a forward Dirac map from $\Gr(\omega)$ to $A\times (-A)$, i.e. the pushforward Dirac structure of $\Gr(\omega)$ to $M \times M$ is $A \times (-A)$. In this situation, $(G,\omega)$ is called an integration of $A$. As shown in \cite{MR2068969,MR4480214}, Any Dirac structure integrable as a Lie algebroid admits such a quasi-symplectic integration.

\subsection{Poisson groupoids}\label{Poisson in terms of quasi-symplectic}

A Manin triple, because it involves a pair of Dirac structures, gives rise to a corresponding pair of quasi-symplectic groupoids, as explained above. In fact, each of these groupoids inherits additional structure as a result of the presence of the other Dirac structure; each groupoid is actually a Poisson groupoid. 
 \begin{definition}[\cite{MR959095,weisygr}] 
A {\em Poisson groupoid} is a Lie groupoid $G \rightrightarrows M$ equipped with a Poisson structure $\pi$ on $G$ such that the graph of the multiplication is a coisotropic submanifold:
\[ \Gr(\mm)=\{(g,h,\mm(g,h))\ :\ \sss(g)=\ttt(h)\}\hookrightarrow(G,\pi) \times (G,\pi) \times (G,-\pi). \]  
A Poisson groupoid with nondegenerate Poisson structure is called a {\em symplectic groupoid}.
\end{definition}

Let $G \rightrightarrows M$ be a Poisson groupoid and $A$ its Lie algebroid. A fundamental result in the theory of Poisson groupoids is that the unit embedding $M \hookrightarrow G$ is coisotropic so that $A^*$ inherits a Lie algebroid structure and $(A\oplus A^*, A, A^*)$ is a Manin triple \cite{MR3451395,MR1746902, MR1262213}. Conversely, if $(E,A, B)$ is a Manin triple and $G_A$ is the source simply connected integration of $A$, then $G_A$ inherits a canonical Poisson groupoid structure. 

Take a Manin triple $(E,A, B)$ such that the Courant algebroid $E$ is exact.
We briefly explain how the corresponding Poisson groupoid structures can be recovered from the integration of $A$ and $B$ as Dirac structures. 
Suppose the Dirac structure $A$ has an integration $(G_A\rightrightarrows M,\omega_A)$ as in~\S\ref{qsg}. It follows from Condition 3 in \autoref{def:quasi symplectic} that the Dirac morphism~\eqref{sfdm}, namely 
\[
(\ttt,\sss): (G_A, \Gr(\omega_A)) \dashrightarrow (M \times M, A \times (-A) )
\]
is a strong Dirac map in the sense of~\cite[Definition 2.4]{MR2642360}. As a consequence of~\cite[Proposition 1.15]{MR2642360}, any complementary Dirac structure on the codomain may be pulled back to obtain a complementary Dirac structure on the domain.  In particular, we obtain a Manin triple over $G_A$:
\begin{equation*}
((\DT G_A,H),\ \Gr(\omega_A),\  (\ttt,\sss)^!(B \times (-B) )),
\end{equation*}
and similarly for an integration $(G_B,\omega_B)$ of $B$. 
The multiplicative Poisson structures $\Pi_A, \Pi_B$ on $G_A, G_B$ are then determined by taking differences as in~\eqref{poismt}.  For instance,
\begin{align}
    \Gr(\Pi_A) = \Gr(\omega_A) - (\ttt,\sss)^!(B \times (-B)) = \Gr(\omega_A) -\ttt^!B+\sss^!B. \label{eq:side poisson via dirac}
\end{align}
See also \cite[Thm. 4.3]{buriglsev}. 

The pair of Poisson groupoids $(G_A,\Pi_A), (G_B,\Pi_B)$ described above ultimately derive from the same Manin triple, and so must be related in a groupoid-theoretic fashion; this relation, found by Weinstein~\cite{MR959095}, is that these groupoids form the \emph{side groupoids} of a symplectic double groupoid.  The remainder of this section is devoted to reviewing and extending the theory of symplectic double groupoids. 

\subsection{Symplectic double groupoids}

A {\em double groupoid} is a groupoid object in the category of groupoids, so it is a diagram of groupoid structures 
\[\begin{tikzcd}
D \arrow[d, shift right] \arrow[d, shift left] \arrow[r, shift right] \arrow[r, shift left] & G_A \arrow[d, shift right] \arrow[d, shift left] \\
G_B \arrow[r, shift right] \arrow[r, shift left]                                            & M                                               
\end{tikzcd}\] in which the horizontal structure maps $D\rightrightarrows G_A$ are groupoid morphisms with respect to the vertical groupoid structures; $G_A$ and $G_B$ are called the side groupoids of $D$. More succinctly, we denote such a double groupoid as $(D, G_A, G_B, M)$ or as $D$ if there is no need to specify the side groupoids. It is convenient to view an element $d \in D$ as in Fig.~\ref{fig:square in double gpd}, where the two vertical sides represent elements in $G_A$, the two horizontal sides represent elements in $G_B$, and the four vertices represent elements in $M$. Let $\sss_A, \ttt_A, \sss_B, \ttt_B$ denote the source and target maps of $D \rightrightarrows G_A$ and $D\rightrightarrows G_B$ respectively. Then in Fig.~\ref{fig:square in double gpd} we have that $a_1 = \sss_A(d)$, $ a_2 = \ttt_A(d)$, $ b_1 = \sss_B(d)$ and $ b_2 = \ttt_B(d)$. Denote by $\mm_A$ and $\mm_B$ the respective multiplications of $D \rightrightarrows G_A$ and $D \rightrightarrows G_B$; $\mm_A$ and $\mm_B$ are also called the horizontal and vertical compositions of $D$, respectively.
\begin{wrapfigure}{r}{4cm} \begin{tikzpicture}[line cap=round,line join=round,>=stealth,x=1cm,y=1cm,scale=0.3]
	\clip(-4.3,-5.12) rectangle (11.7,4.0);
 \draw[fill=black!10] (0.44,-3.78) rectangle (6.86,2.64);
	\draw [<-,line width=0.9pt] (0.44,2.64) -- (6.86,2.64);
	\draw [<-,line width=0.9pt] (6.86,2.64) -- (6.86,-3.78);
	\draw [->,line width=0.9pt] (6.86,-3.78) -- (0.44,-3.78);
	\draw [->,line width=0.9pt] (0.44,-3.78) -- (0.44,2.64);
	\begin{scriptsize}
		
		
		
		
		\draw[color=black] (3.7,3.27) node {$b_2$};
		\draw[color=black] (7.8,-0.4) node {$a_1$};
		\draw[color=black] (3.68,-4.19) node {$b_1$};
		\draw[color=black] (-0.2,-0.4) node {$a_2$};
		\draw[color=black] (3.55,-0.5) node {$d$};
	\end{scriptsize}
\end{tikzpicture}\caption{An element $d\in D$}\label{fig:square in double gpd}\end{wrapfigure}
 A {\em double Lie groupoid} is a double groupoid as above in which all the side groupoid structures are Lie groupoid structures and $(\sss_A,\sss_B):D \rightarrow G_A \times_M G_B$ is a submersion \cite{MR1170713}. Note that, with respect to \cite{MR1170713, MR1174393}, we drop the surjectivity requirement for this map, following \cite{MR1054741}: this allows us to construct examples of double Lie groupoids via the method of Lu-Weinstein in~\S\ref{subsubsection: Lu-Weinstein}.

The final piece of data we need is a symplectic form on the manifold $D$ which is multiplicative in both the horizontal and vertical directions.  This induces Poisson groupoid structures  on the side groupoids $G_A,G_B$, coinciding with the type of structure on the integration of the pair of Dirac structures in a Manin triple. 

\begin{definition}[Symplectic double groupoids \cite{MR959095}]
	A symplectic double groupoid is a double Lie groupoid $(D, G_A, G_B, M)$ equipped with a symplectic form on $D$ that makes both $D \rightrightarrows G_A$ and $D \rightrightarrows G_B$ into symplectic groupoids. 
\end{definition}

It was observed in \cite{MR959095} that any Poisson groupoid $G \rightrightarrows M$ can be embedded as one side of a local symplectic double groupoid $(D, G, G', M)$ and $G$ is in Poisson duality with the other side $G'$ (see \cite[Definition 4.4.1]{MR959095}). Conversely, it was shown in \cite{MR1697617} that the two sides of a symplectic double groupoid are Poisson groupoids in duality. Constructing symplectic double groupoids is difficult in general. However, there is a method to construct symplectic double groupoids for Manin triples sitting in exact Courant algebroids, as we shall see in \S\ref{subsec: lag branes}.

Any double groupoid as above contains a groupoid over $M$ called the \emph{core groupoid} of $D$.  If $D$ is a double symplectic groupoid integrating a Manin triple, then the core is in fact a symplectic groupoid \cite[Thm. 2.9]{MR1697617} which integrates the underlying Poisson structure of the Manin triple. 
 \begin{definition}[Core groupoid]\label{def:sym core} Given a symplectic double groupoid $(D, G_A, G_B, M)$, the space
 $K_D:=\{d\in D\ :\  \sss_A(d)=\uu(p), \sss_B(d)=\uu(p), p\in M\}$ is the {\em (symplectic) core groupoid of $D$} with its symplectic form given by restriction of the symplectic form on $D$. 
\end{definition}


We now explain how to obtain a Manin triple from the symplectic double groupoid $(D, G_A, G_B, M)$. Let $Lie (D) ^A $ be the Lie algebroid of $D \rightrightarrows G_A$ and let $B \to M$ be the Lie algebroid of $G_B \rightrightarrows M$. The tangent functor defines a Lie groupoid structure on $Lie(D)^A$ over $B$ such that the corresponding vector bundle projections determine a groupoid morphism from $Lie(D)^A \rightrightarrows B$ to $G_A \rightrightarrows M$.
Since the structure maps of $Lie(D)^A \rightrightarrows B$ are Lie algebroid morphism, it is a groupoid object in the category of Lie algebroids, such groupoid objects are called {\em LA-groupoids} \cite{MR1174393}. As $(D,\Omega) \rightrightarrows G_A$ is a symplectic groupoid, we have a map $\mu_\Omega:Lie(D)^A\rightarrow T^*G_A$ as in \eqref{eq:IM 2 form}, inducing an isomorphism of vector bundles which is also a groupoid isomorphism:
\begin{equation*}
    \begin{tikzcd}
	Lie(D)^A \arrow[d, shift right] \arrow[d, shift left = 1] \arrow[r, "\cong"] & T^*G_A \arrow[d, shift right] \arrow[d, shift left = 1] \\
	B \arrow[r, "\cong"] & A^*.
\end{tikzcd}\label{eq:iso of LA groupoids}
\end{equation*}

In particular, we obtain that $A$ and $B$ are in duality. We summarize the above discussion into the following proposition.

\begin{proposition}\label{Poisson duality from IM 2-form}
	Let $(D, G_A, G_B, M)$ be a symplectic double groupoid. Let $A$ and $B$ be the Lie algebroids of $G_A$ and $G_B$ respectively. Then 
	\[
	\langle a, b \rangle:=\Omega (a,b), \quad a \in A, b \in B
	\]
	is a nondegenerate pairing between $A$ and $B$. \qed
\end{proposition} 
Using the pairing of Proposition \ref{Poisson duality from IM 2-form}, we get that $B\cong A^*$ and the corresponding Lie brackets extend to a Courant bracket on $A\oplus B$ thus defining a Manin triple $(A\oplus B,A,B)$ (presented as a Lie bialgebroid) \cite{MR1472888,MR1262213}.  

\begin{remark}\label{rem: real and im manin triples of a double} If $D$ is holomorphic symplectic, then we can use $\Re(\Omega)$ and $\Im(\Omega)$ as in \autoref{Poisson duality from IM 2-form} to produce two different real Manin triples $((A\oplus B)_R,A,B)$ and $((A\oplus B)_I,A,B)$.
\end{remark}

\subsection{Double Lie groupoids and their generalized morphisms}

A key aspect of the theory of groupoids is that beyond the usual notion of morphism or functor, there are generalized morphisms and, in particular, Morita equivalences, between groupoids.  These generalized morphisms involve spans or zigzags of groupoid morphisms, which themselves may be related by maps, leading to the fact that groupoids form a 2-category.   In the remainder of this section, we develop the theory of generalized morphisms between double groupoids; we shall see that double groupoids may be related by \emph{horizontal} or $(1,0)$-morphisms as well as \emph{vertical} or $(0,1)$-morphisms, and that sequences or ``paths'' of such morphisms may be related by ``homotopies'' or $(1,1)$-morphisms.  This higher categorical structure on double groupoids is directly applicable to generalized K\"ahler geometry, as we shall see in~\S\ref{section: integration}.  To develop such higher morphisms of double groupoids, we first recall the notion of an action of a double Lie groupoid \cite[Definition 1.5]{MR1170713}.
\begin{definition}\label{def: hor action}
	A {\em right horizontal action of a double Lie groupoid $(D, G_A, G_B, M)$ on a Lie groupoid morphism $p: (\mathcal{W} \rightrightarrows W) \to (G_A \rightrightarrows M)$} (called the moment map) is a Lie groupoid morphism called the action morphism $\mathcal{W} \times_{G_A} D \to \mathcal{W}$ over $W \times_{M} G_B \to W$, where
	\begin{align*}
	&\mathcal{W} \times_{G_A} D = \{(x, d)\in \mathcal{W} \times D \ :\  p(x) = \ttt_A(d)\}, \quad W\times_M G_B = \{(w, b)\in W\times G_B\ :\  p(w)= \ttt(b)\};
	\end{align*}
	and the groupoid structure of $\mathcal{W} \times_{G_A} D \rightrightarrows W\times_M G_B$ is the one obtained from the product Lie groupoid $\mathcal{W} \times D \rightrightarrows W \times G_B$ by restriction.  
\end{definition}
There is an analogous corresponding definition of left and vertical actions of double Lie groupoids. To simplify notation, we will often abbreviate $\mathcal{W}=(\mathcal{W} \rightrightarrows W)$, $G_A=(G_A\rightrightarrows M)$ and $G_B=(G_B\rightrightarrows M)$.
 Also, a horizontal action of a double Lie groupoid $(D, G_A, G_B, M)$ on $p: \mathcal{W} \to G_A$ as defined above may be encapsulated in the action double Lie groupoid $\mathcal{W} \rtimes_{G_A} D=(\mathcal{W} \times_{G_A} D, \mathcal{W},W \times_M G_B,W )$. 
 
 \begin{definition}
 	Let $q: (\mathcal{W} \rightrightarrows W) \to (\mathcal{Q} \rightrightarrows Q)$ be a surjective submersive Lie groupoid morphism. A horizontal action of $(D, G_A, G_B, M)$ on $p: (\mathcal{W} \rightrightarrows W) \to (G_A \rightrightarrows M)$ is \emph{principal} if the double groupoid morphism $\mathcal{W} \rtimes_{G_A} D \to \mathcal{W} \times_{q} \mathcal{W}$ is an isomorphism, where $\mathcal{W} \times_q \mathcal{W}$ is the submersion double groupoid $(\mathcal{W} \times_q \mathcal{W},\mathcal{W} ,W\times_q W,W)$. 
 \end{definition}
  
\begin{definition}\label{def: hor morphism}
	A {\em horizontal generalized morphism (or $(1,0)$-morphism)} $\mathcal{W}: D \dashrightarrow D'$ from $(D, G_A, G_B, M) $ to $(D', G_{A'}, G_{B'}, M')$ consists of the following data: 
 \begin{equation}
\begin{tikzcd}
	D' \arrow[d, shift right] \arrow[r, shift left] \arrow[r, shift right] \arrow[d, shift left] & G_{A'} \arrow[d, shift right] \arrow[d,shift left] & \mathcal{W}  \arrow[r, "p"] \arrow[d, shift right] \arrow[d,shift left] \arrow[l, "p'"'] & G_A \arrow[d, shift right] \arrow[d,shift left] & D \arrow[d, shift right] \arrow[d,shift left] \arrow[l, shift right] \arrow[l,shift left] \\
	G_{B'}\arrow[r, shift right] \arrow[r,shift left]                                  & M'                                & W \arrow[r] \arrow[l]                                  & M                                & G_B \arrow[l, shift right] \arrow[l,shift left]            \mathrlap{.}                     
\end{tikzcd} \label{eq:hormorequ}  
\end{equation}
 \begin{enumerate}
     \item a Lie groupoid $\mathcal{W}  \rightrightarrows W$ equipped with two Lie groupoid morphisms $p: \mathcal{W}  \to G_A$ and $p': \mathcal{W}  \to G_{A'}$;
     \item a horizontal right $D$-action on $p$ and a horizontal left $D'$-action on $p'$,
 \end{enumerate} 
 such that both actions commute, and the left action is principal with $p$ as its corresponding quotient map.
 In this situation, $\mathcal{W}$ is called a {\em horizontal multiplicative bimodule}. If both actions are principal and the quotient maps coincide with the moment maps, then we say that $P$ is a biprincipal bibundle for a {\em horizontal Morita equivalence} of double Lie groupoids. \end{definition}

Of course, the three definitions above have vertical analogues; corresponding to vertical actions, we have vertical generalized morphisms or $(0,1)$-morphisms and vertical Morita equivalences of double Lie groupoids. Note that the trivial horizontal or vertical Morita equivalence from a double Lie groupoid to itself is the one in which the bimodule is a copy of the same groupoid, and the actions are given by horizontal or vertical multiplication on the left and right.  For a horizontal or vertical generalized morphism to be isomorphic to this trivial one, it suffices to find a bisection of it, defined as follows. 

\begin{definition}[Multiplicative bisection]\label{def:mulbis}
	Let $(D, G_A, G_B, M)$ be a double Lie groupoid that acts horizontally on a Lie groupoid morphism $p: \mathcal{W}  \to G_A$. A \emph{multiplicative} section of $p$ is a groupoid morphism $b: G_A \to \mathcal{W} $ such that $p \circ b =\id$; if $\mathcal{W} $ is a bimodule for a Morita equivalence as in \eqref{eq:hormorequ}, $b: G_A \to \mathcal{W}$ is a {\em multiplicative bisection} if it is a groupoid morphism such that $p'\circ b$ is an isomorphism. Vertical multiplicative (bi)sections are defined analogously.
\end{definition}
For simplicity, we will sometimes identify bisections with their images. The following result is immediate.
\begin{proposition}\label{prop: isomorphism induced by bisection}
	Let $P$ be a horizontal  Morita equivalence between double Lie groupoids $D$ and $D'$ as in \eqref{eq:hormorequ}. Then a multiplicative bisection $b: G_{A} \to \mathcal{W}$ induces an isomorphism of double Lie groupoids $\Psi:D \to D'$ given by $\Psi(d) \cdot b(\sss_A(d))=b(\ttt_{A}(d))\cdot d$ for all $d\in D$. \qed
\end{proposition}
In order to relate horizontal and vertical generalized morphisms, we introduce the concept of $(1,1)$-morphism, which is a special kind of 2-morphism (or 2-cell) in the sense of higher category theory. Invertible $(1,1)$-morphisms are the {\em Morita 2-equivalences} mentioned in the introduction. 

\begin{definition}\label{def: 11 morphism}
	A $(1,1)$-morphism of double Lie groupoids consists of the following data. 
 \begin{equation}
\begin{tikzcd}
	D_{1,1}                             & \mathcal{W} \arrow[l] \arrow[r, two heads]                              & D_{0,1}                             \\
	\mathcal{U} \arrow[u] \arrow[d, two heads] & \mathcal{C} \arrow[u] \arrow[d,two heads] \arrow[l] \arrow[r, two heads] & \mathcal{Z} \arrow[d, two heads] \arrow[u ] \\
	D_{1,0}                             & \mathcal{V} \arrow[l] \arrow[r, two heads]                              & D_{0,0}                         
\end{tikzcd}\hspace{5em}
\begin{tikzpicture}[baseline={(0,0)},scale=0.85,every node/.style={transform shape}]
\pgfdeclarelayer{background}
\pgfdeclarelayer{foreground}
\pgfsetlayers{background,main,foreground}
 \coordinate (a) at (-1,-1);
 \coordinate (b) at (1,-1);
 \coordinate (c) at (1,1);
 \coordinate (d) at (-1,1);
 \coordinate (e) at (0,0);
\coordinate (f) at (0,1);
\coordinate (g) at (0,-1);
\coordinate (h) at (1,0);
\coordinate (i) at (-1,0);
\begin{pgfonlayer}{background}
      \draw[fill=black!10] (a)--(b)--(c)--(d)--cycle;
\draw[postaction={on each segment={mid arrow=black}}] (b)--(c);
\draw[postaction={on each segment={mid arrow=black}}] (b)--(a); 
\draw[postaction={on each segment={mid arrow=black}}] (a)--(d);
\draw[postaction={on each segment={mid arrow=black}}] (c)--(d);
            \end{pgfonlayer}
 \node [below left] at (a)
                {${D}_{1,0}$};
\node [circle, fill=black, inner sep=1pt] at (a){};
\node [below right] at (b)
                {$D_{0,0}$};
\node [circle, fill=black, inner sep=1pt] at (b){};                
\node [above right] at (c)
                {$D_{0,1}$};
\node [circle, fill=black, inner sep=1pt] at (c){};
\node [above left] at (d)
                {${D}_{1,1}$};
\node [circle, fill=black, inner sep=1pt] at (d){};
\node  at (e)
                {$\mathcal{C}$};
\node [above] at (f)
                {$\mathcal{W}$};
\node[below] at (g)
                {$\mathcal{V}$};
\node[right] at (h)
                {$\mathcal{Z}$};
\node[left] at (i)
                {$\mathcal{U}$};
\end{tikzpicture}
\label{eq:(1,1) morphism}
\end{equation}
	\begin{enumerate}
		\item Vertical and horizontal generalized morphisms 
		\[\mathcal{Z}: D_{0,0} \dashrightarrow  D_{0,1} , \qquad  \mathcal{U}: D_{1,0} \dashrightarrow D_{1,1},\qquad  \mathcal{V}: D_{0,0} \dashrightarrow D_{1,0}, \qquad  \mathcal{W}: D_{0,1} \dashrightarrow D_{1,1}. \]
		\item A smooth manifold $\mathcal{C}$ which simultaneously carries generalized morphisms of Lie groupoids $\mathcal{C}: \mathcal{Z} \dashrightarrow \mathcal{U}$ and $\mathcal{C}: \mathcal{V} \dashrightarrow \mathcal{W}$; we call $\mathcal{C}$ the {\em double bimodule} of the $(1,1)$-morphism. \end{enumerate}
		These must be compatible in the sense that the following interchange rule holds: 
		\[
		(c\cdot z)\cdot (v \cdot d) = (c\cdot v)\cdot(z\cdot d),
		\]
		for all $c \in \mathcal{C}, v \in \mathcal{V}, z \in \mathcal{Z}$ and $d \in D_{0,0}$ whenever both sides are defined. As a consequence, the effect of such vertical and horizontal actions can be unambiguously denoted as a concatenation of squares as in the diagram below.
  \[ \adjustbox{scale=0.85,center}
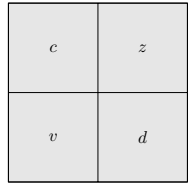
{\begin{tikzpicture}[scale=0.7,every node/.style={transform shape}]
  \draw[fill=black!10] (0,0) rectangle (4,4);
  \draw[step=2.0,black,thin] (0, 0) grid (4.0, 4.0);
  \path
    (1,1) node {$v$}
(3, 1) node {$d$}
    (3, 3) node {$z$}
    (1, 3) node {$c$} 
  ;
\end{tikzpicture}} \]
Additionally, there are three analogous interchange rules corresponding to the actions of the other double groupoids.
 A $(1,1)$-morphism is {\em invertible} if all the generalized morphisms that determine it are Morita equivalences (i.e. they are given by biprincipal bibundles). 
\end{definition}
Each double-headed arrow on the left side of~\eqref{eq:(1,1) morphism} labels the quotient map of a principal action; however, all the generalized morphisms we use in the following are Morita equivalences, so we will drop this notation in the rest of the paper.  Also, the directions of the generalized morphisms are indicated on the right in \eqref{eq:(1,1) morphism}.

Usual Lie groupoids constitute the objects of a bicategory in which generalized morphisms are the 1-morphisms and biequivariant maps of bibundles serve as 2-morphisms, see \cite[Prop. 2.12]{MR2439561}. Similarly, we have morphisms between $(1,1)$-morphisms which are given by equivariant maps between the corresponding pairs of bimodules, as in Fig.~\ref{fig:morphism between 2-cells}.
\begin{definition} A morphism between $(1,1)$-morphisms is given by:
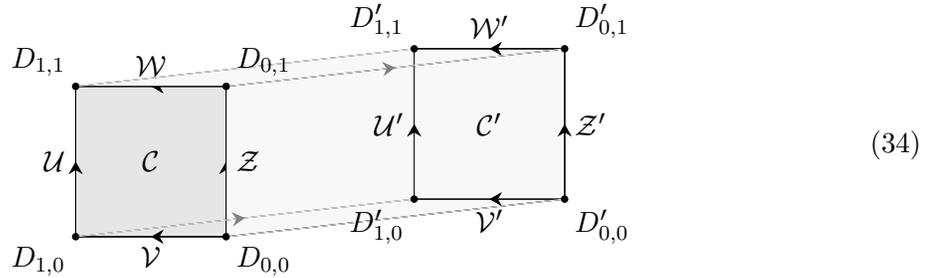
\begin{equation}\label{fig:morphism between 2-cells}
\begin{aligned}
    \centering
   \begin{tikzpicture}[scale=0.75,every node/.style={transform shape}]
\pgfdeclarelayer{background}
\pgfdeclarelayer{foreground}
\pgfsetlayers{background,main,foreground}
 \coordinate (a) at (-1,-1);
 \coordinate (b) at (1,-1);
 \coordinate (c) at (1,1);
 \coordinate (d) at (-1,1);
 \coordinate (e) at (0,0);
\coordinate (f) at (0,1);
\coordinate (g) at (0,-1);
\coordinate (h) at (1,0);
\coordinate (i) at (-1,0);
\coordinate (j) at (-1+4.5,-1+0.5);
 \coordinate (k) at (1+4.5,-1+.5);
 \coordinate (l) at (1+4.5,1+.5);
 \coordinate (m) at (-1+4.5,1.5);
 \coordinate (n) at (4.5,0.5);
\coordinate (o) at (4.5,1.5);
\coordinate (p) at (4.5,-0.5);
\coordinate (q) at (5.5,0.5);
\coordinate (r) at (3.5,0.5);
 
 \begin{pgfonlayer}{background}
      \draw[fill=black!10] (a)--(b)--(c)--(d)--cycle;
\draw[postaction={on each segment={mid arrow=black}}] (b)--(c);
\draw[postaction={on each segment={mid arrow=black}}] (b)--(a); 
\draw[postaction={on each segment={mid arrow=black}}] (a)--(d);
\draw[postaction={on each segment={mid arrow=black}}] (c)--(d);
\draw[fill=black!10] (j)--(k)--(l)--(m)--cycle;
      \draw[fill=black!3,opacity=0.5,ultra thin] (b)--(k)--(l)--(c)--cycle;
      \draw[fill=black!3,opacity=0.5] (c)--(l)--(m)--(d)--cycle;
      \draw[opacity=0.5] (a)--(b)--(k)--(j)--cycle;
\draw[postaction={on each segment={mid arrow=black}}] (k)--(l);
\draw[postaction={on each segment={mid arrow=black}}] (l)--(m); 
\draw[postaction={on each segment={mid arrow=black}}] (j)--(m);
\draw[postaction={on each segment={mid arrow=black}}] (k)--(j);
\draw[draw=gray!5,postaction={on each segment={mid arrow=gray}},opacity=0.4] (c)--(l);
\draw[draw=gray!5,opacity=0.4] (b)--(k);
\draw[draw=gray!5,postaction={on each segment={mid arrow=gray}},opacity=0.4] (a)--(j);
\draw[draw=gray!5,opacity=0.4] (d)--(m);
            \end{pgfonlayer}
 \node [below left] at (a)
                {${D}_{1,0}$};
\node [circle, fill=black, inner sep=1pt] at (a){};
\node [below right] at (b)
                {$D_{0,0}$};
\node [circle, fill=black, inner sep=1pt] at (b){};                
\node [above right] at (c)
                {$D_{0,1}$};
\node [circle, fill=black, inner sep=1pt] at (c){};
\node [above left] at (d)
                {${D}_{1,1}$};
\node [circle, fill=black, inner sep=1pt] at (d){};
\node  at (e)
                {$\mathcal{C}$};
\node [above] at (f)
                {$\mathcal{W}$};
\node[below] at (g)
                {$\mathcal{V}$};
\node[right] at (h)
                {$\mathcal{Z}$};
\node[left] at (i)
                {$\mathcal{U}$};
                 \node [below left] at (j)
                {${D}_{1,0}'$};
\node [circle, fill=black, inner sep=1pt] at (j){};
\node [below right] at (k)
                {$D_{0,0}'$};
\node [circle, fill=black, inner sep=1pt] at (k){};                
\node [above right] at (l)
                {$D_{0,1}'$};
\node [circle, fill=black, inner sep=1pt] at (l){};
\node [above left] at (m)
                {${D}_{1,1}'$};
\node [circle, fill=black, inner sep=1pt] at (m){};
\node  at (n)
                {$\mathcal{C}'$};
\node [above] at (o)
                {$\mathcal{W}'$};
\node[below] at (p)
                {$\mathcal{V}'$};
\node[right] at (q)
                {$\mathcal{Z}'$};
\node[left] at (r)
                {$\mathcal{U}'$};
\end{tikzpicture} 
\end{aligned}  \qquad \begin{bmatrix} \phi_{1,1} & \phi_W & \phi_{0,1} \\
\phi_U &\phi_{\mathcal{C}} & \phi_Z \\
\phi_{1,0}& \phi_V & \phi_{0,0}  \end{bmatrix} \end{equation}
\begin{enumerate}
    \item Double Lie groupoid morphisms $\phi_{i,j}:D_{i,j}\rightarrow D_{i,j}'$ for $i,j=0,1$,
    \item Lie groupoid morphisms $\phi_{U}:\mathcal{U}\rightarrow \mathcal{U}'$, $\phi_{V}:\mathcal{V}\rightarrow \mathcal{V}'$, $\phi_{W}:\mathcal{W}\rightarrow \mathcal{W}'$ and  $\phi_{Z}:\mathcal{Z}\rightarrow \mathcal{Z}'$ which are equivariant for the actions by $D_{i,j}$ via the $\phi_{i,j}$, for $i,j=0,1$, 
    \item A map $\phi_{\mathcal{C}}: \mathcal{C}\rightarrow \mathcal{C}'$ which is equivariant for the actions by $\mathcal{U}, \mathcal{V}, \mathcal{W}, \mathcal{Z}$ via $\phi_U,\phi_V,\phi_W,\phi_Z$.
\end{enumerate}   
\end{definition}
Note that $(1,1)$-morphisms can be composed vertically and horizontally, simply by relying on the usual composition for generalized morphisms of Lie groupoids. These compositions are associative up to canonical isomorphisms between $(1,1)$-morphisms just as in the case of Lie groupoids. This means that we can take isomorphisms classes of $(1,1)$-morphisms as our 2-morphisms to have strict associativity and inverses, leading to the following result.
\begin{theorem}
	Double Lie groupoids, isomorphisms classes of horizontal/vertical generalized morphisms and isomorphisms classes of $(1,1)$-morphisms form a double category. \qed
\end{theorem} 
\begin{remark} One can prove the Theorem above by adapting the proofs of \cite[Prop. 2.12, Prop. 2.12]{MR2439561}. Just as generalized morphisms of Lie groupoids can be expressed in terms of {\em spans} \cite[\S 7.2]{MR2178041}, our $(1,1)$-morphisms can be rewritten as follows. We can replace all the bimodules in the picture with action double groupoids and thus we obtain {\em double spans}, these are commutative diagrams like \eqref{eq:(1,1) morphism} in which all the arrows are double Lie groupoid morphisms. By doing this, double Lie groupoids with $(1,0)$, $(0,1)$ and $(1,1)$-morphisms (without passing to isomorphism classes of these) can be described using the theory of {\em double bicategories} \cite{MR2591970}. 
\end{remark}

\begin{definition}\label{def:triv(1,1)} A {\em vertical trivialization} of an invertible $(1,1)$-morphism as \eqref{eq:(1,1) morphism} is given by a pair of multiplicative bisections $\Lambda_{U}\subset \mathcal{U}$ and $\Lambda_{Z}\subset \mathcal{Z}$ together with a bisection $\Lambda_\text{Ver}\subset \mathcal{C} $ for the $\mathcal{W}$ and $\mathcal{V}$ actions such that it induces an identification $\mathcal{V} \rightarrow \mathcal{W}$ which is equivariant with respect to the corresponding isomorphisms of double Lie groupoids $D_{0,0}\rightarrow D_{0,1}$ and $D_{1,0}\rightarrow D_{1,1}$ determined by $\Lambda_U, \Lambda_Z$. We may similarly define {\em horizontal trivializations} given by $\Lambda_W\subset \mathcal{W}$, $\Lambda_V\subset \mathcal{V}$, and $\Lambda_\text{Ver}\subset \mathcal{C}$. 
\[
\adjustbox{scale=0.85,center}{
    \begin{tikzcd}[scale=1,every node/.style={transform shape}]
	D_{1,1} \arrow[r, shift left] \arrow[r, shift right] \arrow[d, shift left] \arrow[d, shift right] & G_{A_{1,1}} \arrow[d, shift left] \arrow[d, shift right] & \mathcal{W} \arrow[d, shift right] \arrow[d, shift left] \arrow[r] \arrow[l] & G_{A_{0,1}} \arrow[l,"\Lambda_{W}"', bend right]\arrow[d, shift left] \arrow[d, shift right] & D_{0,1} \arrow[d, shift left] \arrow[d, shift right] \arrow[l, shift left] \arrow[l, shift right] \\
	G_{B_{1,1}} \arrow[r, shift left] \arrow[r, shift right]                                  & M_{1,1}                                  & W \arrow[r] \arrow[l]                                  & M_{0,1}   \arrow[l, bend right]                               & G_{B_{0,1}} \arrow[l, shift right] \arrow[l, shift left]                                  \\
	\mathcal{U} \arrow[u] \arrow[d] \arrow[r, shift left] \arrow[r, shift right]               & U \arrow[d] \arrow[u]               & \mathcal{C} \arrow[u] \arrow[d] \arrow[r] \arrow[l]                  & Z \arrow[l,"\Lambda_\text{Hor}"', bend right]\arrow[d] \arrow[u]           & \mathcal{Z} \arrow[l, shift left] \arrow[l, shift right] \arrow[u] \arrow[d]               \\
	G_{B_{1,0}} \arrow[u,"\Lambda_{U}", bend left]\arrow[r, shift right] \arrow[r, shift left]                                  & M_{1,0} \arrow[u, bend left]                                 & V\arrow[u,"\Lambda_\text{Ver}", bend left] \arrow[r] \arrow[l]                                  & M_{0,0}\arrow[u, bend right]  \arrow[l, bend left]                                 & G_{B_{0,0}}\arrow[u,"\Lambda_{Z}"', bend right] \arrow[l, shift left] \arrow[l, shift right]                                  \\
	D_{1,0} \arrow[u, shift right] \arrow[u, shift left] \arrow[r, shift right] \arrow[r, shift left] & G_{A_{1,0}} \arrow[u, shift right] \arrow[u, shift left] & \mathcal{V} \arrow[u, shift right] \arrow[u, shift left] \arrow[r] \arrow[l] & G_{A_{0,0}}\arrow[l,"\Lambda_{V}", bend left] \arrow[u, shift left] \arrow[u, shift right] & D_{0,0} \arrow[u, shift left] \arrow[u, shift right] \arrow[l, shift left] \arrow[l, shift right] 
\end{tikzcd}  }\]
An invertible $(1,1)$-morphism is {\em trivialized} if it admits vertical and horizontal trivializations as in the above diagram, such that they induce a commutative diagram of identifications relating all the spaces displayed in \eqref{eq:(1,1) morphism}. 
\end{definition}
\begin{example}\label{ex:triv(1,1)mor} The simplest example of a $(1,1)$-morphism is the trivial one in which all the double Lie groupoids and bimodules in \eqref{eq:(1,1) morphism} are equal to a fixed double Lie groupoid, the actions consist of left and right horizontal and vertical multiplications, and the multiplicative bisections are the horizontal and vertical identity bisections. 
\end{example}  
\subsection{Symplectic Morita equivalences and multiplicative bisections}
To complete our treatment of the Morita theory of double Lie groupoids, we explain how the morphisms introduced above interact with the symplectic geometry of symplectic double groupoids. 

 We say that a horizontal action as in \autoref{def: hor action} is \emph{symplectic} if $D$ is a symplectic double groupoid, $\mathcal{W}$ is a symplectic groupoid and the action morphism is {\em symplectic}, in the sense that its graph $\{(w,d,w\cdot d)\}\subset \mathcal{W} \times D \times (-\mathcal{W})$ is Lagrangian, where $(-\mathcal{W})$ denotes taking the negative of the symplectic form on $\mathcal{W}$. Similarly, \autoref{def: hor morphism} and \autoref{def: 11 morphism} can be naturally extended to include symplectic structures. 
 \begin{definition} A $(1,1)$-morphism as in~\autoref{def: 11 morphism} between symplectic double groupoids is {\em symplectic} if $\mathcal{U}, \mathcal{V}, \mathcal{W}, \mathcal{Z}$ have symplectic structures such that all horizontal and vertical morphisms are symplectic, and the double bimodule $\mathcal{C}$ has a symplectic stucture such that all the actions on $\mathcal{C}$ are symplectic.  
 \end{definition} 

Finally, we explain how the bisections of~\autoref{def:triv(1,1)} interact with the symplectic forms.  
\begin{proposition}\label{prop: change of the sympletcic forms}
	Let $(\mathcal{W},\Omega_{\mathcal{W}}):(D,\Omega)\dashrightarrow (D',\Omega')$ be a horizontal symplectic Morita equivalence equipped with a multiplicative bisection $b: G_{A} \to \mathcal{W}$, which induces isomorphisms $\Phi: D \to \mathcal{W}$ and $\Psi: D \to D'$.  Then we have the following relation between symplectic forms: 	
 \[
\Phi^*\Omega_\mathcal{W} - \Omega = \ttt_{A}^*F, \qquad \Psi^* \Omega' - \Omega = \ttt_{A}^*F - \sss_{A}^*F,
\]
where $F = b^*\Omega_\mathcal{W}$. In particular, $\Phi,\Psi$ are symplectomorphisms if $b$ is Lagrangian. 
\end{proposition}
\begin{proof}
	Consider the map $(b\circ \ttt_{A}, \id, \Phi): D \to \mathcal{W} \times D \times \mathcal{W}$,
	whose image is in the graph $\Gamma$ of the  action. As $\Gamma$ is Lagrangian in $\mathcal{W} \times D \times (-\mathcal{W})$, we have $(b\circ \ttt_A)^*\Omega_\mathcal{W} + \Omega - \Phi^*\Omega_\mathcal{W} = 0$, i.e.
	\begin{align}\label{eq: compare the symplectic forms 1}
		\Phi^*\Omega_\mathcal{W} - \Omega = \ttt_{A}^*F. 
	\end{align}
	Similarly, let $\Phi': D' \to \mathcal{W}$ be the isomorphism induced by $b$ and let $b': G_{A'} \to \mathcal{W}$ be the inverse of $p'|_{b(G_{A})}$, then 
	\begin{align}\label{eq: compare the symplectic forms 2}
		(\Phi')^*\Omega_\mathcal{W} - \Omega' = \sss_{A'}^*(b')^*\Omega.
	\end{align}
	As $\Psi = (\Phi')^{-1}\circ \Phi$, by \eqref{eq: compare the symplectic forms 1} and \eqref{eq: compare the symplectic forms 2}, we then have the required relation 
	\[
	\Psi^*\Omega' = \ttt^*_{A}F + \Omega - \Phi^*((\Phi')^{-1})^*\sss_{A'}^*(b')^*F = \Omega + \ttt_{A}^*F- \sss_{A}^*F.\vspace{-2em}
	\]
\end{proof}

An important aspect of symplectic Morita equivalences between double symplectic groupoids is that their Lagrangian bisections may be deformed by Hamiltonian flows.  This will be particularly important in~\S\ref{subsec: GK deform}, where such deformations will induce deformations of generalized K\"ahler metrics.  

Let $\mathcal{W}:D\dashrightarrow D'$ be a horizontal symplectic Morita equivalence as in \eqref{eq:hormorequ}, equipped with a multiplicative Lagrangian bisection $L_0$, which induces an isomorphism $\Psi_0: D \to D'$. We may deform the bisection by a Hamiltonian flow: let $f \in C^{\infty}(M)$ and suppose that the Hamiltonian vector field of  $p^*\delta^*f \in C^{\infty}(\mathcal{W})$ is complete, where $\delta^*f = \ttt^*f - \sss^*f \in C^{\infty}(G_{A})$ is the simplicial pullback of $f$ to $G_{A}$. Applying the flow  of $X_{p^*\delta^*f}$ to $L_0$, we obtain a family of bisections $L_t$ and induced isomorphisms $\Psi_t: D \to D'$.  We now show that $\Psi_t$ may be factorized as the composition of $\Psi_0$ with a Hamiltonian automorphism of $D$ generated by the same function, but pulled back to $D$.

\begin{corollary}\label{The induced isomorphisms} 
Let $\mathcal{W}:D\dashrightarrow D'$ be a horizontal symplectic Morita equivalence with Lagrangian bisection $L_0$, and use the Hamiltonian flow of $f\in C^\infty(M)$ to deform $L_0$ as above, obtaining a family $L_t$ of bisections, inducing the isomorphisms $\Psi_t:D\to D'$.  Then we have the following factorization:
	\[
	\Psi_t = \Psi_0 \circ \Upsilon_{t},
	\]
where $\Upsilon_t$ be the family of automorphisms of $D$ given by the Hamiltonian flow of 
\[
\delta_{A}^*\delta^*f = \ttt_{A}^*\delta^*f-\sss_{A}^*\delta^*f\in C^{\infty}(D,\R).
\]
\end{corollary}
\begin{proof} Let $\Phi: D \to \mathcal{W}$ be the diffeomorphism induced by $L_0$, which is a symplectomorphism by \autoref{prop: change of the sympletcic forms}. So $\Phi$ intertwines the Hamiltonian flows of $p^* \delta^*f$ and $\sss_A^*\delta^*f$ since $\Psi^*p^* \delta^*f = \sss_A^*\delta^*f$. Now we have that the Hamiltonian vector field of $\sss^*_A \delta^*f $ (which is the left-invariant extension $(d\delta^* f)^l$) produces a bisection $B_t$ in $D$, obtained by flowing $\texttt{u}(G_A)$ for time $t$. By definition, $\Phi$ maps $B_t$ to $L_t$.
Now notice that $\Upsilon_t$ coincides with conjugation by the bisection $B_t$ and, by our previous observation, $L_t$ is obtained as the composite bisection $L_0\ast B_t$ resulting from composing $G_{A'}\xleftarrow{p'} \mathcal{W} \xrightarrow{p} G_{A}$ with the self Morita equivalence of $G_{A}$ given by $D$. Since conjugation takes composition of bisections to composition of isomorphisms, we obtain the result. \end{proof} 

\section{Integration of generalized K\"ahler structures}\label{section: integration}

In this section, we solve the problem of integration for generalized K\"ahler structures. As shown in \S\ref{section: The holomorphic Manin triples of GK structures}, a generalized K\"ahler structure determines a pair of holomorphic Manin triples related by gauge equivalences. On the other hand, we know from \S\ref{section: Symplectic double groupoids} that these Manin triples are the infinitesimal counterparts of symplectic double groupoids. We now show that the gauge equivalences  between Manin triples determine horizontal and vertical Morita equivalences between the corresponding symplectic double groupoids, which assemble to produce a holomorphic symplectic $(1,1)$-morphism.  We also observe that the symmetry of the infinitesimal data in \autoref{Inf descprtion of GK} under complex conjugation leads to the existence of a real structure on the corresponding $(1,1)$-morphism, rendering it what we call~\emph{self adjoint}.  Finally, we show that the $(1,1)$-morphism associated to a generalized K\"ahler structure comes equipped with a Lagrangian trivialization of its real part, a fact which is essential in~\S\ref{sec:gk potential} for the existence of a generalized K\"ahler potential.  The full structure of the integration is summarized in~\autoref{thm: the (1,1) morphism of a GK}.

The solution we describe, since it requires the integration of an exact Manin triple, involves holomorphic symplectic manifolds with four times the dimension of the base manifold.  A recent approach to generalized K\"ahler geometry in terms of $(2,2)$ superfields also involves quadrupling the dimension of the space \cite{hull2024n22superfieldsgeometryrevisited}; it would be interesting to connect these approaches if possible.

\subsection{The Poisson groupoids of a generalized K\"ahler manifold}\label{subsec:poi gro gk man}
Let $(M, \J_A, \J_B)$ be a Generalized K\"ahler manifold and let $(\E_{\pm}, \A_{\pm}, \B_{\pm})$ be the associated holomorphic Manin triples over $X_{\pm} = (M, I_{\pm})$. We suppose that $\A_{\pm}$, $\B_\pm$ integrate to source-simply-connected quasi-symplectic groupoids $(G_{\A_{\pm}},\Omega_\pm^A)\rightrightarrows X_\pm$ and $(G_{\B_{\pm}},\Omega_\pm^B)\rightrightarrows X_\pm$, respectively. Let $G_A$ and $G_B$ be the underlying real Poisson groupoids of $G_{\A_{-}}$ and $G_{\B_{-}}$, respectively. By \autoref{cor: the underlying re and im}, $\A_{+}$ and $\A_{-}$ (respectively, $\B_{+}$ and $\B_{-}$) have the same underlying real parts. As a result, $G_A$ and $G_B$ are also the underlying real Poisson groupoids of the holomorphic Poisson groupoids $G_{\A_+}$ and $G_{\B_+}$ corresponding to $(\E_+, \A_+,\B_+)$. With the above preparation, we now state the first important result that reveals the structure of $G_A$ and $G_B$.
\begin{proposition}\label{prop: deg symplectic GK}
	The holomorphic Poisson groupoid structures $\Pi^A_{\pm}$ and $\Pi^B_{\pm}$ on $G_{\A_\pm}$ and $G_{\B_\pm}$ constructed via \eqref{eq:side poisson via dirac} have gauge equivalent matched pairs:
 \[
	L_{\Pi_+^B} = e^{2iF_B}L_{\Pi_{-}^B}, \quad L_{\Pi_+^A} = e^{2iF_A} L_{\overline{\Pi_-^A}},
	\]
 where $F_A$, $F_B$ are symplectic forms integrating, respectively, $-\pi_A$ and $-\pi_B$ as in \eqref{poisab}.
\end{proposition}
\begin{proof}
	We prove the result for $G_{\B_\pm}$, as the case of $G_{\A_\pm}$ is similar.
As explained in \S\ref{qsg}, $G_{\B_\pm}$ is equipped with quasi-symplectic structure $\Omega_{\pm}^B\in \Omega^{(2,0)_\pm}(G_{\B_{\pm}})$ integrating the holomorphic Dirac structure $\B_{\pm}$. As explained in \S\ref{Poisson in terms of quasi-symplectic}, the holomorphic Poisson structure $\Pi^B_{\pm}$ is then given by \eqref{eq:side poisson via dirac}:
\[ \Gr(\Pi_{\pm}^B)=\Gr(\Omega^B_{\pm}) -\delta^!{\A_{\pm}}, \quad \delta^!{\A_\pm}:=\ttt^!{\A_{\pm}} - \sss^!{\A_{\pm}}. \]
Hence, by passing to the matched pairs of the holomorphic Dirac structures above, we obtain
	\[
	 L_{\Pi_{\pm}^B} =  \Gr(\Omega^B_{\pm})  -\delta^!L_{\A_{\pm}}, \quad \delta^!L_{\A_\pm}:=\ttt^!L_{\A_{\pm}} - \sss^!L_{\A_{\pm}}.
	\]
	By \eqref{pmrel}, we have that $\delta^!L_{\A_+} = e^{-i\delta^*(\omega_++\omega_-)}\delta^!L_{\A_{-}}$. Thus 
	\begin{align*}
		 L_{\Pi_{+}^B} &= \Gr (\Omega_{+}^B  ) - \delta^!L_{\A_+} \\
		 &= \Gr (\Omega_{+}^B - \Omega _{-}^B) +  \Gr(\Omega_{-}^B) + \Gr(i\delta^*(\omega_{+}+\omega_{-})) -\delta^!L_{\A_{-}} \\
		 & =  \Gr(\Omega_{+}^B - \Omega _{-}^B + i\delta^*(\omega_{+}+\omega_{-}) ) + \Gr(\Omega_-^B) - \delta^!L_{\A_{-}}\\
		& = \Gr(\Omega_{+}^B - \Omega _{-}^B + i\delta^*(\omega_{+}+\omega_{-}) ) + L_{\Pi_{-}^B}.
	\end{align*}
	As a result, $L_{\Pi_{+}^B} = e^{2i F} L_{\Pi_{-}^B}$, where 	
 \begin{equation}
     F =  \frac{1}{2i} [\Omega_{+}^B - \Omega _{-}^B + i\delta^*(\omega_{+}+\omega_{-}) ]. \label{eq:F in gauge equivalence}
 \end{equation}	
	To verify that $F$ integrates $-\pi_B$, note that, by \autoref{cor: the underlying re and im} and \eqref{eq:re and im part dirac via q symplectic}, if we let $F_B$ be the unique multiplicative symplectic integration of $-\pi_B$ (\cite[Thm. 5.1]{MR2068969}), we have that 
	\begin{align*}\begin{aligned}
	    \Re(\Omega_{+}^B)=\Re(\Omega_{-}^B), \qquad
		\Im \Omega_{+}^B = -\delta^*\omega_{+} + F_B ,
		\qquad \Im \Omega^B_{-} = \delta^*\omega_{-} - F_B.
	\end{aligned} 
	\end{align*}
	It thus follows from \eqref{eq:F in gauge equivalence} that $F=F_B$. \end{proof}
 To interpret the gauge equivalences of \autoref{prop: deg symplectic GK} in a purely holomorphic fashion, we need to integrate the Poisson groupoids $G_{\A_\pm} $, $G_{\B_\pm} $ to symplectic double groupoids and hence obtain holomorphic symplectic Morita equivalences that relate them. 
\subsection{The Lu-Weinstein construction of symplectic double groupoids}\label{subsubsection: Lu-Weinstein}
As a first step to explain the global holomorphic meaning of \autoref{prop: deg symplectic GK}, we establish the existence of holomorphic symplectic double groupoids $(D_{\pm}, G_{\A_{\pm}}, G_{\B_{\pm}}, X_{\pm})$ integrating $(\E_\pm,\A_\pm,\B_\pm)$ under minimal assumptions. We have that, given $(G_{\A_{\pm}},\Omega_\pm^A)\rightrightarrows X_\pm$ and $(G_{\B_{\pm}},\Omega_\pm^B)\rightrightarrows X_\pm$ quasi-symplectic groupoids as in \S\ref{subsec:poi gro gk man}, we may explicitly construct (symplectic) double groupoids $D_\pm$ that we will call of {\em Lu-Weinstein type}: 
\begin{align}  {D_\pm} =\{(x,v,u,y) \in (G_{{\A_\pm}} ){}_{\sss}\times_\ttt G_{{\B_\pm}}\times (G_{{\B_\pm}}) {}_\sss\times_\ttt G_{\A_\pm}\ :\  \, \ttt(x)=\ttt(u),\, \sss(v)=\sss(y) \};\label{eq:luweidougro} 
\end{align} 
with the two multiplications defined by 
\begin{align*}  {\mm}_{{B}} ((x,v,u,y) ,(x',w,v,y') )=( xx',w,u, yy'),\qquad {\mm}_{{A}} ((x,v,u,y) ,(y,v',u',z) )=(x,vv',uu',z);
\end{align*}
and the two inversion maps are
\begin{align*}
	\ii_A(x,v,u,y) = (x^{-1},u,v,y^{-1}),\qquad \ii_B(x,v,u,y) = (y,v^{-1},u^{-1}, x).
\end{align*}
Recall that the {\em pair groupoid over a set $S$}, denoted by $P(S)$, is $S\times S \rightrightarrows S$ with the multiplication defined by the formula $\mm((x,y),(y,z))=(x,z)$. Using this concept, $D_\pm$ can be viewed as the fiber product of the pair groupoid $G_{\A_\pm} \times G_{\A_\pm} \rightrightarrows G_{\A_\pm}$ and the product $G_{\B_\pm} \times G_{\B_\pm} \rightrightarrows X_\pm \times X_\pm$ over the pair groupoid over $P(X_\pm^2) \rightrightarrows X_\pm^2$ as in the following diagram.
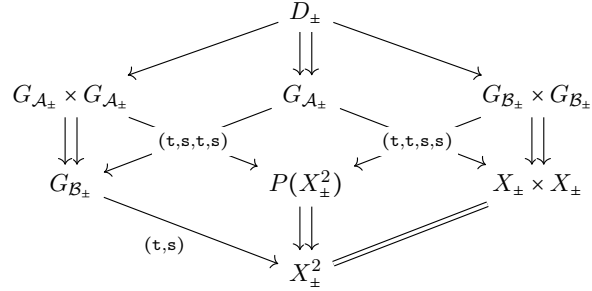
\[ \adjustbox{scale=0.85,center}{\begin{tikzcd}[arrows={crossing over}]
	&  & D_\pm \arrow[lld] \arrow[rrd] \arrow[d, shift right] \arrow[d, shift left] &  &                                                                                                   \\
	G_{\A_\pm}\times G_{\A_\pm}  \arrow[d, shift left] \arrow[d, shift right] &  & G_{\A_\pm} \arrow[lld] \arrow[rrd]                                            &  & G_{\B_\pm}\times G_{\B_\pm} \arrow[lld, "{(\ttt,\ttt,\sss,\sss)}" description] \arrow[d, shift right] \arrow[d, shift left] \\
	G_{\B_\pm} \arrow[rrd, "{(\ttt,\sss)}"']                                                                       &  & P(X_\pm^2)\arrow[llu,<-, "{(\ttt,\sss,\ttt,\sss)}" description] \arrow[d, shift right] \arrow[d, shift left]                                       &  & X_\pm\times X_\pm \arrow[lld, no head, equal]                                                        \\
	&  & X_\pm^2                                                                                    &  &                                                                                                  
\end{tikzcd}} \]
The transversality of the $\A$ and $\B$ foliations implies that these morphisms are transverse and hence $D_\pm \rightrightarrows G_{
A_\pm}$ is a Lie groupoid. We can exchange the roles of $G_{
A_\pm}$ and $G_{
B_\pm}$ to ensure $D_\pm \rightrightarrows G_{
B_\pm}$ is also a Lie groupoid. Also, $D_\pm$ inherits compatible complex structures $I_\pm$ from $G_{
A_\pm}$ and $G_{
B_\pm}$.
\begin{proposition}\label{prop:2form luwei} Let $(G_{\A_{\pm}},\Omega_\pm^A)$ and $(G_{\B_{\pm}},\Omega_\pm^B)$ be quasi-symplectic groupoids over $X_\pm$
    as in \S\ref{subsec:poi gro gk man}. Then the Lu-Weinstein double groupoids $D_\pm$ in \eqref{eq:luweidougro}, endowed with the 2-forms $\Omega_\pm=\delta^*_B\Omega_\pm^B-\delta^*_A\Omega_\pm^A$,  are holomorphic symplectic double groupoids integrating the Manin triples $(\E_\pm,\mathcal{A}_\pm,\mathcal{B}_\pm)$.   
\end{proposition}
\begin{proof} The 2-forms $\Omega_\pm$ are automatically of type $(2,0)_\pm$ and they are closed. As a consequence, they are holomorphic. The fact that they integrate the desired Manin triples $(\E_\pm,\A_\pm,\B_\pm)$ (and hence the Poisson structures on the side groupoids) is established in \cite[Thm. 4.10]{Dan1}. \end{proof} 
\begin{remark}\label{rem:sign symplectic form doubles} In the situation of \autoref{prop:2form luwei}, let $Lie(D)^A$, $Lie(D)^B$ be the respective Lie algebroids of $D\rightrightarrows G_A$ and $D\rightrightarrows G_B$ with anchors $\rho_A:Lie(D)^A \rightarrow TG_A $ and $\rho_B:Lie(D)^B \rightarrow TG_B $. Then we have that $(\rho_A,\mu_{\Omega_\pm})(Lie(D)^A)=\Gr(-\Pi^A_\pm)$ and $(\rho_B,\mu_{\Omega_\pm})(Lie(D)^B)=\Gr(\Pi^B_\pm)$, where $\Gr(\Pi^A_\pm)$ is defined via \eqref{eq:side poisson via dirac}. This implies that, when we construct the Poisson structures on the side groupoids using \eqref{eq:side poisson via dirac}, we have to change the sign on one of them in order to view them as induced by a symplectic double groupoid. \end{remark}
\begin{remark}\label{rem:genluwei} Note that the Lu-Weinstein construction \eqref{eq:luweidougro} can be applied to any pair of Lie groupoids over the same base whose singular orbit foliations are transverse; the output is then a double Lie groupoid. For the original version of this construction see \cite{MR1054741}; its adaptation to more general Manin triples appeared in \cite{Dan1}. 
\end{remark}

\begin{remark}\label{minlw} The Lu-Weinstein integration of a Manin triple is minimal, in the sense that any other integration with the same side groupoids maps to the Lu-Weinstein integration by means of a double Lie groupoid morphism which is a local diffeomorphism. Such integrations can then be equipped with the pullback of the symplectic form of \autoref{prop:2form luwei}.
\end{remark}

\subsection{Gauge transformations and Morita equivalences: the Bursztyn shuffle}\label{subsec: gauge and morita, symplectic type} In this section we recall the general principle according to which a gauge equivalence between holomorphic Poisson structures integrates to a symplectic Morita equivalence between symplectic groupoids, thus setting the stage to integrate the gauge equivalences of \autoref{prop: deg symplectic GK}.    This principle was the main tool used in~\cite{MR4612595,MR4466669} to integrate generalized complex structures and generalized K\"ahler structures of symplectic type, and it first appeared for real Poisson structures in~\cite{MR1973074}.  We refer to it as the \emph{Bursztyn shuffle}, as it allows one to explicitly deform one symplectic groupoid into the other and into the symplectic bimodule between them. 

Let $I_\pm$ be a pair of complex structures on $M$, and let $\Pi_\pm$ be $I_\pm$-holomorphic Poisson structures whose matched pairs are  gauge equivalent, i.e. there is a closed 2-form $B$ on $M$ such that 
\begin{equation}\label{gapoi}
L_{\Pi_+} = e^B L_{\Pi_-}.
\end{equation}
In this situation, we have the following result, which is a modification of
\cite[Proposition 6.3, 6.4]{MR4612595}.
 \begin{proposition}(Bursztyn shuffle)\label{thm:BG} Given gauge-equivalent holomorphic Poisson structures as in~\eqref{gapoi}, as well as a holomorphic symplectic groupoid $G_- = (G,\Omega_-)$ integrating $\Pi_-$, then the complex 2-forms 
 \begin{align}
 \Omega_+ = \Omega_- + \ttt^*B-\sss^*B,\qquad 
 \Omega_Z = \Omega_- + \ttt^*B
 \end{align}
endow the underlying real Lie groupoid $G$ with new complex and holomorphic symplectic structures such that $G_+ = (G,\Omega_+)$ integrates $\Pi_+$ and $Z=(G,\Omega_Z)$ is a holomorphic symplectic Morita equivalence,
	\[
	\begin{tikzcd}
		G_+=(G,\Omega_+)  \arrow[d,"\sss", shift left] \arrow[d, "\ttt"', shift right] & Z=(G,\Omega_Z) \arrow[ld,"\ttt"'] \arrow[rd,"\sss"] & G_-=(G,\Omega_-) \arrow[d,"\sss", shift left] \arrow[d, "\ttt"',shift right]                         \\
		(M,I_+,\Pi_+)                                 & \mbox{}                 & (M,I_-,\Pi_-)
	\end{tikzcd}
	\]
 with actions given by the underlying multiplication on $G$.
\end{proposition}

We now make use of this procedure to integrate the Poisson groupoid structures described in~\autoref{prop: deg symplectic GK}.  For this purpose, we will need to deform our symplectic double groupoids in two independent directions to obtain a $(1,1)$-morphism of double symplectic groupoids.

\subsection{The two-dimensional Bursztyn shuffle}\label{section: 2dbs}
In the following result we combine \autoref{thm:BG} with \autoref{prop: deg symplectic GK} to produce a $(1,1)$-morphism that integrates \eqref{manindiag}. By assuming that we are given a holomorphic symplectic double groupoid $D_-$ with sides $G_{\A_-}$ and $G_{\B_-}$ as in \S\ref{subsec:poi gro gk man}, we can successively obtain all the holomorphic symplectic spaces in \eqref{eq:(1,1) morphism via 2d shuffle} below via deformations of $D_-$ as in \autoref{thm:BG}.
\begin{proposition}[2D Bursztyn shuffle]\label{2d Bursztyn shuffle}
	Let $(D_-,G_{\A_-},G_{\B_-},X_-)$ be a holomorphic symplectic groupoid that integrates the Manin triple $(\E_-,\A_-,\B_-)$ underlying a generalized K\"ahler structure. Then there is a holomorphic symplectic $(1,1)$-morphism: 
\begin{equation}
\begin{tikzcd}
D_-'                   & \mathcal{W} \arrow[l] \arrow[r]                     & D_+                             \\
\mathcal{Z}' \arrow[u] & \mathcal{C} \arrow[u] \arrow[l] \arrow[r] \arrow[d] & \mathcal{Z} \arrow[u] \arrow[d] \\
D_+' \arrow[u]         & \mathcal{W}' \arrow[r] \arrow[l]                    & D_-                    
\end{tikzcd}\label{eq:(1,1) morphism via 2d shuffle}  
\end{equation} 
in which the holomorphic symplectic manifolds above are: 
 \begin{align*}
  &D_-'=(D,\Omega_-+ 2i( \delta_A^*F_A + \delta_B^*F_B )) \quad &\mathcal{W}=(D,\Omega_- + 2i(\ttt_A^*F_A+\delta_B^*F_B)) \quad &D_+=(D,\Omega_- +2i\delta_B^*F_B)   \\
 &\mathcal{Z}'=(D,\Omega_-+ 2i(\delta_A^*F_A+\ttt_B^*F_B)) \quad &\mathcal{C}=(D,\Omega_-+2i(\ttt_A^*F_A+\ttt_B^*F_B)) \quad &\mathcal{Z}=(D,\Omega_-+2i\ttt_B^*F_B) \\
 &D_+'=(D,\Omega_-+ 2i\delta_A^*F_A)
 \quad &\mathcal{W}'=(D,\Omega_-+ 2i\ttt_A^*F_A) \qquad & D_-=(D,\Omega_-) \end{align*}
  where $D$ is the underlying smooth double groupoid of $D_-$. The $(1,1)$-morphism structure is determined as in \autoref{ex:triv(1,1)mor}. 
  \end{proposition}

\begin{remark}\label{rem:sign double gpd} Our convention for the sign in the symplectic form on the holomorphic double Lie groupoid $D_-$ is as follows. Differentiating the symplectic structure $\Omega_-$ should induce $\Pi_-^B$ on $G_B$ and $-\Pi_-^A$ on $G_A$, see \autoref{rem:sign symplectic form doubles}. \end{remark}
\begin{proof}[Proof of \autoref{2d Bursztyn shuffle}] Since $\Pi_{-}^B$ and $\Pi_{+}^B$ are gauge equivalent via $2iF_B$, i.e., $L_{\Pi_+^B} = e^{2iF_B}L_{\Pi_{-}^B}$, we may start from the symplectic groupoid $(D_-, \Omega_-) \rightrightarrows (G_{\B_{-}}, \Pi_-^B)$ and build the symplectic groupoid $(D, \Omega_+) \rightrightarrows (G_{\B_+}, \Pi_+^B)$ whose underlying smooth groupoid is the same as $D$ and with symplectic form $\Omega_+ = \Omega_- - 2i(\sss_B^*F_B-\ttt_B^*F_B)$, by using the Bursztyn shuffle in the vertical direction. Then $\mathcal{Z}=(D, \Omega_+  + 2i t_B^*F_B)$ serves as a  bimodule for a Morita equivalence from $D_-\rightrightarrows G_{\B_-}$ to $D_+\rightrightarrows G_{\B_+}$, see Theorem \ref{thm:BG}.
We now verify that $D_-, \mathcal{Z}, D_+$ are related by an invertible $(0,1)$-morphism. 	Note that the complex structure on $(D,\Omega_+)$ is given by $\Im(\Omega_+)^{-1}\Re(\Omega_+)$.
	To check that $D_+=(D,\Omega_+)$ is also a symplectic groupoid over $G_{\A_+}$ we note that since $\ttt_B, \sss_B$ are both groupoid morphisms from $D \rightrightarrows G_{A}$ to $G_{B} \rightrightarrows M$ and $\Omega$ is multiplicative for $D \rightrightarrows G_{\A}$ and $F_B$ is multiplicative for $G_{B}\rightrightarrows M$, $\Omega_+$ is also multiplicative for $D \rightrightarrows G_{A}$. By the formula of the complex structure of $D_+$, this also shows that it is multiplicative for both sides and hence $D_+$ is a holomorphic symplectic double groupoid. To see that the induced side Poisson groupoids of $D_+$ are the same as $G_{\A_+}$ and $G_{\B_+}$ note that the underlying real parts of $\Omega_+$ and $\Omega_-$ coincide and so the underlying real parts of $\Pi_+^{A}$ and $\Pi_+^B$ are the same as those of $\Pi_-^A$ and $\Pi_-^B$ respectively. So the result follows from the fact that a holomorphic Poisson structure is determined by its real part and the complex structure. Proceed similarly to show that $\Omega_+  + 2i t_B^*F_B$ determines a holomorphic symplectic groupoid structure $\mathcal{Z} \rightrightarrows Z$ on the underlying smooth groupoid $D\rightrightarrows G_A$. 

 Having established the rightmost side of the diagram~\eqref{eq:(1,1) morphism via 2d shuffle},  we now perform horizontal Bursztyn shuffle to obtain the top side.  According to \autoref{prop: deg symplectic GK} and \autoref{rem:sign double gpd}, the Poisson structure induced by $\Omega_+$ on $G_{{A}}$ is $-\Pi_+^A$ which satisfies $  L_{-\overline{\Pi_-^A}}=e^{2iF_A}L_{-\Pi_+^A}$. Therefore, in the same manner as above, we can obtain a horizontal Morita equivalence of holomorphic symplectic double groupoids. Starting from $D_+=(D,\Omega_+) = (D, \Omega_- + 2i \delta_B^*F_B)$, we can build $D_-' = (D, \Omega_+ +2i\delta_A^*F_A) = (D,\Omega_-+ 2i(\delta_A^*F_A+\delta_B^*F_B))$,
which is a holomorphic symplectic double groupoid integrating $(\overline{\E_-},\overline{\A_-}, \overline{\B_-})$, and 
\[
 \mathcal{W}=(D, \Omega_+ + 2i\ttt^*_AF_A) = (D, \Omega_- + 2i(\delta_B^*F_B+\ttt_A^*F_A))
\]
defines a new holomorphic symplectic groupoid structure $\mathcal{W}\rightrightarrows W$ on the underlying Lie groupoid $D \rightrightarrows G_B$ which constitutes the bimodule for a Morita equivalence from $D_+$ to $D_-'$. 

Repeating the argument to obtain the bottom side of~\eqref{eq:(1,1) morphism via 2d shuffle},  we use $ L_{-\overline{\Pi_+^A}}=e^{2iF_A}L_{-\Pi_-^A}$ to obtain a horizontal Morita equivalence from $D_-$ to $D_+'$ with bimodule $\mathcal{W}'\rightrightarrows W'$, and then, moving vertically, a Morita equivalence from $D_+'$ to $D_-'$ with bimodule $\mathcal{Z'}\rightrightarrows Z'$ determined by $L_{\overline{\Pi_-^B}} = e^{2iF_B}L_{\overline{\Pi_{+}^B}}$. 

Finally, since the symplectic structures on $\mathcal{W}'$ and $\mathcal{W}$ differ by $2i\delta_B^*F_B$, their underlying Poisson structures are gauge equivalent via $2iF_B$ and, similarly, the underlying Poisson structures of $\mathcal{Z}'$ and $\mathcal{Z}$ are gauge equivalent via $2iF_A$. Therefore we see that $\mathcal{C}$ may be obtained as a symplectic bimodule either vertically from $\mathcal{W}'$ or horizontally from $\mathcal{Z}$ by Bursztyn shuffle, showing it is a double bimodule, as required.
\end{proof}
 \begin{remark}\label{rem:symplectic type A and B leaves} If $D_-$ is of Lu-Weinstein type then we have the following. If $\mathcal{O}_B$ is the $\B $-leaf that passes through $p\in X_\pm$, it inherits the structure of a Generalized K\"ahler structure of symplectic type \cite[Remark 3.10]{MR3232003}. On the other hand, if we take the $D_\pm$-orbit of $\uu(p)\in G_{\A_\pm}$, that is a union of symplectic leaves of $G_{\A_\pm}$, but it is also a groupoid given simply by the restriction $G_{\A_\pm}|_{\mathcal{O}_B} \rightrightarrows \mathcal{O}_B$. Similarly, $Z|_{\mathcal{O}_B}$ is a $\mathcal{Z}$-orbit and so it is symplectic. Then we get that, automatically, the restricted Morita equivalence
    \[ \begin{tikzcd}[scale cd=.85]
G_{\mathcal{A}_+}|_{\mathcal{O}_B} \arrow[d, shift right] \arrow[d, shift left]  &  & Z|_{\mathcal{O}_B} \arrow[lld] \arrow[rrd] &  & G_{\mathcal{A}_-}|_{\mathcal{O}_B} \arrow[d, shift right] \arrow[d, shift left]  \\
\mathcal{O}_B                                                                                                 &  &                                            &  & \mathcal{O}_B                                                                                                 
\end{tikzcd}\]
is a symplectic Morita equivalence and the analogous property holds for $\A$-leaves.
\end{remark}
 \subsection{Self-adjoint $(1,1)$--morphisms}\label{subsec:real str}
 As we saw in~\autoref{Inf descprtion of GK}, the four Manin triples underlying a generalized K\"ahler structure consist of two complex conjugate pairs.  We now explain how this symmetry persists at the global level, giving rise to a real structure on the corresponding $(1,1)$-morphism.
 \begin{definition}
     A holomorphic symplectic $(1,1)$-morphism $\mathcal{C}$ as in \eqref{eq:(1,1) morphism} is {\em self-adjoint} if there is an isomorphism of holomorphic symplectic $(1,1)$-morphisms 
     \[ \tau:\mathcal{C}\rightarrow \overline{\mathcal{C}}^\top, \qquad \overline{\tau}^\top \tau=\text{id} \]
     where $\overline{\mathcal{C}}^\top$ is the complex conjugate of the double inverse of $\mathcal{C}$, with respect to both vertical and horizontal compositions 
     as in  Fig.~\ref{fig:real structure}. 
 \end{definition}
 \begin{figure}[H]
    \centering
    \begin{tikzpicture}[scale=0.75,every node/.style={transform shape}]
\pgfdeclarelayer{background}
\pgfdeclarelayer{foreground}
\pgfsetlayers{background,main,foreground}
 \coordinate (a) at (-1,-1);
 \coordinate (b) at (1,-1);
 \coordinate (c) at (1,1);
 \coordinate (d) at (-1,1);
 \coordinate (e) at (0,0);
\coordinate (f) at (0,1);
\coordinate (g) at (0,-1);
\coordinate (h) at (1,0);
\coordinate (i) at (-1,0);
\coordinate (j) at (-1+4.5,-1+0.5);
 \coordinate (k) at (1+4.5,-1+.5);
 \coordinate (l) at (1+4.5,1+.5);
 \coordinate (m) at (-1+4.5,1.5);
 \coordinate (n) at (4.5,0.5);
\coordinate (o) at (4.5,1.5);
\coordinate (p) at (4.5,-0.5);
\coordinate (q) at (5.5,0.5);
\coordinate (r) at (3.5,0.5);
 
 \begin{pgfonlayer}{background}
      \draw[fill=black!10] (a)--(b)--(c)--(d)--cycle;
\draw[postaction={on each segment={mid arrow=black}}] (b)--(c);
\draw[postaction={on each segment={mid arrow=black}}] (b)--(a); 
\draw[postaction={on each segment={mid arrow=black}}] (a)--(d);
\draw[postaction={on each segment={mid arrow=black}}] (c)--(d);
\draw[fill=black!10] (j)--(k)--(l)--(m)--cycle;
      \draw[fill=black!3,opacity=0.5,ultra thin] (b)--(k)--(l)--(c)--cycle;
      \draw[fill=black!3,opacity=0.5] (c)--(l)--(m)--(d)--cycle;
      \draw[opacity=0.5] (a)--(b)--(k)--(j)--cycle;
\draw[postaction={on each segment={mid arrow=black}}] (k)--(l);
\draw[postaction={on each segment={mid arrow=black}}] (l)--(m); 
\draw[postaction={on each segment={mid arrow=black}}] (j)--(m);
\draw[postaction={on each segment={mid arrow=black}}] (k)--(j);
\draw[draw=gray!5,postaction={on each segment={mid arrow=gray}},opacity=0.4] (c)--(l);
\draw[draw=gray!5,opacity=0.4] (b)--(k);
\draw[draw=gray!5,postaction={on each segment={mid arrow=gray}},opacity=0.4] (a)--(j);
\draw[draw=gray!5,opacity=0.4] (d)--(m);
            \end{pgfonlayer}
 \node [below left] at (a)
                {${D}_{+}'$};
\node [circle, fill=black, inner sep=1pt] at (a){};
\node [below right] at (b)
                {$D_{-}$};
\node [circle, fill=black, inner sep=1pt] at (b){};                
\node [above right] at (c)
                {$D_{+}$};
\node [circle, fill=black, inner sep=1pt] at (c){};
\node [above left] at (d)
                {${D}_{-}'$};
\node [circle, fill=black, inner sep=1pt] at (d){};
\node  at (e)
                {$\mathcal{C}$};
\node [above] at (f)
                {$\mathcal{W}$};
\node[below] at (g)
                {${\mathcal{W}'}$};
\node[right] at (h)
                {$\mathcal{Z}$};
\node[left] at (i)
                {${\mathcal{Z}'}$};
                 \node [below left] at (j)
                {$\overline{D}_{+}$};
\node [circle, fill=black, inner sep=1pt] at (j){};
\node [below right] at (k)
                {$\overline{D_-'}$};
\node [circle, fill=black, inner sep=1pt] at (k){};                
\node [above right] at (l)
                {$\overline{D_+'}$};
\node [circle, fill=black, inner sep=1pt] at (l){};
\node [above left] at (m)
                {$\overline{D}_-$};
\node [circle, fill=black, inner sep=1pt] at (m){};
\node  at (n)
                {$\overline{\mathcal{C}}$};
\node [above] at (o)
                {$(-\overline{\mathcal{W}'})$};
\node[below] at (p)
                {$(-\overline{\mathcal{W}})$};
\node[right] at (q)
                {$(-\overline{\mathcal{Z}}')$};
\node[left] at (r)
                {$(-\overline{\mathcal{Z}})$};
\end{tikzpicture} 
    \caption{Self-adjoint $(1,1)$-morphism}
    \label{fig:real structure}
\end{figure} 
\begin{remark}
Note that such a self-adjoint $(1,1)$-morphism induces a real structure $\tau_\mathcal{C}$ on the double bimodule, which satisfies $\tau_\mathcal{C}^*\Omega_\mathcal{C} = \overline{\Omega}_\mathcal{C}$.
\end{remark}

We now verify that the $(1,1)$-morphism constructed via the Bursztyn shuffle \eqref{eq:(1,1) morphism via 2d shuffle} is self-adjoint, under a mild assumption. 
From the proof of \autoref{2d Bursztyn shuffle}, the two sides of $D_{-}'$ are $\overline{G_{\A_-}}$ and $\overline{G_{\B_-}}$. As a result, $D_-'$ coincides with the complex conjugate of $\overline{D_-}$ (equivalently, $D_+'$ is conjugate to $\overline{D_+}$), as long as
\begin{align}\label{eq: imomega}
\overline{\Omega}_- - \Omega_- =2i(\delta_A^*F_A+\delta_B^* F_B).
\end{align}
 In fact, this equation holds for the Lu-Weinstein integration, see \autoref{prop:2form luwei}, and so we can assume that it holds for $D_-$. Equation \eqref{eq: imomega} also implies that any pair of holomorphic symplectic manifolds in \eqref{eq:(1,1) morphism via 2d shuffle} related by reflection through the center are isomorphic via double groupoid inversion and complex conjugation. In the following, a minus sign before a space indicates taking the negative of its symplectic form. 
\begin{proposition}\label{prop:real str} If \eqref{eq: imomega} holds, then we have symplectomorphisms
\begin{align} \ii_A\circ\ii_B: \overline{\mathcal{Z}} \rightarrow \mathcal{Z}', \qquad \ii_A\circ\ii_B: \overline{\mathcal{W}} \rightarrow \mathcal{W}', \qquad \ii_B:\mathcal{Z}\rightarrow (-\overline{\mathcal{Z}}'), \qquad\ii_A:\mathcal{W}\rightarrow (-\overline{\mathcal{W}}'); \label{eq:symplecto 1,1 morphism} 
\end{align}
 where $\ii_A$ and $\ii_B$ are the inverses of $D\rightrightarrows G_A$ and $D\rightrightarrows G_B$, respectively. Moreover, $ \ii_A\circ\ii_B=\ii_B\circ\ii_A$ defines a symplectic real structure on $\mathcal{C} = (D, \Omega_{\mathcal{C}})$: it is a symplectomorphism as below: 
\[\begin{tikzcd} 
\mathcal{C} =(D, \Omega_{\mathcal{C}})\arrow[rr, "\ii_A\circ\ii_B"] &  & \overline{\mathcal{C}}=(D, \overline{\Omega_{\mathcal{C}}}).
\end{tikzcd}\]
\end{proposition}
\begin{proof} We only analyse the first symplectomorphism in \eqref{eq:symplecto 1,1 morphism} as the others are similar. Using \autoref{2d Bursztyn shuffle}, we get that the complex conjugate of $\mathcal{Z}=(D, \Omega_- + 2i\ttt^*_BF_B)$ can be identified with $\mathcal{Z}'=(D, \Omega_- + 2i(\delta_A^*F_A+\ttt_B^*F_B))$ via $\ii_B\circ \ii_A$. Since 
\[
\ii_A^*\ii_B^*(\overline{\Omega_-} -2i\ttt_B^*F_B)=\ii_A^*(-\overline{\Omega_-} -2i\sss_B^*F_B) = \overline{\Omega_-} +2i\sss_B^*F_B,
\]
by \eqref{eq: imomega}, we have 
\[
\ii_A^*\ii_B^*(\overline{\Omega_-} -2i\ttt_B^*F_B) = \Omega_- +2i(\delta_A^*F_A+\ttt^*_BF_B).
\] 
	Finally, it also follows from \eqref{eq: imomega} that 
$	\ii_B^*\ii_A^*\Omega_\mathcal{C}=\ii_B^*\ii_A^*(\Omega_- + 2i(\ttt_A^*F_A+\ttt_B^*F_B)) = \overline{\Omega_\mathcal{C}}$. 
\end{proof}
We may combine the symplectomorphisms $\ii_B:\mathcal{Z}\rightarrow (-\overline{\mathcal{Z}}')$ and $\ii_A:\mathcal{W}\rightarrow (-\overline{\mathcal{W}}')$ with the identity maps for the double groupoids and the real structure $\ii_A\circ\ii_B:\mathcal{C}\rightarrow \overline{\mathcal{C}}$ to define an isomorphism between holomorphic symplectic $(1,1)$-morphisms as in Fig.~\ref{fig:real structure}, where the $(1,1)$-morphism structure on the right side is given by the double inverse of the $(1,1)$-morphism on the left, with respect to both vertical and horizontal compositions.


 
\subsection{Lagrangian bisections and the real symplectic core}\label{subsec: lag branes}

When a symplectic Morita equivalence is produced by the Bursztyn shuffle as in Proposition~\ref{thm:BG}, the bimodule $Z=(G,\Omega_Z)$ inherits a natural bisection given by the identity bisection of $G$.  This bisection is not necessarily holomorphic in the new complex structure determined by $\Omega_Z$, and the pullback of $\Omega_Z=\Omega_- + \ttt^*B$ to the bisection simply recovers the original gauge transformation $B$.  For generalized K\"ahler geometry, our gauge transformations are all \emph{pure imaginary}, by Proposition~\ref{prop: deg symplectic GK}.  In such a case, the bisection is Lagrangian with respect to the real part of the holomorphic symplectic form, and the pullback of $\Omega_Z$ recovers the imaginary gauge transformation.  The presence of such real Lagrangian bisections is essential in generalized  K\"ahler geometry, as they encode the generalized K\"ahler metric, see \autoref{thm: reconstruction}. 

Accordingly, in \autoref{2d Bursztyn shuffle}, the embedding of the subgroupoid $G_B \rightrightarrows M$ into $\mathcal{Z}\rightrightarrows Z$ is a multiplicative bisection which is also Lagrangian with respect to the real part of the symplectic form on $\mathcal{Z}$. In the same manner, there is a multiplicative bisection $\Lambda_W$ in $\mathcal{W}$, and also bisections $\Lambda_{\text{Hor}}$, $\Lambda_{\text{Ver}}$ in $\mathcal{C}$ with respect to the respective Morita equivalences relating $\mathcal{Z}$ with $ \overline{\mathcal{Z}}$ and $\mathcal{W} $ with $\overline{\mathcal{W}}$; all of these bisections are Lagrangian with respect to the real part of the respective background symplectic forms. 

Let $(X,\Omega)$ be a holomorphic symplectic manifold. To simplify the description, we shall say that a submanifold $\mathcal{L}\subset X$ is {\em Re-Lagrangian (respectively, Im-Lagrangian)}, if it is Lagrangian with respect to $\Re(\Omega) $ (respectively, $\Im(\Omega)$); we say that $\mathcal{L}$ is {\em Re-symplectic (respectively, Im-symplectic)} if it is symplectic with respect to $\Re(\Omega)$ (respectively, $\Im(\Omega)$).

\begin{definition}\label{def:real triv} Suppose that a holomorphic symplectic $(1,1)$-morphism $\mathcal{C}   $ as in \eqref{eq:(1,1) morphism} is self-adjoint via the real structure $\tau$. Given a trivialization by Re-Lagrangian bisections of $\mathcal{C}$, we say that it is a {\em real trivialization} if, after taking real parts, we obtain a commutative diagram:
\begin{equation}
 \begin{tikzcd}
{\mathcal{C}} \arrow[rr, "\tau"] \arrow[d, "\Lambda"'] &  & {\mathcal{C}^\top} \arrow[d, "\Lambda"] \\
{E(D_{0,0})} \arrow[rr, "{\tau_{{0,0}}}"']                 &  & {E(D_{0,0})^\top}                                 
\end{tikzcd} \qquad \tau_{0,0}=\begin{bmatrix} \text{id} & \ii_A & \text{id} \\
\ii_B &\ii_A\circ \ii_B & \ii_B \\
\text{id}& \ii_A & \text{id}  \end{bmatrix}  \label{eq:real triv}   
\end{equation}
where $\Lambda$ is the isomorphism to the trivial $(1,1)$-morphism $E(D_{0,0})$ over $D_{0,0}$ determined by the given bisections and $\tau_{{0,0}}$ is the involution constructed from the vertical and horizontal inverse maps as shown above, see \eqref{fig:morphism between 2-cells}.
\end{definition}

An interesting consequence of the fact that the symplectic bimodule $\mathcal{C}$ of a self-adjoint $(1,1)$-morphism has an anti-holomorphic symplectic involution is that its fixed point locus $\mathcal{S}$, if nonempty, defines a distinguished real symplectic submanifold of $\mathcal{C}$.  In the remainder of this section, we show the remarkable fact that, in some cases, the choice of a Lagrangian submanifold of $\mathcal{S}$ gives rise to a real trivialization of the $(1,1)$-morphism as defined above.    

\begin{proposition}\label{S is LS}
	If the fixed point locus $\mathcal{S} \subset \mathcal{C}$ of the symplectic real structure $\tau$ on a self-adjoint $(1,1)$-morphism is non-empty, then it is $\Re$-symplectic and $\Im$-Lagrangian. We call it the {\em real symplectic core} of the $(1,1)$-morphism \eqref{diagram:(1,1)morphism}.
\end{proposition}
\begin{proof}
	Since $\tau^*\overline{\Omega_\mathcal{C}} = \Omega_\mathcal{C}$, we see that $\tau$ is symplectic for $\Re(\Omega_\mathcal{C})$ and anti-symplectic for $\Im(\Omega_\mathcal{C})$. Thus its fixed point set is $\Re$-symplectic and $\Im$-Lagrangian.  
\end{proof}

\begin{example} In the case of \autoref{2d Bursztyn shuffle}, if the starting symplectic double groupoid $D_-$ is of Lu-Weinstein type, then all the holomorphic symplectic bimodules are smoothly identified with the same manifold of \eqref{eq:luweidougro}. Thus the real symplectic core of \eqref{diagram:(1,1)morphism} (being the fixed point set of $\tau=\ii_A\circ\ii_B$) is  
\begin{align}\label{Fixed locus S}
\mathcal{S} =\{(x,y^{-1}, y, x^{-1}): (x, y) \in G_A \times_{(\sss,\sss)}  G_B, \ttt(x) = \ttt(y)\}
\end{align}
Moreover, since the underlying smooth manifold of $\mathcal{C}$ coincides with $D$ as in \eqref{eq:luweidougro}, the fixed  locus of its real structure can be smoothly identified with the {\em core groupoids $K_{D_\pm }$ in $D_\pm$}. In fact, the map
\begin{align}
    \mathcal{S} \rightarrow K_{D_\pm}, \qquad &(x,y^{-1},y,x^{-1}) \mapsto (x,\uu (\sss(x)),y,\uu(\sss(y)))  \label{eq:ident sympl core},  
\end{align}
 is an isomorphism, see \autoref{def:sym core}. Denote by $\omega_\mathcal{S}$ the real symplectic form on $\mathcal{S}$. As a consequence of our previous discussion, $(\mathcal{S},\omega_\mathcal{S})$ is symplectomorphic to $(K_{D_\pm},2\Re(\Omega_\pm)|_{K_{D_{\pm}}})$ via \eqref{eq:ident sympl core}. With respect to this identification, the natural projections from $\mathcal{S} $ to the base manifolds are identified with the source and target of $K_{D_\pm}$. Recall that the real parts of the symplectic forms on $D_\pm$ coincide due to the self-adjointness of \eqref{diagram:(1,1)morphism}.     
\end{example}
 
  A general self-adjoint $(1,1)$-morphism can be written as in \eqref{diagram:(1,1)morphism} below, using its real structure. So we obtain a pair of smooth groupoids $K_{D_\pm}^\R\rightrightarrows X_\pm^\R$ underlying the core groupoids $K_{D_\pm}$ and their conjugates, equipped with the symplectic forms $\omega_\pm:=2\Re(\Omega_\pm)|_{K_{D_{\pm}}}$.  
\begin{proposition}\label{prop: morita core groupoids} A self-adjoint $(1,1)$-morphism as in \eqref{diagram:(1,1)morphism} with non empty real symplectic core $(\mathcal{S},\omega_\mathcal{S})$ serves as the symplectic bimodule for a pair of commuting symplectic actions of $(K_{D_\pm}^\R,\omega_\pm)$: 
\begin{equation}
  \begin{tikzcd}[scale cd=.85]
(K_{D_-}^\R,\omega_-) \arrow[d, shift right] \arrow[d, shift left]  &  & {(\mathcal{S},\omega_{\mathcal{S}})} \arrow[lld, "p_-"'] \arrow[rrd, "p_+"] &  & (K_{D_+}^\R,\omega_+) \arrow[d, shift right] \arrow[d, shift left] \\
X_-                                                                                             &  &                                                                         &  & X_+;                                                                                             
\end{tikzcd}  \label{eq:symplectic core as a morita} 
\end{equation} 
if the $(1,1)$-morphism is constructed as in \eqref{eq:(1,1) morphism via 2d shuffle}, using a symplectic double groupoid of Lu-Weinstein type, then this is a real symplectic Morita equivalence. 
\end{proposition}
\begin{proof} 
One may define one of the actions $k\cdot s$ as in the following diagram, where $s\in \mathcal{S}$ and $k\in K_{D_-}^\R$:
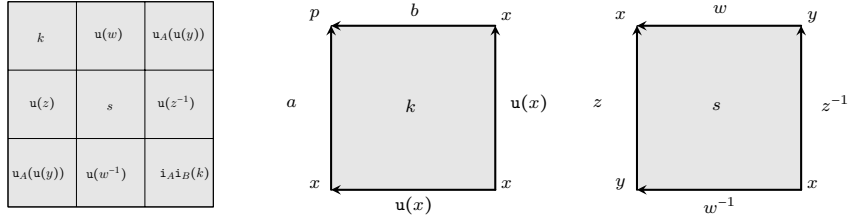
\[ \qquad\qquad \adjustbox{scale=0.65}{\begin{tikzpicture}[scale=0.7,every node/.style={transform shape}]
\clip(-0.5,-0.5) rectangle (6.5,6.5);
  \draw[fill=black!10] (0,0) rectangle (6,6);
  \draw[step=2.0,black,thin] (0, 0) grid (6.0, 6.0);
  \path
    (1.8, 1.3) node[below left] {$\uu_A\uu(y)$}
(3.6, 1.3) node[below left] {$\uu({w^{-1}})$}
    (4.3, 1.3) node[below right] {$\ii_A\ii_B(k)$}
    (0.5, 2.7) node[above right] {$\uu(z)$}
    (3.2, 2.7) node[above left] {$s$}
(5.6, 2.7) node[above left] {$\uu(z^{-1})$}
(1.2, 4.7) node[above left] {$k$}
(3.5, 4.7) node[above left] {$\uu(w)$}
(5.9, 4.7) node[above left] {$\uu_A\uu(y)$}
  ;
\end{tikzpicture}}\adjustbox{scale=0.85}{\begin{tikzpicture}[line cap=round,line join=round,>=stealth,x=1cm,y=1cm,scale=0.4]
		\clip(-3.5,-5.12) rectangle (9.7,5.3);
  \draw[fill=black!10] (0.44,-3.78) rectangle (6.86,2.64);
		\draw [<-,line width=0.9pt] (0.44,2.64) -- (6.86,2.64);
		\draw [<-,line width=0.9pt] (6.86,2.64) -- (6.86,-3.78);
		\draw [->,line width=0.9pt] (6.86,-3.78) -- (0.44,-3.78);
		\draw [->,line width=0.9pt] (0.44,-3.78) -- (0.44,2.64);
		\begin{scriptsize}
			
			\draw[color=black] (-0.2,-3.55) node {$x$};
			
			\draw[color=black] (7.3,-3.55) node {$x$};
			
			\draw[color=black] (7.3,3.0) node {$x$};
			
			\draw[color=black] (-0.2,3.07) node {$p$};
			\draw[color=black] (3.7,3.27) node {$b$};
			\draw[color=black] (8.2,-0.4) node {$\uu(x)$};
			\draw[color=black] (3.68,-4.39) node {$\uu(x)$};
			\draw[color=black] (-1.1,-0.4) node {$a$};
			\draw[color=black] (3.55,-0.5) node {$k$};
		\end{scriptsize}
	\end{tikzpicture}\hspace{1em}\begin{tikzpicture}[line cap=round,line join=round,>=stealth,x=1cm,y=1cm,scale=0.4]
		\clip(-1.3,-5.12) rectangle (11.7,5.3);
  \draw[fill=black!10] (0.44,-3.78) rectangle (6.86,2.64);
		\draw [<-,line width=0.9pt] (0.44,2.64) -- (6.86,2.64);
		\draw [<-,line width=0.9pt] (6.86,2.64) -- (6.86,-3.78);
		\draw [->,line width=0.9pt] (6.86,-3.78) -- (0.44,-3.78);
		\draw [->,line width=0.9pt] (0.44,-3.78) -- (0.44,2.64);
		\begin{scriptsize}
			
			\draw[color=black] (-0.2,-3.55) node {$y$};
			
			\draw[color=black] (7.3,-3.55) node {$x$};
			
			\draw[color=black] (7.3,3.0) node {$y$};
			
			\draw[color=black] (-0.2,3.07) node {$x$};
			\draw[color=black] (3.7,3.27) node {$w$};
			\draw[color=black] (8.2,-0.4) node {$z^{-1}$};
			\draw[color=black] (3.68,-4.39) node {$w^{-1}$};
			\draw[color=black] (-1.1,-0.4) node {$z$};
			\draw[color=black] (3.55,-0.5) node {$s$};
		\end{scriptsize}
	\end{tikzpicture}} \]
 The other action is defined similarly, using the action of $K_{D_+}^\R$ viewed as a subspace of $D_+$ and $\overline{D_+}$ instead of $D_-$ and $\overline{D_-}$.
The last statement follows from using \eqref{eq:ident sympl core} to identify these pair of symplectic actions with the Morita autoequivalence of $(K_{D_+}^\R,\omega_+)\cong(K_{D_-}^\R,\omega_-)$ given by right and left multiplication. \end{proof}
 We now explain how, assuming that $D_-$ is of Lu-Weinstein type, real trivializations of \eqref{diagram:(1,1)morphism} are generated by a single Lagrangian submanifold in its real symplectic core. 
\begin{proposition}\label{thm: key Lagrangian in the central brane} In a real trivialization of a self adjoint $(1,1)$-morphism, the intersection of the Lagrangian bisections in the double bimodule is  a Lagrangian bisection of the real symplectic core. 

Conversely, in a $(1,1)$-morphism which is constructed as in \eqref{eq:(1,1) morphism via 2d shuffle} using a symplectic double groupoid of Lu-Weinstein type, real trivializations are determined by Lagrangian bisections of the real symplectic core. 
\end{proposition}
\begin{remark} The significance of this result is given by \autoref{thm: reconstruction}, where we identify Lagrangian bisections of the symplectic core which determine real trivializations giving rise to generalized K\"ahler metrics. 
\end{remark}
\begin{proof}[Proof of \autoref{thm: key Lagrangian in the central brane}] The first part of the statement is a consequence of the fact that \eqref{eq:real triv} allows us to view the two bisections in the double bimodule as the embeddings of the side groupoids in a double groupoid. So their intersection is their common unit submanifold which is fixed by $\ii_A \circ \ii_B$.

For the second part of the statement, let $\lambda$ be a Lagrangian bisection of \eqref{eq:symplectic core as a morita}. Such a bisection is of the form:
	\[\lambda (M)= \{\lambda(p)=(\lambda_Z(p)^{-1},\lambda_W(q)^{-1},\lambda_W(q),\lambda_Z(p))\ :\ p\in M \}\subset \mathcal{S}. \]  
	The components of $\lambda$, $\lambda_W$ and $\lambda_Z$, are automatically (coisotropic) bisections for the Morita equivalences 
	
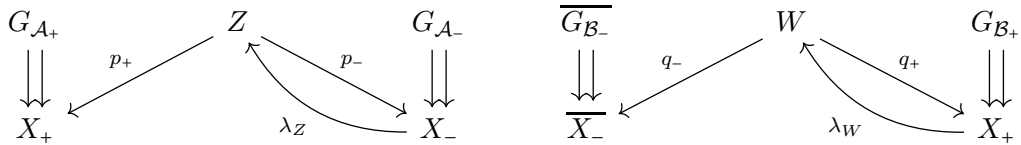
\begin{equation*}
	 \begin{tikzcd}[scale cd=.85]
		G_{\mathcal{A}_+} \arrow[d, shift right] \arrow[d, shift left]  &  & Z \arrow[rrd, "p_-"] \arrow[lld, "p_+"'] &  & G_{\mathcal{A}_-} \arrow[d, shift right] \arrow[d, shift left] & \overline{G_{\mathcal{B}_-}} \arrow[d, shift left] \arrow[d, shift right]  &  & W \arrow[lld, "q_-"'] \arrow[rrd, "q_+"] &  & G_{\mathcal{B}_+} \arrow[d, shift right] \arrow[d, shift left]  \\
		X_+                                                                                          &  &                                        &  & X_- \arrow[llu, "\lambda_Z", bend left]                                                       & \overline{X_-}                                                                                          &  &                           &  & X_+ \arrow[llu, "\lambda_W", bend left]                                                      
	\end{tikzcd}   \label{eq: Z and W morita} 
	\end{equation*} 
	and then they determine smooth Lie groupoid  isomorphisms $\phi_Z:G_{\mathcal{A}_-}   \rightarrow G_{\mathcal{A}_+} $, $\phi_W:G_{\mathcal{B}_+} \rightarrow \overline{G_{\mathcal{B}_-}} $. An immediate property that these isomorphisms satisfy is that 
	\begin{align} &\phi_W|_{X_-}\circ \phi_Z|_{X_-}=\text{id}:X_- \rightarrow \overline{X_-} \label{eq:mulbis1} \\
		&\phi_Z|_{{X_+}}\circ \phi_W|_{{X_+}}=\text{id}:X_+\rightarrow \overline{X_+} \label{eq:mulbis2} 
		. \end{align} 
	Note that, forgetting to the smooth category, both Morita equivalences above are self Morita equivalences of Lie groupoids. Consider the real quasi-symplectic forms $\beta_A$, $\beta_B$ induced by the integration of $(\A_\pm)_R $ and $(\B_\pm)_R $ respectively, see \autoref{cor: the underlying re and im}. By multiplicativity of $\beta_A$, $\beta_B$ and by virtue of being induced by bisections, the isomorphisms $\phi_A$, $\phi_B$ satisfy the following identities:
	\begin{align*}  \phi_W^*\beta_B=\beta_B+\delta^*\lambda_W^*\beta_B,  \qquad \phi_Z^*\beta_A=\beta_A+\delta^*\lambda_Z^*\beta_A. \end{align*}  
	Now we may define the following multiplicative bisections:
	\begin{align*} &\Lambda_Z:  G_{\mathcal{B}_-} \rightarrow \mathcal{Z}, \quad \Lambda_Z(b)=(\lambda_Z(\ttt(b)),b,\phi_W^{-1}(b),\lambda_Z(\sss(b))), \quad b\in G_{\B_-}  \\
		 &\Lambda_W:  G_{\mathcal{A}_+} \rightarrow \mathcal{W},\quad  \Lambda_W(a)=(\phi_Z^{-1}(a),\lambda_W(\sss(a)),\lambda_W(\ttt(a)),a), \quad a\in G_{\A_+}. 
	\end{align*}  
	Note that we are using the fact that $\mathcal{Z} $ and $\mathcal{W}$ are diffeomorphic to $D$ as in \eqref{eq:luweidougro}; equation \eqref{eq:mulbis1} implies that $\Lambda_Z$ is well defined, whereas \eqref{eq:mulbis2} does the same for $\Lambda_W$. Now the real parts of the symplectic forms $\Omega_{\mathcal{Z} }$ and $\Omega_{\mathcal{W} }$ coincide with $\Re{(\Omega_-)}$ as in \autoref{prop:2form luwei} according to \autoref{thm: the (1,1) morphism of a GK}. So we get that 
	\[ -\Lambda_Z^* (\Omega_{\mathcal{Z} })_R=\phi_W^* \beta_B - \beta_B +\delta^*( \lambda_Z\circ \phi_W|_{X_+})^* \beta_A=\delta^*\lambda_W^*\beta_B ++\delta^*( \lambda_Z\circ \phi_W|_{X_+})^* \beta_A=0. \]
	The last equation above follows from the fact that $\lambda$ is a Lagrangian quadruple section of $\mathcal{S}$ and 
	\[ \lambda(\phi_W(\ttt(b)))=(\lambda_Z(\phi_W(\ttt(b)))^{-1},\lambda_W(\ttt(b))^{-1},\lambda_W(\ttt(b)),\lambda_Z(\phi_W(\ttt(b))))\quad  \forall b \in G_{\mathcal{B}_+ }, \]
	which implies that $\ttt^*\lambda_W^*\beta_B+\ttt^*\phi_W^*\lambda_Z^*\beta_A=0$ and, similarly, $\sss^*\lambda_W^*\beta_B+\sss^*\phi_W^*\lambda_Z^*\beta_A=0$. 
	Analogous considerations imply that $\Lambda_W^* (\Omega_{\mathcal{W} })_R =0$. By complex conjugation, we automatically have definitions for $\Lambda_{\overline{Z}}$ and $\Lambda_{\overline{W}}$.
	
	Finally, notice that the bisections induce a diffeomorphism $\psi_H:Z \rightarrow \overline{Z} $ given by the composition indicated in the following diagram 
	\[ \adjustbox{scale=0.8,center}{\begin{tikzcd}
		\overline{G_{\A-}} \arrow[d, shift right] \arrow[d, shift left] \arrow[dd, "(\psi_Z')^{-1}"', bend right=49] & \mathcal{W} \arrow[r] \arrow[l] \arrow[d, shift right] \arrow[d, shift left] & G_{\A+} \arrow[d, shift right] \arrow[d, shift left] \arrow[ll, "\phi_Z^{-1}"', bend right] \\
		\overline{X_-}                                                                                    & W \arrow[l] \arrow[r]                                                        & X_+                                                                                     \\
		\overline{Z} \arrow[u]                                                                            & \mathcal{C} \arrow[u] \arrow[r] \arrow[l]                                              & Z \arrow[u] \arrow[l, "\Lambda_\text{Hor}", bend left] \arrow[uu, "\psi_Z"', bend right=49]     
	\end{tikzcd} \begin{tikzcd}
a\cdot \lambda_Z(\phi_Z^{-1}(\sss(a))) \arrow[rr, "\psi_Z", maps to]                    &  & a              \\
\phi_Z^{-1}(a)\cdot\lambda_Z(\phi_Z^{-1}(\sss(a)))^{-1} \arrow[rr, "\psi_Z'", maps to] &  & \phi_Z^{-1}(a).
\end{tikzcd} }\]
	Now we can define $\Lambda_\text{Hor}(z)$ for $z=a\cdot \lambda_Z(\phi_Z^{-1}(s(a)))$ and $a\in G_{\mathcal{A}_+ }$ as the following element of $\mathcal{C}$: 
	\[  \adjustbox{scale=0.85,center}{\begin{tikzpicture}[line cap=round,line join=round,>=stealth,x=1cm,y=1cm,scale=0.4]
		\clip(-12.3,-5.12) rectangle (14.7,4.3);

     \draw[fill=black!10] (0.44,-3.78) rectangle (6.86,2.64);
		\draw [<-,line width=0.9pt] (0.44,2.64) -- (6.86,2.64);
		\draw [<-,line width=0.9pt] (6.86,2.64) -- (6.86,-3.78);
		\draw [->,line width=0.9pt] (6.86,-3.78) -- (0.44,-3.78);
		\draw [->,line width=0.9pt] (0.44,-3.78) -- (0.44,2.64);
		\begin{scriptsize}
			
			\draw[color=black] (-0.5,-3.55) node {$\sss(a)$};
			
			\draw[color=black] (8.6,-3.55) node {$\phi_Z^{-1}(\sss(a))$};
			
			\draw[color=black] (7.5,3.2) node {$\ttt(a)$};
			
			\draw[color=black] (-2.8,3.2) node {$\phi_Z^{-1}(\ttt(a))=\phi_W(\ttt(a))$};
			\draw[color=black] (3.7,3.27) node {$\lambda_W(\ttt(a))$};
			\draw[color=black] (10.5,-0.4) node {$z=a\cdot \lambda_Z(\phi_Z^{-1}(\sss(a)))$};
			\draw[color=black] (3.68,-4.39) node {$\lambda_W(\sss(a))^{-1}$};
			\draw[color=black] (-5.3,-0.4) node {$ \psi_H(z):=\phi_Z^{-1}(a)\cdot\lambda_Z(\phi_Z^{-1}(\sss(a)))^{-1}$};
			\draw[color=black] (3.55,-0.5) node {$\lambda(\phi_Z^{-1}(\sss(a)))$};
		\end{scriptsize}
	\end{tikzpicture}} \]
	using again the fact that $\lambda$ is Lagrangian, we get that $(\Lambda_\text{Hor}^*\Omega_\C)_R=0$. Now we have the equivariance property of $\Lambda_\text{Hor}$ which shows it is a horizontal trivialization: for suitable $a_\pm \in G_{\A_\pm}$ and $z\in Z$ as above, $z\mapsto \Lambda_\text{Hor}(z)(\psi_H(z),\lambda_W(\sss(a))^{-1},\lambda_W(\ttt(a)),z)$ satisfies:
 \[  \Lambda_\text{Hor}(a_+\cdot z\cdot a_-)=\Lambda_W(a_+)\cdot \Lambda_\text{Hor}(z)\cdot \Lambda_{\overline{W}}(a_-) .\]
 We proceed similarly to define a bisection $\Lambda_\text{Ver}:W \rightarrow \mathcal{C}$ which determines a vertical trivialization of \eqref{diagram:(1,1)morphism}; all these bisections together trivialize \eqref{diagram:(1,1)morphism} in the sense of \autoref{def: 11 morphism}.  \end{proof}  

\begin{remark}\label{rem: lu-wei bimodules} In order to define the Lagrangian bisections described in the proof of \autoref{thm: key Lagrangian in the central brane}, instead of assuming from the start that all the symplectic bimodules are diffeomorphic to $D$ as in \eqref{eq:luweidougro}, it is enough to assume that they are of the following form: 
\begin{equation*}\begin{aligned}
    &\mathcal{Z}=\left\{(z,g,g',z')\in Z\times G_{\B_-}\times G_{\B_+}\times Z\left| \begin{aligned}
        p_-(z)=\ttt(g),\, p_+(z)=\ttt(g'),\\
        p_-(z')=\sss(g),\, p_+(z')=\sss(g')
    \end{aligned}\right.\right\}, \\
    &\mathcal{W}=\left\{(h,w,w',h')\in  \overline{G_{\A_-}}\times W\times W\times G_{\A_+}\left|\begin{aligned}
         q_-(w)=\sss(h),\, q_+(w)=\sss(h'),\\
         q_-(w')=\ttt(h),\, q_+(w')=\ttt(h')
    \end{aligned}\right.\right\};
\end{aligned}
\end{equation*}
and, similarly, $\mathcal{C}$ should be the fibred product of $Z$, $\overline{Z}$, $W$ and $\overline{W}$ over $X_-$, $ \overline{X_-} $, $ X_+ $, $ \overline{X_+}$. Moreover, we need the action of $(x,k,g,x')\in D_-\subset G_{\A_-}\times G_{\B_-} \times G_{\B_-}\times G_{\A_-}$ on $(z,g,g',z')\in \mathcal{Z}$ to be given by
\begin{equation}
    (z,g,g',z')\cdot(x,k,g,x')=(z\cdot x,k,g',z'\cdot x'), \label{eq:luweiact}
\end{equation} 
and analogous expressions should define the actions on $\mathcal{W}$ and on $\mathcal{C}$. In this situation, we say that the symplectic bimodules are of {\em Lu-Weinstein type}.
\end{remark}
\subsection{Integration of generalized K\"ahler structures}
We summarize the main results of \S\ref{section: 2dbs}, \S\ref{subsec:real str} and \S\ref{subsec: lag branes} in the following theorem, which describes the solution to the integration problem for generalized K\"ahler structures. 

\begin{theorem}\label{thm: the (1,1) morphism of a GK}
	A generalized K\"ahler structure $(\J_A,\J_B) $ on $M$ for which the underlying real Poisson structures $\pi_A, \pi_B$ are integrable determines a self-adjoint $(1,1)$-morphism of holomorphic symplectic double groupoids  
	\begin{equation}
	\begin{tikzcd}
\overline{D_-}                              & {\mathcal{W}} \arrow[r] \arrow[l]                   & D_+                               \\
(-\overline{\mathcal{Z}}) \arrow[d] \arrow[u] & \mathcal{C} \arrow[d] \arrow[r] \arrow[u] \arrow[l] & {\mathcal{Z}} \arrow[d] \arrow[u] \\
\overline{D_+}                              & (-\overline{\mathcal{W}} )\arrow[r] \arrow[l]         & D_-                              
\end{tikzcd}\label{diagram:(1,1)morphism}
 \end{equation} 
 Moreover, there is a natural Lagrangian embedding of $M$ in the real symplectic core of $\mathcal{C}$, giving rise to a real trivialization of this $(1,1)$-morphism.
\end{theorem}
\begin{proof}
The integrability of $\pi_A, \pi_B$ implies, by \autoref{cor: the underlying re and im}, that the imaginary parts of $\A_\pm,\B_\pm$ are integrable; by \cite{MR4480214}, this implies integrability, in particular, of the holomorphic Dirac structures $\A_-,\B_-$ to a holomorphic symplectic double groupoid $D_-$ integrating the Manin triple $(\E_-,\mathcal{A}_-,\mathcal{B}_-)$, as in \autoref{prop:2form luwei}. \autoref{2d Bursztyn shuffle} implies then that $(\E_+,\A_+,\B_+)$ is also integrable by a holomorphic symplectic double groupoid $D_+$ and, moreover, these integrations and their conjugates are related by the holomorphic symplectic $(1,1)$-morphism \eqref{eq:(1,1) morphism via 2d shuffle}. The holomorphic symplectic form $\Omega_-$ of \autoref{prop:2form luwei} satisfies \eqref{eq: imomega} and then \autoref{prop:real str} implies that \eqref{eq:(1,1) morphism via 2d shuffle} is of the form \eqref{diagram:(1,1)morphism} and, also, that $\mathcal{C}$ is equipped with a symplectic real structure. Note that diagram \eqref{diagram:(1,1)morphism} is obtained from \eqref{eq:(1,1) morphism via 2d shuffle} via the identifications of the side Morita equivalences given by Fig.~\ref{fig:real structure}.

Finally, since the underlying smooth $(1,1)$-morphism of \eqref{diagram:(1,1)morphism} is the trivial one as in \autoref{ex:triv(1,1)mor}, $\mathcal{C}$ can be identified with the underlying smooth manifold $D$ of $D_-$ and hence there is a canonical inclusion $M\subset \mathcal{C}$ as the intersection between horizontal and vertical identity bisections. Because the real structure on  $\mathcal{C}$ is given by $\ii_A\circ\ii_B$ (\autoref{prop:real str}), this includes $M$ in the real symplectic core. Following the proof of \autoref{thm: key Lagrangian in the central brane}, we see that the induced Lagrangian bisections $\Lambda_{\text{Hor}}\subset\mathcal{C}$ and $\Lambda_{\text{Ver}}\subset\mathcal{C}$ are given by the unit inclusions $G_A,G_B\subset D$, respectively. Similarly, $\Lambda_{W}\cong G_A$ and $\Lambda_{Z}\cong G_B$ are the unit inclusions, constituting a real trivialization of \eqref{diagram:(1,1)morphism}. 
\end{proof}
\subsection{Examples}\label{subsec:examples}
Before presenting a family of novel examples in \S\ref{sec: lie groups}, we first describe how our formalism applies to generalized K\"ahler structures of symplectic type, which includes usual K\"ahler structures, as well as the case of commuting type generalized K\"ahler structures, in which the complex structures $I_+,I_-$ commute. 
\subsubsection{Generalized K\"ahler structures of symplectic type}\label{subsec:sym type}

\noindent\textbf{Infinitesimal data:}
As explained in \cite[\S 1]{MR4466669}, the holomorphic Manin triples associated to a symplectic type generalized K\"ahler  manifold are $( \DT_{1,0}X_\pm,\Gr(\sigma_{\pm}), T_{1,0}X_\pm)$, where $\sigma_{\pm}$ are Hitchin's Poisson structures; they are related by the gauge equivalence 
\[
L_{\sigma_+} = e^{2iF} L_{\sigma_-}, \qquad F=\pi_B^{-1}. 
\]

\noindent\textbf{Integration:}
 Let $(G_-, \Omega_-) \rightrightarrows X_-$ be a holomorphic symplectic groupoid integrating $\sigma_-$. Then the Bursztyn shuffle gives a holomorphic symplectic groupoid $(G_+,\Omega_+) \rightrightarrows X_+$ integrating $\sigma_+$, a holomorphic symplectic Morita equivalence $(Z,\Omega_Z)$ between the groupoids, and a $\Re$-Lagrangian bisection $L$ for $(Z,\Omega_Z)$.  To produce the $(1,1)$-morphism \eqref{diagram:(1,1)morphism}, we use the Lu-Weinstein construction \eqref{eq:luweidougro} as follows.  Define the following spaces:
\begin{equation*}
\begin{split}
    G_{\A_\pm} &= G_\pm\\
    Z &=Z
\end{split}\hspace{2em}
\begin{split}
    G_{\B_{\pm}} &= P(X_{\pm})\\
    \mathcal{Z} &= P(Z)
\end{split}\hspace{2em}
\begin{split}
    D_{\pm} &= P(G_{\pm})\\
    W &= \overline{X_-} \times X_+
\end{split}\hspace{2em}
\begin{split}
    \mathcal{C} &= \overline{Z} \times Z\\
    \mathcal{W} &= \overline{G_-}\times G_+,
\end{split}
\end{equation*}
where $P(Y) = Y \times Y$ denotes the pair groupoid over $Y$, with unit embedding given by the diagonal embedding $Y=Y_{\Delta}\hookrightarrow P(Y)$. We remark that the symplectic form on $\mathcal{W}$ is $\overline{\Omega_-} \times (- \Omega_+)$ and the symplectic form on $\mathcal{C}$ is $\overline{\Omega_Z} \times (- \Omega_Z)$. The symplectic real structure on $\mathcal{C}$ is given explicitly by the map $(z_1, z_2) \mapsto (z_2, z_1)$. The real symplectic core $\mathcal{S}$ is thus the diagonal submanifold $Z_{\Delta} \hookrightarrow \overline{Z} \times Z$. \\

\noindent\textbf{Lagrangian bisection of the real symplectic core:}
Let $\lambda$ be a Lagrangian bisection of $\mathcal{S}$. All the other bisections are generated by $\lambda$ according to \autoref{thm: key Lagrangian in the central brane} and they can be described as follows. Since $\lambda$ identifies the underlying real groupoids $G_{-}$, $G_+$, we let $G$ be the underlying smooth manifold of $G_{\pm}$ and $M$ be the underlying smooth manifold of $X_{\pm}$. Note that $\lambda$ projects to a $\Re$-Lagrangian bisection $L_Z$ in $Z$. Then  
\[
\Lambda_{Z} = P(L_Z), \quad \Lambda_{W} = G_{\Delta}, \quad \Lambda_\text{Hor}= Z_{\Delta}, \quad \Lambda_\text{Ver} = \overline{L_Z}\times L_Z. 
\]
In this way, we recover the treatment of generalized K\"ahler structure of symplectic type in~\cite{MR4466669}.

\begin{example}[K\"ahler manifolds]\label{ex: kahler case} In the case of a K\"ahler manifold, $X_+=X_-=X$ and we have $G_\pm=T^*X$, the holomorphic cotangent bundle. Then $Z=(T^*X,\omega_{\text{can}}+{2iF})$ is a symplectic affine bundle which is isomorphic to the one constructed by Donaldson in \cite{MR1959581}. $\Re$-Lagrangian bisections of $Z$ correspond to K\"ahler metrics within the fixed K\"ahler class $[F]$.  Note that the factor of two is occurring as a result of our convention, from \S\ref{subsec:poi gro gk man}, to integrate the Manin triples for the Courant algebroids $\E_\pm$, which have classes $[\H_\pm]=2H$.  The factor is then removed in the reconstruction~\eqref{factwo} of the generalized K\"ahler structure from the choice of Lagrangian bisections. 
\end{example}

\subsubsection{Generalized K\"ahler structures of commuting type}\label{subsec: comm}

\noindent\textbf{Infinitesimal data:}
A generalized K\"ahler manifold is of commuting type if $[I_+,I_-] = 0$, or, equivalently, if $\sigma_{\pm} = 0$. As shown in \cite[Theorem A]{MR2287917}, such spaces are locally ``twisted'' products of K\"ahler manifolds. By \cite[Remark 3.10]{MR3232003}, the induced generalized K\"ahler structures structures on $\A$ and $\B$ leaves are K\"ahler since they are of symplectic type and their Hitchin Poisson structures vanish.  

To treat the local twisted product, we assume our manifold is $M = M_1 \times M_2$, where $M_1, M_2$ are equipped with complex structures $I_{1}$ and $I_2$. We let $g_1$ be a $M_2$-family of Kähler metrics on $(M_1, I_1)$ and $g_2$ be a $M_1$-family of Kähler metrics on $(M_2, I_2)$. Consider the tensors 
\[
I_{\pm} = \begin{pmatrix} 
	I_1 & 0 \\
	0 & \pm I_2 
\end{pmatrix} , \qquad g = \begin{pmatrix}
	g_1 & 0 \\
	0 & g_2 
\end{pmatrix}, \qquad \omega_{\pm} :=gI_\pm= \begin{pmatrix}
	F_1 & 0 \\
	0 & \pm F_2
\end{pmatrix};
\]
where $F_1 = gI_1$, and $F_2 = gI_2$. The induced holomorphic Manin triples on $M$ are isomorphic to
\begin{align*}
	&(\E_+,\A_+, \B_+) = (\DT_{1,0} M_1\times \DT_{1,0}M_2,\ \ T_{1,0}^*M_1 \times T_{1,0}M_2,\ \ T_{1,0}M_1 \times T_{1,0}^*M_2), \\
	&(\E_-,\A_-, \B_-) = (\DT_{1,0} M_1\times \DT_{0,1}M_2,\ \  T_{1,0}^*M_1 \times T_{0,1}M_2,\ \ T_{1,0}M_1 \times T_{0,1}^*M_2),
\end{align*}

\noindent\textbf{Integration:}
We construct a $(1,1)$-morphism as in \eqref{diagram:(1,1)morphism} as follows: 
\begin{equation*}
\begin{split}
    G_{\A_+} &= T^*M_1 \times P(M_2)\\
    G_{\A_-} &= T^*M_1 \times P(\overline{M_2})\\
    D_+ &= P(T^*M_1) \times P(T^*M_2)\\
    Z &= T^*M_1 \times M_2 \times \overline{M_2}\\
    W &= \overline{M_1} \times M_1 \times T^*M_2\\
    \mathcal{C} &= \overline{T^*M_1} \times T^*M_1 \times \overline{T^*M_2} \times T^*M_2\\
\end{split}\hspace{3em}
\begin{split}
 G_{\B_+} &= P(M_1) \times T^*M_2 \\
 G_{\B_-} &= P(M_1) \times \overline{T^*M_2}. \\
	 D_- &=  P(T^*M_1) \times P(\overline{T^*M_2}), \\
	\mathcal{Z} &= P(T^*M_1) \times T^*M_2 \times \overline{T^*M_2}, \\
	\mathcal{W} &= \overline{T^*M_1}\times T^*M_1 \times P(T^*M_2).
\end{split}
\end{equation*}
Let $\omega_1, \omega_2$ be the canonical holomorphic symplectic forms on $T^*M_1$ and $T^*M_2$, and let $p_1, p_2$ be the corresponding cotangent bundle projections. Then the symplectic forms on the above spaces are determined as follows: $D_+$ is equipped with $\delta^*\omega_1\times -\delta^*\omega_2$, a similar expression holds for $D_-$, and we have 
\begin{align*}
	&\Omega_{\mathcal{Z}} = (\omega_1 + 2i p_1^*F_1) \times -(\omega_1 + 2i p_1^*F_1) \times (-\omega_2) \times \overline{\omega_2} \\
	&\Omega_{\mathcal{W}} = \overline{\omega_1}\times (-\omega_1) \times (\omega_2 + 2ip_2^*F_2) \times -(\omega_2 + 2ip_2^*F_2)\\
 & \Omega_{\mathcal{C}}=(\omega_1-2ip_1^*F_1)\times (-\omega_1-2ip_1^*F_1)\times (\omega_2-2ip_2^*F_2)\times (-\omega_2-2ip_2^*F_2).
\end{align*}
The real structure on $\mathcal{C}$ is given by $(u_1, u_2, v_1, v_2) \mapsto (-u_2, -u_1, -v_2, -v_1)$, giving rise to the real symplectic core $\mathcal{S}$ defined by $ T^*M_1 \times T^*M_2\cong \{ (u_1,-u_1,v_1,-v_1)\ : \ (u_1,v_1)\in T^*M_1 \times T^*M_2 \}$. \\

\noindent\textbf{Lagrangian bisection of the real symplectic core:} 
The zero section of the cotangent bundle defines a Lagrangian bisection in $\mathcal{S}$, which determines the following Lagrangian bisections: 
\begin{equation*}
\begin{split}
    \Lambda_Z &= P(Z_{M_1}) \times (T^*M_2)_{\Delta}\\
    \Lambda_\text{Hor} &= (T^*M_1)_{\Delta} \times P(Z_{M_2})
\end{split}
\qquad
\begin{split}
    \Lambda_W &= (T^*M_1)_{\Delta} \times P(Z_{M_2})\\
    \Lambda_\text{Ver} &= P(Z_{M_1}) \times (T^*M_2)_{\Delta},
\end{split}
\end{equation*}
where $Z_{M_1}$ and $Z_{M_2}$ are the zero sections of $T^*M_1$ and $T^*M_2$ respectively. The choice of a more general Lagrangian can be formulated in terms of Hamiltonian deformations, see \autoref{ex:ham def com type} below.

\section{The generalized K\"ahler potential}\label{sec:gk potential}
According to \autoref{thm: the (1,1) morphism of a GK}, a generalized K\"ahler manifold determines, under natural conditions, a holomorphic symplectic $(1,1)$-morphism. We may interpret this morphism as a geometric realization of the {\em generalized K\"ahler class}, since, as we have seen in \autoref{ex: kahler case}, it generalizes Donaldson's geometric realization of the usual K\"ahler class, see also \cite[\S 2]{MR4466669}. Moreover, \autoref{thm: key Lagrangian in the central brane} shows that we obtain a real trivialization of the (1,1) morphism \eqref{diagram:(1,1)morphism} by means of the choice of a Lagrangian bisection of its real symplectic core. We now show in \autoref{thm: reconstruction} how this choice, under a generic positivity condition, determines canonically a generalized K\"ahler metric. This allows us to interpret generalized K\"ahler metrics as positive Lagrangian bisections in the real symplectic core.  With an appropriate choice of Darboux coordinates on the real symplectic core, therefore, we obtain a local description of a generalized K\"ahler metric in terms of a single real-valued function, the generalized K\"ahler potential. 
\subsection{From Lagrangian bisections to generalized K\"ahler metrics} \label{subsec: reco}

\begin{definition} We say that a pseudo-Riemannian metric $g$ and a pair of  complex structures $I_{\pm}$ compatible with $g$ constitute a {\em pseudo generalized K\"ahler structure} if the associated hermitian forms $\omega_{\pm}$ satisfy \eqref{dcpm} for a closed 3-form $H$.
\end{definition}
In this subsection, we prove a differentiation result that takes a a self-adjoint holomorphic symplectic $(1,1)$-morphism as in Fig.~\ref{fig:real structure}, together with a real trivialization subject to an open nondegeneracy condition, to a pseudo generalized K\"ahler structure on $M$.


\begin{theorem}\label{thm: reconstruction}A self-adjoint $(1,1)$-morphism of symplectic double groupoids, equipped with a real trivialization in which the Re-Lagrangian bisections are Im-symplectic, determines canonically a pseudo generalized K\"ahler structure on the underlying smooth base manifold of the $(1,1)$-morphism. 

Moreover, the corresponding metric depends only on the Lagrangian bisection of the real symplectic core given by the intersection of the Lagrangian bisections in the central bimodule. 
\end{theorem}
\begin{remark}Assuming that the $(1,1)$-morphism is given by double groupoids and bimodules of Lu-Weinstein type (\autoref{rem: lu-wei bimodules}), a real trivialization is determined by a Lagrangian bisection of the real symplectic core, see \autoref{thm: key Lagrangian in the central brane}. 
\end{remark}
\begin{proof}[Proof of \autoref{thm: reconstruction}] Using the real structure on the $(1,1)$-morphism, we may assume that it is of the form \eqref{diagram:(1,1)morphism}. The strategy in our proof is as follows.
 First, we pull back the holomorphic symplectic forms on $\mathcal{W}$ and $\mathcal{Z}$ to $\Lambda_W$ and $\Lambda_Z$ to obtain two real multiplicative symplectic forms.  In the second step, we differentiate these symplectic forms to obtain real Poisson structures. In the third step, we extract a nondegenerate symmetric tensor $g$ from the pair of Poisson structures and we show that $I_\pm$ are compatible with $g$. Finally, we prove the integrability condition \eqref{dcpm}.

\textbf{Step 1.} By \autoref{prop: isomorphism induced by bisection}, the Re-Lagrangian bisections $\Lambda_{Z}$ and $\Lambda_{W}$ identify the underlying smooth manifolds of $D_-, D_+, Z$ and $W$. We will let $G_A, G_B, $ and $M$ be the underlying smooth manifolds of $G_{\A_{\pm}}, G_{\B_{\pm}}$ and $M$. We start with an analysis of $G_A$ and $G_B$. 

Let $\Omega_{\mathcal{Z}}$ and $\Omega_{\mathcal{W}}$ be the holomorphic symplectic forms on $\mathcal{Z}$ and $\mathcal{W}$, and let
\begin{equation}\label{factwo}
F_A= \tfrac{1}{2i}\Lambda_W^*\Omega_{\mathcal{W}}, \qquad F_B= \tfrac{1}{2i}\Lambda_Z^*\Omega_{\mathcal{Z} }.
\end{equation}
By \autoref{prop: change of the sympletcic forms} and a similar argument as in \cite[Theorem 5.3]{MR4466669}, $G_{A}$ and $G_B$ are endowed with gauge equivalent Poisson groupoid structures $(\Pi_{+}^A, \overline{\Pi_-^A}, F_A)$ and $(\Pi_+^B, \Pi_-^B, F_B)$ respectively. In other words, we obtain the converse of \autoref{prop: deg symplectic GK}. Since these Lagrangian bisections are part of a real trivialization, the identification between $D_-$ and $\overline{D_-}$ induced by $\Lambda_Z$ and $\Lambda_W$ is the identity. By \cite[\S1.1]{MR4466669}, we may encode then these structures by means of the following equations
\begin{align} 
	&F_BI_++I_-^*F_B=0, \qquad F_AI_--I_+^*F_A=0 \label{eq: on FAFB 1}\\
	&I_+=I_-+2Q_BF_B, \qquad -I_-=I_++2Q_AF_A, \label{eq: on FAFB 2}
\end{align}
where $Q_A$ and $Q_B$ are the real parts of the holomorphic Poisson structures on $G_{\mathcal{A}_- }$ and ${G}_{\mathcal{B}_- }$ respectively.

\textbf{Step 2.} Since both $F_A$ and $F_B$ are multiplicative symplectic forms, by using \eqref{eq:IM 2 form}, we can identify the underlying smooth Lie algebroids $A$ and $B$ of $G_A$ and $G_B$ with graphs of Poisson structures $\pi_A, \pi_B$ on $M$:
\begin{align}\label{eq: pi1 and pi2}
&\Gr(\pi_A)=(\rho_A,\mu_{F_A})(A), \qquad \Gr(\pm\pi_B) =(\rho_B,\pm \mu_{F_B})(B).  
\end{align}
We now check that $\Gr(\pi_A)$ and $\Gr(\pm \pi_B)$ are complementary and hence the embeddings of $A$ and $B$ as graphs of Poisson structures induce vector bundle isomorphisms 
\[
\begin{tikzcd}
    A \oplus B\arrow[r,"{\Phi_{\pm}}","\cong"'] &\DT M
\end{tikzcd}
\]
\[
\Phi_\pm (a,b) = \rho_A(a)+ \rho_B(b) + \mu_{F_A}(a) \pm \mu_{F_B}(b), \quad (a,b) \in A\oplus B. 
\]
It suffices to check that the pairings of $\Gr(\pi_A)$ with $\Gr(\pm \pi_B)$ are nondegenerate. A consequence of these isomorphisms is then that $A\oplus B$ is an exact Courant algebroid. We use the identity
\begin{align}\label{Relation between F_A F_B}
	(\Omega_\pm )_I:= \Im(\Omega_{\pm}) =-\delta_A^* F_A \pm \delta_A^* F_B, 
\end{align}
which follows from applying \autoref{prop: change of the sympletcic forms} to the isomorphisms determined by the real trivialization. Computing the pairing for all $(a,b)\in A \oplus B$, we obtain 
\begin{align}\label{eq: a nondeg pairing}
	\begin{split}
		\langle (\rho(a), \mu_{F_A}(a) ), (\rho(b), \pm \mu_{F_B}(b)) \rangle &= \langle \mu_{F_A}(a), \rho(b)\rangle  \pm \langle \mu_{F_B}(b), \rho(a)\rangle  \\
		&= F_A(dt_A(a), dt_A(b))  \pm F_B(dt_B(a), dt_B(b)) \\
		&= -(\Omega_{\mp})_{I} (a,b),  
	\end{split}
\end{align}
where the last identity follows from \eqref{Relation between F_A F_B}. Note that, according to \autoref{Poisson duality from IM 2-form}, the pairings above are nondegenerate and hence $((\DT M,0),\Gr(\pi_A),\Gr(\pm\pi_B))$ are Manin triples. 

\textbf{Step 3.} Since $\Gr(\pi_A)$ and $ \Gr(\pm\pi_B)$ are complementary, we may define  2-forms as follows:
\begin{align}\label{eq: from poisson to hermtian forms}
\omega_{\pm} = (\pi_B \pm \pi_A)^{-1} .
\end{align}
We now prove that $\omega_{\pm}$ satisfy the generalized K\"ahler compatibility conditions, namely 
\begin{align}
	\omega_+ I_{\pm} &= - I_{\mp}^*\omega_{-} \label{eq:compatibility1} \\
	I_+^*\omega_+ &= I_-^* \omega_- \label{eq:compatibility2}
\end{align}
Both equations then imply that $g = -\omega_{\pm}I_{\pm}$ is a metric (possibly indefinite) compatible with $I_{\pm}$. 

To prove \eqref{eq:compatibility1}, it suffices to show that 
\begin{align}\label{eq: pi1 pi_2 are mixed type}
	I_{\pm}\pi_A = \pi_A I^*_{\mp}, \qquad -I_{\pm} \pi_B = \pi_B I^*_{\mp} .
\end{align} 
Using \eqref{eq: pi1 and pi2} and \eqref{eq: on FAFB 1}, we have for all $a \in A$ and $b \in B$,
\begin{align*} 
	& (\rho_A,\mu_{F_A})(I_\pm a) = I_\pm \rho_A(a) + I_\mp^* \mu_{F_A}(a)=I_\pm \pi_A(\mu_{F_A}(a))+ I_\mp^* \mu_{F_A}(a) \\
	& (\rho_B,\mu_{F_B})(I_\pm b) = I_\pm \rho_B(b)  -I_\mp^* \mu_{F_B}(b)=I_\pm \pi_B(\mu_{F_B}(b))+ I_\mp^* \mu_{F_B}(b),
\end{align*}  
establishing \eqref{eq: pi1 pi_2 are mixed type} and hence \eqref{eq:compatibility1}.

To prove \eqref{eq:compatibility2}, we first note that the equation is equivalent to 
\begin{align}\label{eq: long}
I_+^*-I_-^* = I_-^*(\omega_-\omega_+^{-1}-1).
\end{align}
We expand the left hand side using  \eqref{eq: on FAFB 1} and \eqref{eq: on FAFB 2} as follows
\begin{align}\label{eq: I+minusI-}
I_+^*-I_-^*=\mu_{F_B} \circ (I_+-I_-)\circ \mu_{F_B}^{-1} = 2\mu_{F_B}Q_B|_{T^*M} = 2\mu_{F_B}I_+R^B_{+}|_{T^*M};
\end{align}
where $R_{\pm}^B$ are the imaginary parts of the holomorphic Poisson structures $\Pi_{\pm}^B $ of $G_{\B_{\pm}}$. Note that the maps $(R^B_{\pm})^\sharp: T^*G_B \to TG_B$ restrict to maps $T^*M \to B$ which are given by $\phi_{\pm} \circ \rho_A^*: T^*M \to B$,  
where $\phi_{\pm}:A^* \rightarrow B$ are the dualities induced by the inclusions of $A$ and $B$ inside $(\E_{\pm})_I$. As a result, 
\begin{align}\label{Eq 0 for muFB}
	2\mu_{F_B} I_+R_+^B|_{T^*M}=-2I_-^*\mu_{F_B}\circ \phi_+\circ \rho_A^*. 
\end{align}
This may be simplified further as follows, 
\begin{align}\label{Eq 1 for muFBphirho}
	\langle \mu_{F_B} \circ \phi_+ \circ \rho_A^*\a, \omega_-^{-1} \mu_{F_A}(a) \rangle = -\langle \phi_+^{-1}\phi_+ \circ \rho_A^*\a, a \rangle = -\langle \a, \pi_A\mu_{F_A}(a)\rangle, \quad a \in A,
\end{align}
where we use the following observation: for all $a \in A, b \in B$,
	\begin{align}
		\langle \phi_\pm^{-1}(b),a \rangle =\mp \langle \omega_\mp^{-1} \mu_{F_A} (a),\mu_{F_B} (b) \rangle.\label{eq: rewrite the pairing}
	\end{align} 
which holds by the following computation (and similarly for the opposite sign):
\begin{align*}
		\langle \phi_+^{-1} (b), a \rangle &= \langle \rho_B(b) \oplus \mu_{F_B}(b), \rho_A(a)\oplus \mu_{F_A}(a) \rangle = \pi_A(\mu_{F_A}(a), \mu_{F_B}(b)) + \pi_B(\mu_{F_B}(b), \mu_{F_A}(a)) \\
		&= -\langle (\pi_B - \pi_A) \mu_{F_A}(a), \mu_{F_B} (b) \rangle =- \langle \omega_{-} \mu_{F_A}(a), \mu_{F_B}(b) \rangle .
	\end{align*}
For the right hand side of~\eqref{eq: long}, notice that $1-\omega_-\omega_+^{-1}= -2\omega_-\pi_A$. Therefore, 
\begin{align}\label{Eq 2 for muFBphirho}
	\begin{split}
		\langle (1-\omega_-\omega_+^{-1})\a, \omega_{-}^{-1}\mu_{F_A}(a) \rangle &= -2 \langle \omega_- \pi_A\a, \omega_{-}^{-1} \mu_{F_A}(a) \rangle = 2 \langle \mu_{F_A}(a), \pi_A\a \rangle  \\
		&= -2 \langle \a, \pi_A\mu_{F_A}(a) \rangle  
	\end{split}
\end{align}
By \eqref{Eq 1 for muFBphirho} and \eqref{Eq 2 for muFBphirho}, we have  
\begin{align}\label{Eq 3 for muFBphirho}
	2\mu_{F_B} \circ \phi_+ \circ \rho_{A}^* = 1- \omega_-\omega_+^{-1}.
\end{align}
Combining \eqref{eq: I+minusI-}, \eqref{Eq 0 for muFB}, and \eqref{Eq 3 for muFBphirho}, we conclude \eqref{eq: long} and hence \eqref{eq:compatibility2}.

\textbf{Step 4.} It remains to check the integrability of $\omega_{\pm}$. We note that $\omega_{\pm}$ determine the  maps \eqref{eq:IM 2 form} corresponding to $F_{A}^{(1,1)_{\pm}}$ and $F_{B}^{(1,1)_{\pm}}$. In other words,
\begin{equation}\label{IMFAFB}
		\pm  \omega_\pm \circ \rho_A= \mu_{F_A^{(1,1)_\pm}},  \qquad
		 \omega_\pm \circ \rho_B= \mu_{F_B^{(1,1)_\pm}}. 
\end{equation} 
To obtain the $(1,1)$ components of $F_A, F_B$, we decompose $F_AI_\pm$ and $F_BI_\pm$ into symmetric and skew-symmetric components  as follows:
\[ F_AI_\pm= S_A^\pm + \beta^\pm_A, \quad F_BI_\pm= S_B^\pm + \beta^\pm_B.
\]
Applying equations \eqref{eq: on FAFB 1}, we obtain
\begin{equation} \label{eq:symske}
\begin{split} 
S_A^+&=-S_A^-=\pm F^{(1,1)_\pm}_A I_\pm\\
S_B^+&=S_B^-= F^{(1,1)_\pm}_B I_\pm
\end{split}
\qquad
\begin{split}
    \beta_A^+&=\beta_A^-= F^{(2,0)_\pm+(0,2)_\pm}_A I_\pm\\
    \beta_B^+&=-\beta_B^-=\pm F^{(2,0)_\pm+(0,2)_\pm}_B I_\pm
\end{split}
\end{equation} 
Since all the tensors above are multiplicative, we obtain the following counterpart to~\eqref{Relation between F_A F_B}: 
\[ 
(\Omega_\pm )_R=-\delta^* (F_A I_\pm)\pm \delta^*( F_BI_\pm)= -\delta^* \beta_A^\pm \pm \delta^* \beta_B^\pm. 
\]
Since the 2-forms $(\Omega_\pm )_R$ are symplectic and multiplicative over both $G_A$ and $G_B$, we get that $\mu_{(\Omega_\pm )_R}:(\text{Lie}(D)^B \rightrightarrows A)\cong( T^*G_B \rightrightarrows B^*)$ as in \eqref{eq:IM 2 form} are isomorphisms and  just as in \eqref{eq: a nondeg pairing}, the following pairings are nondegenerate
\[ 
\langle \mu_{(\Omega_\pm )_R}^B (a),b \rangle =(\Omega_\pm )_R(a,b)=-\beta_A^\pm(a,\rho(b))\pm \beta^\pm_B(\rho(a),b) \qquad   (a,b) \in A\oplus B. 
\] 
Then the maps $(\rho_A,\mu_{\beta_A^\pm }):A \rightarrow \mathbb{T}M  $ and $(\rho_B,\pm \mu_{\beta_B^\pm }) :B \rightarrow \mathbb{T}M $ are embeddings and their images are two transverse maximal isotropic subbundles $\tilde{A}$ and $\tilde{B}$. In other words, we have that 
\[
\Psi_\pm:=(\rho_A,\mu_{\beta_A^\pm })\oplus (\rho_B,\pm \mu_{\beta_B^\pm }):A\oplus B \rightarrow \tilde{A}\oplus \tilde{B}=\mathbb{T}M 
\]
coincide, i.e. $\Psi:=\Psi_+=\Psi_-$. Since $\Psi$ preserves pairings and is compatible with the anchors, it defines a splitting of the Courant algebroid $(A\oplus B)_R$, see \autoref{rem: real and im manin triples of a double}, which identifies the Courant algebroid $(A\oplus B)_R $ with $ (\DT M,H)$, where $H\in \Omega^3_{cl}(M)$ is the curvature of the splitting. 

By definition, $\beta_A^\pm$ and $\pm\beta_B^\pm$ are quasi-symplectic forms integrating the Dirac structures $\tilde{A}$ and $\tilde{B}$, so  satisfy
\begin{align}  d \beta^\pm_A= \delta^*H, \qquad \pm d \beta^\pm_B= \delta^*H. \label{eq:dirgk} 
\end{align} 
	Note that for any closed 2-form $F$ we have that $d^c F^{(1,1)}=d( F^{(2,0)+(0,2)}I)$ and so, combining this with \eqref{eq:dirgk} and \eqref{eq:symske}, we have 
 \begin{align}\label{dcFAFB}
 d^c_{\pm} F_{A}^{(1,1)_{\pm}} = \delta^*H, \quad \pm d^c_{\pm} F_{B}^{(1,1)_{\pm}} = \delta^*H.
 \end{align}
 By taking the (2,1)+(1,2) components of both sides, it follows that
	\begin{align}\label{IMFAFB2} 
     dF_A^{(1,1)_\pm}=\delta^* \eta_\pm, \quad \pm dF_B^{(1,1)_\pm}=\delta^* \eta_\pm, \quad \eta_\pm=iH^{(2,1)_\pm}-iH^{(1,2)_\pm}. \end{align}
 Using \eqref{IMFAFB} and \eqref{IMFAFB2}, together with the following equation for $\mu = \mu_{F_{A}^{(1,1)_{\pm}}}$ or $\mu_{F_{B}^{(1,1)_{\pm}}}$ and $\eta = \eta_{\pm}$ proven in \cite[Prop. 3.5 (ii)]{MR2068969}: for all $a, b \in \Gamma(A)$,
    \[
    \mu ([a,b] )= \L_{\rho(a)} \mu(b) - \L_{\rho(b)}(\mu(a)) + d \langle (\mu(a),\rho_A(b)\rangle - \iota_{\rho_A(a)\wedge \rho_A(b)} \eta; 
    \]
    we obtain the identities  
	\begin{align}\label{IMEQ}
    \pm d \omega_\pm (\rho_A(\cdot), \rho_A(\cdot), \cdot)=\eta_\pm(\rho_A( \cdot), \rho_A(\cdot), \cdot),\quad d \omega_\pm (\rho_B(\cdot), \rho_B( \cdot), \cdot)=\pm \eta_\pm(\rho_B(\cdot), \rho_B(\cdot), \cdot).
    \end{align}
    For example, for $\mu = \mu_{F_A^{(1,1)_{+}}}$, we can compute for all $a, b \in \Gamma(A), X \in \Gamma(TM)$,
    \begin{align*}
    \omega_{+}(\rho_A[a,b],X) = \iota_X \mu_{F_A^{(1,1)_{+}}}([a,b]) &= \iota_X \L_{\rho_A(a)} \omega_+(\rho_A(b))-\iota_X\L_{\rho_A(b)}\omega_+(\rho_A(a)) \\
    &+\iota_Xd\omega_+(\rho_A(a),\rho_A(b))-\iota_X\iota_{\rho_A(a)\wedge \rho_A(b)} \eta_+ \\
    &= \iota_X \iota_{[\rho_A(a),\rho_A(b)]}\omega_+ +\iota_Xd\omega_+(\rho_A(a),\rho_A(b))-\iota_X\iota_{\rho_A(a)\wedge \rho_A(b)} \eta_+,
    \end{align*}
    which shows one of the identities in \eqref{IMEQ}. The other one may be proven similarly.   
    
	Since the $A$-orbits and the $B$-orbits are transverse, \eqref{IMEQ} then implies that 
	$d \omega_\pm =\pm \eta_\pm$. On the other hand, using this transversality property and taking (3,0)+(0,3) components of both sides of \eqref{dcFAFB}, we get that $H$ is of type $(2,1) +(1,2)$ and therefore $d^c \omega_\pm =\pm H$, as needed.  

 Finally, note that $g$ is given by $\omega_\pm$ and these only depend on $\pi_A$,  $\pi_B$ and $I_\pm$. But these are the Poisson structures inherited by the bisections $\lambda_Z\subset Z$ and $\lambda_Z\subset W$ from the imaginary parts of the background holomorphic Poisson bivectors. As a consequence, $g$ only depends on $\lambda_Z,\lambda_W$, which are the projections of $\lambda=\Lambda_{\text{Ver}}\cap \Lambda_{\text{Hor}}$ to $Z$ and $W$, where $\Lambda_{\text{Ver}}, \Lambda_{\text{Hor}}\subset \mathcal{C}$ are the bisections in the central bimodule that are part of our chosen real trivialization.
\end{proof}

\subsection{The generalized Kähler potential and Hamiltonian deformations}\label{subsec: GK deform} 
Based on \autoref{thm: key Lagrangian in the central brane}, a Lagrangian bisection of the real symplectic core determines a real trivialization of the $(1,1)$-morphism \eqref{diagram:(1,1)morphism} of a generalized K\"ahler  manifold, as long as our symplectic double groupoids are of Lu-Weinstein type \eqref{eq:luweidougro}. Moreover, \autoref{thm: reconstruction} implies that such a trivialization induces a pseudo generalized K\"ahler metric. We will say that such a Lagrangian bisection is {\em positive} if the corresponding metric is positive definite.
\begin{definition}[Generalized Kähler potential] Consider a positive Lagrangian bisection $\lambda$ in the real symplectic core $\mathcal{S}$ of the integrating $(1,1)$-morphism \eqref{diagram:(1,1)morphism} of a generalized K\"ahler  manifold. Introduce a tubular neighborhood symplectic embedding $\iota:T^*L\hookrightarrow \mathcal{S}$ for an auxiliary Lagrangian submanifold $L\subset \mathcal{S}$. We say that a {\em generalized Kähler potential} for $\lambda$ is a function $f\in C^\infty(L)$ such that $\Gr(df) =\lambda\cap \iota(T^*L)$.
\end{definition} 


By analogy with the K\"ahler case, in which the cohomology class $[F]$ of the 
K\"ahler form classifies the Morita bimodule up to isomorphism (See Example~\ref{ex: kahler case}), we propose the following definition of generalized K\"ahler  classes.

\begin{definition}[Generalized Kähler class] The generalized K\"ahler  class of a generalized K\"ahler manifold is the isomorphism class of an integrating $(1,1)$-morphism as in \eqref{diagram:(1,1)morphism}. 
\end{definition}  
Now that we have separated the holomorphic degrees of freedom from the smooth ones, we may fix the $(1,1)$-morphism and consider deformations only of the Lagrangian bisection of $\mathcal{S}$. In this way, 
\autoref{thm: reconstruction} may be applied to the construction and deformation of generalized Kähler metrics. The main idea is to deform the bisection of $\mathcal{S}$ by composing it with Lagrangian bisections of the groupoids that act upon it. Many of the explicit deformations of generalized K\"ahler metrics, such as those in~\cite{MR2217300, MR2681704,MR2287917}, are examples of this more general construction.

Suppose that we begin with the situation of \autoref{thm: key Lagrangian in the central brane}, in which a bisection $\lambda$ of the real symplectic core $\mathcal{S}$ is given by bisections $\lambda_Z, \lambda_W$ of $Z$ and $W$. We now act on $\lambda$ by a real Lagrangian section of the symplectic core groupoid of $D_+$.

A bisection of the core groupoid $K_{D_+}$ may be described as follows:  
\begin{equation}
\lambda_+ = \{\lambda_+(p)=(\uu(p), \uu(p), \lambda_B(\phi_A(p)), \lambda_A(p))\ :\ p\in X_+\},  \label{eq:non st core}
\end{equation}
where $\lambda_A, \lambda_B$ are $C^{\infty}$ bisections of $G_{\A_+}$, $G_{\B_+}$ respectively such that $\phi_B|_{X_+} \circ \phi_A|_{X_+} = \id_{X+}$, where $\phi_A, \phi_B$ are the groupoid isomorphisms induced by $\lambda_A, \lambda_B$.  Together these bisections define a bisection $\lambda_+$ of the core $K_{D_+}$.

As a consequence of \autoref{prop: morita core groupoids}, $\lambda_{+}$ acts on $\lambda$ as shown in Figure~\ref{coreact}, producing the composite bisections $\lambda_Z' = \lambda_A *\lambda_Z$, $\lambda_W' = \lambda_W * \lambda_B$, which determines a new bisection $\lambda'$ of $\mathcal{S}$.
 \begin{figure}[H]
 \begin{equation*}
  \qquad\adjustbox{scale=0.65}{\begin{tikzpicture}[scale=0.9,every node/.style={transform shape}]
\clip(-0.5,-0.5) rectangle (6.5,6.5);
  \draw[fill=black!10] (0,0) rectangle (6,6);
  \draw[step=2.0,black,thin] (0, 0) grid (6.0, 6.0);
  \path
    (1, 1) node {$\ii_A\ii_B\lambda_+(p)$}
    (3, 1) node {$\uu(\lambda_W(p)^{-1})$}
    (5, 1) node {$\uu_A\uu(x)$}
    (1, 3) node {$\uu\lambda_Z(x)^{-1}$}
    (3,3) node {$\lambda(x)$}
   (5, 3) node {$\uu\lambda_Z(x)$}
(1, 5) node {$\uu_A\uu(x)$}
(3, 5) node {$\uu\lambda_W(p)$}
(5, 5) node {$\lambda_+(p)$};
 \end{tikzpicture}}
\hspace{3em}
\adjustbox{scale=0.95}{
\begin{tikzpicture}[line cap=round,line join=round,>=stealth,x=1cm,y=1cm,scale=0.4]
		\clip(-2.7,-5.12) rectangle (11.7,5.3);
  \draw[fill=black!10] (0.44,-3.78) rectangle (6.86,2.64);
		\draw [<-,line width=0.9pt] (0.44,2.64) -- (6.86,2.64);
		\draw [<-,line width=0.9pt] (6.86,2.64) -- (6.86,-3.78);
		\draw [->,line width=0.9pt] (6.86,-3.78) -- (0.44,-3.78);
		\draw [->,line width=0.9pt] (0.44,-3.78) -- (0.44,2.64);
		\begin{scriptsize}
			
			\draw[color=black] (-0.2,-3.55) node {$p$};
			
			\draw[color=black] (7.3,-3.55) node {$x$};
			
			\draw[color=black] (7.3,3.0) node {$p$};
			
			\draw[color=black] (-0.2,3.07) node {$x$};
			\draw[color=black] (3.7,3.27) node {$\lambda_W(p)$};
			\draw[color=black] (8.2,-0.4) node {$\lambda_Z(x)$};
			\draw[color=black] (3.68,-4.39) node {$\lambda_W(p)^{-1}$};
			\draw[color=black] (-1.1,-0.4) node {$\lambda_Z(x)^{-1}$};
			\draw[color=black] (3.55,-0.5) node {$\lambda(x)$};
		\end{scriptsize}
	\end{tikzpicture} 
 \hspace{-1em}
  \begin{tikzpicture}[line cap=round,line join=round,>=stealth,x=1cm,y=1cm,scale=0.4]
		\clip(-3.5,-5.12) rectangle (9.7,5.3);
  \draw[fill=black!10] (0.44,-3.78) rectangle (6.86,2.64);
		\draw [<-,line width=0.9pt] (0.44,2.64) -- (6.86,2.64);
		\draw [<-,line width=0.9pt] (6.86,2.64) -- (6.86,-3.78);
		\draw [->,line width=0.9pt] (6.86,-3.78) -- (0.44,-3.78);
		\draw [->,line width=0.9pt] (0.44,-3.78) -- (0.44,2.64);
		\begin{scriptsize}
			
			\draw[color=black] (-0.2,-3.55) node {$p$};
			
			\draw[color=black] (7.3,-3.55) node {$p$};
			
			\draw[color=black] (7.3,3.0) node {$q$};
			
			\draw[color=black] (-0.2,3.07) node {$p$};
			\draw[color=black] (3.7,3.27) node {$\lambda_B(q)$};
			\draw[color=black] (8.2,-0.4) node {$\lambda_A(p)$};
			\draw[color=black] (3.68,-4.39) node {$\uu(p)$};
			\draw[color=black] (-1.1,-0.4) node {$\uu(p)$};
			\draw[color=black] (3.7,-0.5) node {$\lambda_{+}(p)$};
		\end{scriptsize}
	\end{tikzpicture}}
\end{equation*}
\caption{Action of core groupoid on symplectic core}\label{coreact}
\end{figure}

 The proof of \autoref{thm: key Lagrangian in the central brane} then immediately implies that the resulting bisection $\lambda'$ is Lagrangian, yielding the following.
\begin{lemma}
\label{prop:def lag real core} The bisection $\lambda'$ is Lagrangian if $\lambda_{+}$ is a real Lagrangian bisection of $K_{D_+}$. The multiplicative Lagrangian bisections induced by $\lambda'$ by \autoref{thm: key Lagrangian in the central brane} are then given by  $\Lambda_Z'= \Lambda_A\ast \Lambda_Z$ and $\Lambda_W'=\Lambda_W\ast \Lambda_B$, where
\begin{align}
  \begin{aligned} 
    \Lambda_A(b)&=(\lambda_A(\ttt(b)),b,\phi_B^{-1}(b),\lambda_A(\sss(b))), \quad b\in G_{\B_+};  \\  
    \Lambda_B(a)&=(\phi_A^{-1}(a),\lambda_B(\sss(a)),\lambda_B(\ttt(a)),a), \quad a\in G_{\A_+}. 
  \end{aligned}   \label{eq:def branes}  
\end{align} \end{lemma}

  Based on the setup described above, we obtain a description of a group acting transitively on lagrangian bisections of the real symplectic core, in which the space of generalized K\"ahler metrics in a fixed generalized K\"ahler class sits as an open set.  Summarizing the above, we have the following. 
  \begin{theorem}
The space of Lagrangian bisections of the real symplectic core $\mathcal{S}$ is a bitorsor for the groups of real Lagrangian bisections of $K_{D_\pm}$, via the action above.
  \end{theorem} 
Note that if we use a holomorphic Lagrangian bisection $\lambda_+$ of $K_{D_+}$, then $\Lambda_A$ and $\Lambda_B$ are holomorphic Lagrangian, and as result, the generalized K\"ahler metric determined by $\lambda$ will not be affected by composition with $\lambda_+$.  This is analogous to the well-known ambiguity in the usual K\"ahler potential, whereby K\"ahler potentials differing by the real part of a holomorphic function give rise to the same metric.

\begin{example}[Hamiltonian deformations of generalized K\"ahler metrics] A  plentiful source of real Lagrangian bisections of the core of $D_+$ is the following:  begin with a real-valued function $f \in C^{\infty}(X_+, \R)$, pull it back to the double groupoid $D_+$ as follows:
\begin{align*}
\ttt^*f\in C^\infty(G_{\A_+},\R),\qquad
-\sss^*f\in C^\infty(G_{\B_+},\R),\qquad
-\sss_A^*\ttt^* f\in C^\infty(D_+,\R).
\end{align*}
The corresponding Hamiltonian vector fields $X_A, X_B, X_D$ are then related as follows:
\begin{align*}
\begin{split}
(\sss_A)_* X_D &= X_A\\
 (\ttt_B)_* X_D &= X_B
\end{split}
\begin{split}
(\ttt_A)_* X_D &= 0\\
 (\sss_B)_* X_D &= 0.
\end{split}
\end{align*}
As a result, the simultaneous Hamiltonian flow preserves the core $K_+$, taking its identity bisection to a real Lagrangian bisection $\lambda_+(t)$ as in \eqref{eq:non st core}.

By composition as in Lemma~\ref{prop:def lag real core}, and by flowing for sufficiently small times (to preserve nondegeneracy and/or positivity), such  Hamiltonian bisections of $K_{D_+}$ act on the space of generalized  K\"ahler metrics in a fixed class.  In the following example, we demonstrate this more explicitly in a special case. 
\end{example}

\begin{example}[Hamiltonian deformations in the commuting type case]\label{ex:ham def com type} We continue the discussion about commuting type generalized K\"ahler  manifolds started in \S\ref{subsec: comm}. In what follows, we combine \autoref{The induced isomorphisms} with \autoref{prop:def lag real core}. Let $f$ be a real-valued function on $M_1 \times M_2$. In order to describe the Hamiltonian deformation of the generalized K\"ahler  metric via $f$, it suffices to examine how $F_A\in\Omega^2(G_A)$ and $F_B\in\Omega^2(G_B)$ are changed. 
	Notice that $F_B$ is the pullback of $\Omega_{\mathcal{Z}}$ by the map 
	\[
	\Lambda_{Z}:P(M_1) \times T^*M_2 \to \mathcal{Z}, \qquad    (x,y, u) \mapsto ( 0_x, 0_y, u, u). 
	\]
	For the first two components of $F_B$, the Hamiltonian deformation determined by $f$ changes their pullback $(F_1,-F_1)$ by adding $(td_1d_1^c f,-td_1d_1^c f)$, where $d_1, d_1^c$ are, respectively, the de Rham and the $I_1$-twisted de Rham differentials on $M_1$; on the other hand, $-(\omega_2)_I$ is the third component of $F_B$ and is unchanged by the Hamiltonian deformation. A similar discussion applies to $F_A$. Thus the Hamiltonian flow of $f$ deforms $F_A$ and $F_B$ to 
	\begin{align*}
	    F_A^t&=(-(\omega_1)_I,(F_2 + td_2d_2^c f,-(F_2+td_2d_2^cf))) \in \Omega^2(T^*M_1\times P(M_2))\\
     F_B^t&= ((F_1 + td_1d_1^cf, -(F_1+td_1d_1^cf)), -(\omega_2)_I)\in \Omega^2(P(M_1)\times T^*M_2).
	\end{align*}	
	The real Poisson structures $\pi_A^t$ and $\pi_B^t$ induced by $F_A^t$ and $F_B^t$ can be described simply by inverting the symplectic forms $F_1 + td_1d_1^c f$ and $F_2+ td_2d_2^cf$. Using \eqref{eq: from poisson to hermtian forms}, we derive the following expressions for the deformed Hermitian forms $\omega_{\pm}^t$ in terms of the original ones $\omega_\pm$:
	\[
	\omega_{\pm}^t = \omega_{\pm} + t (d_2d_2^cf \pm d_1d_1^cf),
	\]
	thus recovering the deformation found in~\cite[Proposition 2]{MR2287917}.
\end{example}
\begin{remark}
	We refer the reader to \cite[Section 8]{MR4466669} for Hamiltonian deformations of generalized K\"ahler metrics of symplectic type.
\end{remark}

\section{Generalized K\"ahler structures on compact Lie groups}\label{sec: lie groups}
We illustrate our main results by integrating a well-known family of generalized K\"ahler structures carried by the even-dimensional semisimple compact Lie groups.  To do this, we must construct holomorphic symplectic manifolds related by real structures as in~\autoref{thm: the (1,1) morphism of a GK}, together with a real Lagrangian bisection as in~\autoref{thm: reconstruction} which determines the generalized K\"ahler metric. We show that for such compact Lie groups $K$, the required holomorphic symplectic manifolds may be described as moduli spaces of flat  $G$-connections on an annulus with decorated boundary, where $G$ is the complexification of $K$.  

\subsection{Generalized K\"ahler structures on compact Lie groups and Manin triples}\label{subsec:intgkcomlie} 

Let $K$ be an even-dimensional, connected, simply-connected Lie group which is a product of a compact semisimple factor and an abelian factor, and let $G$ be its complexification. Take a bi-invariant metric $g$ on $K$ determined by an Ad-invariant metric $s$ on its Lie algebra $\mathfrak{k}$, and suppose that $I: \mathfrak{k} \rightarrow \mathfrak{k}$ is an orthogonal complex structure whose $+i$-eigenspace $\mathfrak{g}_+\subset \mathfrak{g}$ is a Lie subalgebra of the Lie algebra of $G$. Let  $\mathfrak{g}_-=\overline{\mathfrak{g}_+ } $ denote the $-i$-eigenspace of $I$. Such complex structures are determined as follows: take a triangular decomposition $\mathfrak{g}=\mathfrak{n}_+\oplus \mathfrak{h}\oplus \mathfrak{n}_-$, that is, $\mathfrak{h}=\mathfrak{t}_{\C}$ is the complexification of the Lie algebra $\mathfrak{t}\subset \mathfrak{k}$ of the product of a maximal torus with the abelian factor, and $\mathfrak{n}_+,\mathfrak{n}_-$ are the sums of corresponding positive and negative roots. Then, by choosing an $s$-orthogonal complex structure on $\mathfrak{t}$ with corresponding decomposition $\mathfrak{h}=\mathfrak{t}_+\oplus \mathfrak{t}_-$, we can define $\mathfrak{g}_+=\mathfrak{n}_+\oplus \mathfrak{t}_+$ and $\mathfrak{g}_-= \mathfrak{t}_-\oplus\mathfrak{n}_-$ as the $\pm i$-eigenspaces of a complex structure $I:\mathfrak{k}\to \mathfrak{k}$, see \cite[Remark 3.4]{MR2642360}.   Consequently, $(\mathfrak{g} ,\mathfrak{g}_+ ,{\mathfrak{g}_- }) $ is a Manin triple of complex Lie algebras with respect to the complexified pairing $s_{\mathbb{C}}$ on $\g$. In the rest of this section, we denote by $s_R$ and $s_I$ the real and imaginary parts of $s_{\mathbb{C}}$ and we denote by $G_+,G_- \subset G$ the connected subgroups integrating $\mathfrak{g}_+ $ and ${\mathfrak{g}_- } $, respectively.

Our main object of study in this section is the generalized K\"ahler structure on $K$ determined by the above choices and given (as a bihermitian structure) by 
\begin{align}  (g,I^r,I^l),    \label{eq:gkcomlie} \end{align} 
where $I^l,I^r$ are the left and right invariant extensions of $I$, see \cite[Example 2.25]{MR3232003}. 

In addition to the complex Manin triple $(\mathfrak{g} ,\mathfrak{g}_+ ,{\mathfrak{g}_- })$ described above, we obtain two other real Manin triples, $(\mathfrak{g} ,\mathfrak{k} , \mathfrak{g}_\pm)$, where the inner product is the imaginary part $s_I$. Throughout this section, we assume that the multiplication maps $G_\pm \times K\to G$ are diffeomorphisms, a fact guaranteed if $K$ is simply connected and semisimple, see \cite[Prop. 2.43]{luphd}. This also holds for (the universal cover of) the group $SU(2)\times U(1)$, which we will use below to explicitly illustrate our construction. A consequence of the identifications $G=G_\pm \times K$ is that we obtain \emph{dressing actions} of $G\pm$ on $K$,  as follows: 
\begin{align}  ak={}^aka^k, \quad \quad bk={}^bkb^k, \quad  a\in G_-,\, b\in G_+,\,  k \in K ; \label{eq:dres1} \end{align}  
where ${}^ak,{}^bk\in K$ denote the left dressing action of $a$ (resp. $b$) on $k$ and $a^k\in G_-$, $b^k\in G_-$, see \cite[Prop. 2.41]{luphd}.
\begin{example}\label{ex:hopf surf 0} Consider $SU(2)\times U(1)$ with its complexification $SL_2(\mathbb{C}) \times \mathbb{C}^*$. We will pass to the universal covers ${G} =SL_2(\mathbb{C}) \times \mathbb{C} \rightarrow SL_2(\mathbb{C}) \times \mathbb{C}^*$ and $K=SU(2)\times \mathbb{R}\rightarrow SU(2)\times U(1)$ for simplicity. We use the pairing on $\mathfrak{g} $ given by $s_\mathbb{C}((X,u),(Y,v))=-\frac{1}{2} \text{tr} (XY)-uv $. Consider the basis of $\mathfrak{su}(2)$:
\[ u_1=\begin{bmatrix} 0 & i \\ i & 0 \end{bmatrix}, \quad u_2=\begin{bmatrix} 0 & -1 \\ 1 & 0 \end{bmatrix},\quad u_3=\begin{bmatrix} i & 0 \\ 0 & -i \end{bmatrix} \]
and $v=i\in i\mathbb{R}=\mathfrak{u}(1)$. So $\{u_i,v\}_{i=1,2,3}$ is a basis of $\mathfrak{k}$ and we can define a complex structure $I:\mathfrak{k} \rightarrow \mathfrak{k}$ by putting $I(u_1)=u_2$ and $I(v)=u_3$. Then the corresponding Lagrangian subgroups of $G$ are
\[ G_-=\left\{\left.\left(\begin{bmatrix} e^{iz} & w \\ 0 & e^{-iz} \end{bmatrix},z\right)\right|w,z\in \mathbb{C}  \right\}, \quad G_+=\left\{\left.\left(\begin{bmatrix} e^{-iz} & 0 \\ w & e^{iz} \end{bmatrix},z\right)\right|w,z\in \mathbb{C}  \right\}.\vspace{-6mm}\] 
\end{example}
	In the remainder of this section, we explicitly construct the integration of the above generalized K\"ahler structure \eqref{eq:gkcomlie}, using the theory of quasi-Hamiltonian moduli spaces of $G$-connections on surfaces.
	
	
	\subsection{Quasi-Hamiltonian structures on spaces of groupoid representations} 
	
	We begin with a brief review of the construction of symplectic moduli spaces of flat bundles over decorated surfaces, following \cite{MR3330247,MR3001637}. Let $\Sigma$ be a compact connected and oriented surface, equipped with a finite set of marked points $V\subset \partial \Sigma$ which intersects every boundary component; in this situation, we say that $(\Sigma,V)$ is a {\em marked surface}. Our main example will use an annulus with four marked points on each boundary circle; see Fig.~\ref{fig:doucomlie}. The primary object of interest is the space $M$ of groupoid morphisms 
    \[
    M=\hom(\Pi_1(\Sigma,V),G),
    \]
    where $\Pi_1(\Sigma,V)$ is the fundamental groupoid of $\Sigma $ based at $V$. The manifold $M$ parametrizes isomorphism classes of flat $G$-bundles over $\Sigma$ endowed with a trivialization over $V$. The key insight which makes such spaces relevant to the integration problem is that they carry a {\em quasi-Hamiltonian} structure \cite{MR1638045}, which gives rise to symplectic quotient spaces via a moment map reduction procedure.  The appropriate moment map is described as follows. 

	The orientation on $\Sigma$ allows us to define a {\em boundary graph $\Gamma $} in which $V$ is the set of vertices; the set of edges $E$ consists of the oriented boundary arcs in $\partial \Sigma$ between pairs of points in $V$.  This boundary graph then determines a group-valued moment map
    \[
    \mu: M\to G^E,
    \]
    sending a representation $\rho$ to $\prod_{e\in E} \rho(e)$.  This map is equivariant for the action of the group $G^V$ of gauge transformations, which acts as follows.   A gauge transformation $g\in G^V$ acts on a representation $\rho$ via 
        \begin{align} (g\cdot \rho)(\gamma)=g_{\mathtt{t}(\gamma)}\rho(\gamma)g_{\mathtt{s}(\gamma)}^{-1}\quad  \forall\ \rho \in M, \ \gamma\in \Pi_1(\Sigma,V),\label{eq:gauact} \end{align}   
   and on an element $h\in G^E$ via $(g\cdot h)(e) = g_{\mathtt{t}(e)}h(e) g_{\mathtt{s}(e)}^{-1}$.    
The manifold $M$, which inherits a complex structure from the complex Lie group $G$, is equipped with a holomorphic 2-form $\Omega_{\Sigma,V}$ known as a quasi-symplectic structure, which enables the procedure of quasi-Hamiltonian reduction developed in \cite{MR1638045}.  In the next subsection, we use this procedure to produce the holomorphic symplectic manifolds we require.   

To write the quasi-symplectic 2-form $\Omega_{\Sigma,V}$ explicitly, we choose a triangulation of $\Sigma$ which includes $V$ and $E$, but which may contain new vertices and edges in the interior.  Once we know the quasi-symplectic form associated to a single triangle, we may assemble these through a natural reduction procedure as described in \cite[\S 1.2]{MR3330247}.  For this reason, we begin with the example of the triangle, following \cite{MR3330247}. 
	\begin{example} A disc with $n$ marked points on its boundary will be denoted by $P_n$ in the following examples. In the case of a triangle $P_3$, we use the homotopy classes of two positively oriented boundary edges of $P_3$ to make the identification $\hom(\Pi_1(P_3),G)\cong G^2$. 
    \[\begin{tikzpicture}[line cap=round,line join=round,>=stealth,x=1cm,y=1cm,scale=0.3]
			\clip(-3.3,-4.73) rectangle (11.7,3.6);
\coordinate (a) at (0.44,2.64);
 \coordinate (b) at (6.86,2.64);
 \coordinate (c) at (6.86,-3.78);
 \coordinate (d) at (0.44,-3.78);
   \draw[fill=black!10] (a)--(b)--(c)--cycle;
			\draw [<-,line width=0.9pt] (0.44,2.64) -- (6.86,2.64);
			\draw [<-,line width=0.9pt] (6.86,2.64) -- (6.86,-3.78);
			\draw [->,line width=0.9pt] (6.86,-3.78) -- (0.44,2.64);
			\begin{scriptsize}
				
				
				
				
				\draw[color=black] (3.7,3.27) node {$g_2$};
				\draw[color=black] (7.8,-0.4) node {$g_1$};
				\draw[color=black] (4.4,0.2) node {$g_2g_1$};
			\end{scriptsize}
		\end{tikzpicture} \]
    The quasi-Hamiltonian $G^3$-space structure on this space has momentum map $\mu:G^2 \rightarrow G^3$ given by  $\mu(g_2,g_1)=(g_2,g_1,g_1^{-1}g_2^{-1})$, and quasi-symplectic 2-form  
		\begin{equation}  
			\Omega_{P_3}|_{(g_1,g_2)}=\frac{1}{2}s_{\mathbb{C} }\left( g_2^*\theta^l,g_1^*\theta^r \right) \quad \forall g_1,g_2\in G ;\label{eq:polwie} 
		\end{equation} 
		where $\theta^l,\theta^r$ are the Maurer-Cartan 1-forms on $G$ and the expressions $g_1^*, g_2^*$ denote pullback by the two projections $G^2\to G$.   
        \end{example}
	
	\begin{example}\label{ex:2forsqu} Consider a square $P_4$ with four marked points and fundamental groupoid generators as shown:
    \[\begin{tikzpicture}[line cap=round,line join=round,>=stealth,x=1cm,y=1cm,scale=0.3]
			\clip(-3.3,-4.73) rectangle (11.7,3.6);
\coordinate (a) at (0.44,2.64);
 \coordinate (b) at (6.86,2.64);
 \coordinate (c) at (6.86,-3.78);
 \coordinate (d) at (0.44,-3.78);
   \draw[fill=black!10] (a)--(b)--(c)--(d)--cycle;
			\draw [<-,line width=0.9pt] (0.44,2.64) -- (6.86,2.64);
			\draw [<-,line width=0.9pt] (6.86,2.64) -- (6.86,-3.78);
			\draw [->,line width=0.9pt] (6.86,-3.78) -- (0.44,-3.78);
			\draw [->,line width=0.9pt] (0.44,-3.78) -- (0.44,2.64);
			\draw [dashed,line width=0.9pt] (6.86,-3.78) -- (0.44,2.64);
			\begin{scriptsize}
				
				
				
				
				\draw[color=black] (3.7,3.27) node {$g_3$};
				\draw[color=black] (8.2,-0.4) node {$g_2$};
				\draw[color=black] (3.68,-4.39) node {$g_1$};
				\draw[color=black] (-1.1,-0.4) node {$g_4$};
				\draw[color=black] (4.4,0.2) node {$g_3g_2$};
			\end{scriptsize}
		\end{tikzpicture} \]
        The quasi-symplectic 2-form  on $\hom(\Pi_1(P_4),G)$ is given as follows: we use the identification $\hom(\Pi_1(P_4),G)\cong G^3$ with coordinates $(g_1,g_2,g_3)$ corresponding to the generators shown above, so that $g_4=g_3g_2g_1^{-1}$.  The 2-form is then  
		\begin{align*} &\Omega_{P_4}|_{(g_1,g_2,g_3)}=\frac{1}{2}\left( s_{\mathbb{C} }(g^*_4\theta^l,g_1^*\theta^r)-s_{\mathbb{C} }(g_3^*\theta^l,g_2^*\theta^r) \right),
			  \end{align*}
		where we have used the decomposition of $P_4$ into triangles shown above to pull back the 2-form $\Omega_{P_3}$ in the previous example \eqref{eq:polwie}. 
  \end{example}
	\begin{example}\label{ex:2forluwei} If $(\Sigma,V)$ is an annulus with four marked points on each boundary component as in Fig.~\ref{fig:doucomlie}, we obtain its quasi-symplectic 2-form  $\Omega_{\Sigma,V}$ by cutting it along the dashed lines indicated in the figure and summing the forms obtained in the previous example:
		\begin{align} \Omega_{\Sigma,V}|_{(g_i,k_i)_{i=1\dots 4}}=\Omega_{P_4}|_{(k_3,g_4,k_4)}+\Omega_{P_4}|_{(k_2,g_3,k_3)}-\Omega_{P_4}|_{(k_1,g_1,k_4)}-\Omega_{P_4}|_{(k_2,g_2,k_1)}, \label{eq:luweicomlie} \end{align}
		where we use the generators $(g_i,k_i)_{i=1\dots 4}$ displayed in Fig.~\ref{fig:doucomlie} to describe the representation. 
	\end{example}  
 \begin{figure}[H]  \centering\begin{tikzpicture}[line cap=round,line join=round,>=stealth,x=1cm,y=1cm, scale=0.45]
			\clip(-5.171814835019457,6.0) rectangle (11.18595630529845,14.9);
   \coordinate (a) at (0,7);
 \coordinate (b) at (7,14);
 \coordinate (c) at (2,9);
 \coordinate (d) at (5,12);
   \draw[fill=black!10] (a) rectangle (b);
   \draw[fill=white] (c) rectangle (d);
			\draw [line width=1pt, <-] (0,7)-- (7,7);
			\draw [line width=1pt, ->] (7,7)-- (7,14);
			\draw [line width=1pt,->] (7,14)-- (0,14);
			\draw [line width=1pt, <-] (0,14)-- (0,7);
			\draw [line width=1pt,<-] (2,9)-- (5,9);
			\draw [line width=1pt,->] (5,9)-- (5,12);
			\draw [line width=1pt,->] (5,12)-- (2,12);
			\draw [line width=1pt,<-] (2,12)-- (2,9);
			\draw [dashed,line width=1pt,<-] (5,12)-- (7,14);
			\draw [dashed,line width=1pt,->] (7,7)-- (5,9);
			\draw [dashed,line width=1pt,->] (0,7)-- (2,9);
			\draw [dashed,line width=1pt,->] (0,14)-- (2,12);
			\begin{scriptsize}
				\draw [fill=black] (0,7) circle (1.7pt);
				\draw [fill=black] (7,7) circle (1.7pt);
				\draw[color=black] (3.548843996618713,6.4) node {$L_1$};
				\draw [fill=black] (7,14) circle (1.7pt);
				\draw[color=black] (7.6,10.6) node {$L_2$};
				\draw [fill=black] (0,14) circle (1.7pt);
				\draw[color=black] (3.548843996618713,14.4) node {$L_3$};
				\draw[color=black] (-0.49,10.64464319092457) node {$L_4$};
				\draw [fill=black] (2,9) circle (1.7pt);
				\draw [fill=black] (5,9) circle (1.7pt);
				\draw[color=black] (3.548843996618713,8.4) node {$L_5$};
				\draw [fill=black] (5,12) circle (1.7pt);
				\draw[color=black] (5.6,10.5) node {$L_6$};
				\draw [fill=black] (2,12) circle (1.7pt);
				\draw[color=black] (3.548843996618713,12.4) node {$L_7$};
				\draw[color=black] (1.4,10.6) node {$L_8$};
			\end{scriptsize}
		\end{tikzpicture}
		\begin{tikzpicture}[line cap=round,line join=round,>=stealth,x=1cm,y=1cm, scale=0.45]
			\clip(-3.171814835019457,6.0) rectangle (11.18595630529845,14.9);
    \coordinate (a) at (0,7);
 \coordinate (b) at (7,14);
 \coordinate (c) at (2,9);
 \coordinate (d) at (5,12);
   \draw[fill=black!10] (a) rectangle (b);
   \draw[fill=white] (c) rectangle (d);
			\draw [line width=1pt, <-] (0,7)-- (7,7);
			\draw [line width=1pt, ->] (7,7)-- (7,14);
			\draw [line width=1pt,->] (7,14)-- (0,14);
			\draw [line width=1pt, <-] (0,14)-- (0,7);
			\draw [line width=1pt,<-] (2,9)-- (5,9);
			\draw [line width=1pt,->] (5,9)-- (5,12);
			\draw [line width=1pt,->] (5,12)-- (2,12);
			\draw [line width=1pt,<-] (2,12)-- (2,9);
			\draw [dashed,line width=1pt,<-] (5,12)-- (7,14);
			\draw [dashed,line width=1pt,->] (7,7)-- (5,9);
			\draw [dashed,line width=1pt,->] (0,7)-- (2,9);
			\draw [dashed,line width=1pt,->] (0,14)-- (2,12);
			\begin{scriptsize}
				\draw [fill=black] (0,7) circle (1.7pt);
				\draw [fill=black] (7,7) circle (1.7pt);
				\draw[color=black] (3.548843996618713,6.654390523370536) node {$g_2$};
				\draw [fill=black] (7,14) circle (1.7pt);
				\draw[color=black] (7.66,10.64464319092457) node {$g_3$};
				\draw [fill=black] (0,14) circle (1.7pt);
				\draw[color=black] (3.548843996618713,14.347115238172494) node {$g_4$};
				\draw[color=black] (-0.79,10.64464319092457) node {$g_1$};
				\draw [fill=black] (2,9) circle (1.7pt);
				\draw [fill=black] (5,9) circle (1.7pt);
				\draw [fill=black] (5,12) circle (1.7pt);
				\draw [fill=black] (2,12) circle (1.7pt);
				\draw[color=black] (6.4,12.7) node {$k_3$};
				\draw[color=black] (6.4,8.286302910004277) node {$k_2$};
				\draw[color=black] (0.8094579827389956,8.386302910004277) node {$k_1$};
				\draw[color=black] (0.9094579827389956,12.7) node {$k_4$};
			\end{scriptsize}
		\end{tikzpicture}
		\caption{Lagrangian boundary conditions and a choice of independent generators for $\Pi_1(\Sigma,V)$}\label{fig:doucomlie}  \end{figure}
  \subsection{Integration via decorated moduli spaces} 
 In the following subsections, we construct the $(1,1)$-morphisms, real structures and real trivializations which integrate the generalized K\"ahler structures introduced in \S\ref{subsec:intgkcomlie}.  The holomorphic symplectic double groupoids and Morita bimodules are all given as moduli spaces of $G$-bundles on the annulus with varying boundary conditions as in Fig.~\ref{fig:doucomlie}.  We make use of the notation and assumptions introduced in \S\ref{subsec:intgkcomlie}. Since there is an identification $G\cong K \times G_\pm$, we can produce the four complex structures $\pm I_\pm$ on $K$ as follows. We may view $K$ as a complex homogeneous space for $G$ in four different ways: as a quotient of $G$ by the right $G_\pm$-action $K\cong G/G_\pm$ or as the quotient of $G$ by the left $G_\pm$-action $K\cong G_\pm\backslash G$. In either case, there is a residual $G$-action which, when restricted to $K$, explains the fact that the quotient complex structure on $K$ is left-invariant in the first case and right-invariant in the second one. 	
 
 \subsubsection{Holomorphic symplectic $(1,1)$-morphisms from boundary conditions} 
 
 To produce a holomorphic symplectic manifold from the complex quasi-Hamiltonian $G^V$-space $\hom(\Pi_1(\Sigma,V),G)$ corresponding to a marked surface, we introduce {\em Lagrangian boundary conditions}, namely, we 
 choose a Lagrangian complex subgroup $L_e\subset G$ for each boundary edge $e\in E$.  In our case $L_e$ will be chosen to be either $G_+$ or $G_-$, depending on which space in the $(1,1)$-morphism is required (see \eqref{eq:dec}).  These Lagrangian boundary conditions define a subgroup $L=\prod_{e\in E} L_e$ of the codomain $G^E$ of the moment map $\mu$, and hence determine a subspace of representations which satisfy the boundary conditions:
\[
\mu^{-1}(L)\subset \hom(\Pi_1(\Sigma,V),G).
\]
The subgroup of the gauge group preserving this subspace is
\[H=\prod_{v\in V}H_v \subset G^V,\]
where $H_v$ is the identity component of the intersection of $L_e$ with $L_{e'}$, for edges $e,e'$ meeting at $v$.  
If the gauge action of $H$ is free and proper, then the quotient 
	\[ \mathfrak{M}_{L}(G)= \mu^{-1}(L) / H \]    
	is a holomorphic symplectic manifold with symplectic form given by the pullback of $\Omega_{\Sigma,V}$, the quasi-symplectic form  on $\hom(\Pi_1(\Sigma,V),G)$ described above (see \cite[Thm. 1.4]{MR3330247}).


To obtain holomorphic symplectic manifolds of the form $\mathfrak{M}_{L}(G)$ which carry compatible groupoid, double groupoid, or bimodule structures, we may exploit the fact that the moduli-theoretic construction is functorial in morphisms between surfaces with boundary conditions.  This technique, which we learned from \cite{MR2352135,Severa,MR3001637}, can be used to produce many interesting symplectic groupoids and double groupoids occurring in Lie theory (see \cite{MR4705016}), as well as the ones we need below.

The main idea is the following:	if a marked surface $(\Sigma,V)$ may be expressed as the gluing of two copies of the same surface $\Sigma=\Sigma_1\cup_S \Sigma_1$, where $(\Sigma_1,V_1)$ is a marked surface and $S\subset \partial \Sigma_1$ is a collection of boundary arcs that satisfies $V_1\cap S=\emptyset$, we apply the moduli construction to the two inclusions 
   \[ 
           (\Sigma_1,V_1)\rightrightarrows(\Sigma,V)
   \]
to obtain a Lie groupoid structure
    \[ \hom(\Pi_1(\Sigma,V) ,G)\rightrightarrows \hom(\Pi_1(\Sigma_1,V_1),G), \]
    in which the source and target maps are obtained by pullback of flat connections along the inclusions $\Sigma_1\to \Sigma$; the groupoid multiplication is given by gluing flat $G$-connections \cite[Thm. 3.5,\S 3.3.1]{MR4705016}. 
    
    If in addition, $\Sigma_1$ is obtained in the same manner, namely, $\Sigma_1=\Sigma_0\cup_T \Sigma_0$ for another collection of boundary arcs $T\subset\partial\Sigma_0$, then $\Sigma$ can alternatively be glued from two copies of the surface $\Sigma_1'=\Sigma_0\cup_{S\cap \partial \Sigma_0} \Sigma_0$, giving rise to the following diagram of surface inclusions
    \begin{equation*}
 \adjustbox{scale=0.9,center}{%
 \begin{tikzcd}
				(\Sigma_0,V_0) \arrow[d, shift right] \arrow[d, shift left] \arrow[r,  shift left] \arrow[r, shift right] &   (\Sigma_1,V_1) \arrow[d,  shift right] \arrow[d, shift left] \\
				{(\Sigma'_1,V'_1)} \arrow[r,  shift right] \arrow[r, shift left]                                                                                                                                                 &   {(\Sigma,V)}                                                                                            
			\end{tikzcd}   }
 		\end{equation*} 
    Taking moduli spaces, the above diagram gives rise to a double groupoid structure on $\hom(\Pi_1(\Sigma,V) ,G)$ with sides $\hom(\Pi_1(\Sigma,V) ,G)$ and $\hom(\Pi_1(\Sigma'_1,V_1') ,G)$, see \cite[Corollary 4.9]{MR4705016} and its proof for a detailed discussion.

    We now apply this method to the oriented annulus $(\Sigma,V)$ equipped with four marked points on each boundary component. Note that $\Sigma$ may be viewed as $\Sigma_1\cup_S \Sigma_1$, where $(\Sigma_1,V_1)$ is the square (disc with four marked points) and $S$ consists of two boundary arcs. In addition, $\Sigma_1=\Sigma_0\cup_T \Sigma_0$, where $(\Sigma_0,V_0)$ is the disc with two marked points and $T$ consists of a single boundary arc. In this case, $ \Sigma_1'$ is homeomorphic to $\Sigma_1$, so the situation simplifies and the double groupoid structure on $\hom(\Pi_1(\Sigma,V) ,G)$ may be explicitly described as follows.
    \begin{lemma} The diagram below indicates the source and target maps of a double groupoid structure \begin{equation}
 		  \begin{tikzcd}[scale cd=.85]
				{\hom(\Pi_1(\Sigma,V),G)} \arrow[dd, "{[g_i,k_i]\mapsto [k_1,g_2,k_2]}"', shift right] \arrow[dd, "{[g_i,k_i]\mapsto [k_4,g_4,k_3]}", shift left] \arrow[rr, "{[g_i,k_i]\mapsto [k_4,g_1,k_1]}", shift left] \arrow[rr, "{[g_i,k_i]\mapsto [k_3,g_3,k_2]}"', shift right] &  & {\hom(\Pi_1(\Sigma_1,V_1),G)} \arrow[dd, "{[x,g,y]\mapsto y}"', shift right] \arrow[dd, "{[x,g,y]\mapsto x}", shift left] \\
				&  &                                                                                                                           \\
				{\hom(\Pi_1(\Sigma_1,V_1),G)} \arrow[rr, "{[x,g,y]\mapsto y}"', shift right] \arrow[rr, "{[x,g,y]\mapsto x}", shift left]                                                                                                                                                 &  & {\hom(\Pi_1(\Sigma_0,V_0),G)}                                                                                            
			\end{tikzcd}   \label{eq:dou gro modspa}
 		\end{equation} 
			with the generators displayed in Fig.~\ref{fig:doucomlie} and with multiplication given by cutting and gluing surfaces, as illustrated in Fig.~\ref{fig:mul via cut}.    
    \end{lemma}
    \begin{proof} The groupoid structure on $\hom(\Pi_1(\Sigma_1,V_1),G) \rightrightarrows \hom(\Pi_1(\Sigma_0,V_0),G)$ is given by the multiplication $(x,g,y)\circ (y,h,z)=(x,gh,z)$. The double groupoid structure is the (generalized) Lu-Weinstein double groupoid \eqref{eq:luweidougro} corresponding to both side groupoids being equal to $\hom(\Pi_1(\Sigma_1,V_1),G) \rightrightarrows \hom(\Pi_1(\Sigma_0,V_0),G)$, see \autoref{rem:genluwei}.  
    \end{proof}   
 \begin{figure}[H] \centering\begin{tikzpicture}[line cap=round,line join=round,>=stealth,x=1cm,y=1cm, scale=0.45]
			\clip(-1.171814835019457,6.0) rectangle (8.18595630529845,14.9);
   \coordinate (a) at (0,7);
 \coordinate (b) at (7,14);
 \coordinate (c) at (2,9);
 \coordinate (d) at (5,12);
 \coordinate (e) at (5,9);
 \coordinate (f) at (7,7);
 \coordinate (g) at (2,12);
 \coordinate (h) at (0,14);
   \draw[fill=black!10] (a) rectangle (b);
   \draw[fill=white] (c) rectangle (d);
   \fill[black!4] (e)--(f)--(b)--(d)--cycle;
			\draw [line width=1pt, <-] (0,7)-- (7,7);
			\draw [line width=1pt, ->] (7,7)-- (7,14);
			\draw [line width=1pt,->] (7,14)-- (0,14);
			\draw [line width=1pt, <-] (0,14)-- (0,7);
			\draw [line width=1pt,<-] (2,9)-- (5,9);
			\draw [line width=1pt,->] (5,9)-- (5,12);
			\draw [line width=1pt,->] (5,12)-- (2,12);
			\draw [line width=1pt,<-] (2,12)-- (2,9);
			\draw [dashed,line width=1pt,<-] (5,12)-- (7,14);
			\draw [dashed,line width=1pt,->] (7,7)-- (5,9);
			\draw [dashed,line width=1pt,->] (0,7)-- (2,9);
			\draw [dashed,line width=1pt,->] (0,14)-- (2,12);
			\begin{scriptsize}
				\draw [fill=black] (0,7) circle (1.7pt);
				\draw [fill=black] (7,7) circle (1.7pt);
				\draw[color=black] (3.548843996618713,6.654390523370536) node {$g_2$};
				\draw [fill=black] (7,14) circle (1.7pt);
				\draw[color=black] (7.66,10.64464319092457) node {$g_3$};
				\draw [fill=black] (0,14) circle (1.7pt);
				\draw[color=black] (3.548843996618713,14.347115238172494) node {$g_4$};
				\draw[color=black] (-0.79,10.64464319092457) node {$g_1$};
				\draw [fill=black] (2,9) circle (1.7pt);
				\draw [fill=black] (5,9) circle (1.7pt);
				\draw [fill=black] (5,12) circle (1.7pt);
				\draw [fill=black] (2,12) circle (1.7pt);
				\draw[color=black] (6.4,12.7) node {$k_3$};
				\draw[color=black] (6.4,8.286302910004277) node {$k_2$};
				\draw[color=black] (0.8094579827389956,8.386302910004277) node {$k_1$};
				\draw[color=black] (0.9094579827389956,12.7) node {$k_4$};
    \draw[color=black] (3.5,10.5) node {$\rho$};
			\end{scriptsize}
		\end{tikzpicture}
  \begin{tikzpicture}[line cap=round,line join=round,>=stealth,x=1cm,y=1cm, scale=0.45]
			\clip(-1.171814835019457,6.0) rectangle (8.18595630529845,14.9);
   \coordinate (a) at (0,7);
 \coordinate (b) at (7,14);
 \coordinate (c) at (2,9);
 \coordinate (d) at (5,12);
 \coordinate (e) at (5,9);
 \coordinate (f) at (7,7);
 \coordinate (g) at (2,12);
 \coordinate (h) at (0,14);
   \draw[fill=black!10] (a) rectangle (b);
   \draw[fill=white] (c) rectangle (d);
   \fill[black!4] (a)--(c)--(g)--(h)--cycle;
			\draw [line width=1pt, <-] (0,7)-- (7,7);
			\draw [line width=1pt, ->] (7,7)-- (7,14);
			\draw [line width=1pt,->] (7,14)-- (0,14);
			\draw [line width=1pt, <-] (0,14)-- (0,7);
			\draw [line width=1pt,<-] (2,9)-- (5,9);
			\draw [line width=1pt,->] (5,9)-- (5,12);
			\draw [line width=1pt,->] (5,12)-- (2,12);
			\draw [line width=1pt,<-] (2,12)-- (2,9);
			\draw [dashed,line width=1pt,<-] (5,12)-- (7,14);
			\draw [dashed,line width=1pt,->] (7,7)-- (5,9);
			\draw [dashed,line width=1pt,->] (0,7)-- (2,9);
			\draw [dashed,line width=1pt,->] (0,14)-- (2,12);
			\begin{scriptsize}
				\draw [fill=black] (0,7) circle (1.7pt);
				\draw [fill=black] (7,7) circle (1.7pt);
				\draw[color=black] (3.548843996618713,6.654390523370536) node {$g_2'$};
				\draw [fill=black] (7,14) circle (1.7pt);
				\draw[color=black] (7.66,10.7) node {$g_3'$};
				\draw [fill=black] (0,14) circle (1.7pt);
				\draw[color=black] (3.548843996618713,14.5) node {$g_4'$};
				\draw[color=black] (-0.79,10.64464319092457) node {$g_3$};
				\draw [fill=black] (2,9) circle (1.7pt);
				\draw [fill=black] (5,9) circle (1.7pt);
				\draw [fill=black] (5,12) circle (1.7pt);
				\draw [fill=black] (2,12) circle (1.7pt);
				\draw[color=black] (6.4,12.7) node {$k_3'$};
				\draw[color=black] (6.4,8.286302910004277) node {$k_2'$};
				\draw[color=black] (0.8094579827389956,8.386302910004277) node {$k_2$};
				\draw[color=black] (0.9094579827389956,12.7) node {$k_3$};
    \draw[color=black] (3.5,10.5) node {$\rho'$};
			\end{scriptsize}
		\end{tikzpicture}
  \begin{tikzpicture}[line cap=round,line join=round,>=stealth,x=1cm,y=1cm, scale=0.45]
			\clip(-1.171814835019457,6.0) rectangle (8.18595630529845,14.9);
   \coordinate (a) at (0,7);
 \coordinate (b) at (7,14);
 \coordinate (c) at (2,9);
 \coordinate (d) at (5,12);
   \draw[fill=black!10] (a) rectangle (b);
   \draw[fill=white] (c) rectangle (d);
			\draw [line width=1pt, <-] (0,7)-- (7,7);
			\draw [line width=1pt, ->] (7,7)-- (7,14);
			\draw [line width=1pt,->] (7,14)-- (0,14);
			\draw [line width=1pt, <-] (0,14)-- (0,7);
			\draw [line width=1pt,<-] (2,9)-- (5,9);
			\draw [line width=1pt,->] (5,9)-- (5,12);
			\draw [line width=1pt,->] (5,12)-- (2,12);
			\draw [line width=1pt,<-] (2,12)-- (2,9);
			\draw [dashed,line width=1pt,<-] (5,12)-- (7,14);
			\draw [dashed,line width=1pt,->] (7,7)-- (5,9);
			\draw [dashed,line width=1pt,->] (0,7)-- (2,9);
			\draw [dashed,line width=1pt,->] (0,14)-- (2,12);
			\begin{scriptsize}
				\draw [fill=black] (0,7) circle (1.7pt);
				\draw [fill=black] (7,7) circle (1.7pt);
				\draw[color=black] (3.548843996618713,6.654390523370536) node {$g_2g_2'$};
				\draw [fill=black] (7,14) circle (1.7pt);
				\draw[color=black] (7.66,10.64464319092457) node {$g_3'$};
				\draw [fill=black] (0,14) circle (1.7pt);
				\draw[color=black] (3.548843996618713,14.5) node {$g_4g_4'$};
				\draw[color=black] (-0.79,10.64464319092457) node {$g_1$};
				\draw [fill=black] (2,9) circle (1.7pt);
				\draw [fill=black] (5,9) circle (1.7pt);
				\draw [fill=black] (5,12) circle (1.7pt);
				\draw [fill=black] (2,12) circle (1.7pt);
				\draw[color=black] (6.4,12.7) node {$k_3'$};
				\draw[color=black] (6.4,8.286302910004277) node {$k_2'$};
				\draw[color=black] (0.8094579827389956,8.386302910004277) node {$k_1$};
				\draw[color=black] (0.9094579827389956,12.7) node {$k_4$};
    \draw[color=black] (3.5,10.5) node {$\mm(\rho,\rho')$};
			\end{scriptsize}
		\end{tikzpicture}
		\caption{The horizontal multiplication of  $\rho,\rho'\in\hom(\Pi_1(\Sigma,V),G)$ is given by cutting the lighter quadrilaterals in the first two surfaces above and gluing the remaining surfaces along the edges labelled by $k_2$ and $k_3$ to produce the third surface}\label{fig:mul via cut}  \end{figure} 
	Now we shall show how different choices of Lagrangian subgroups allow us to promote \eqref{eq:dou gro modspa} to a holomorphic symplectic $(1,1)$-morphism. We just list below the corresponding values of $L=(L_1,\dots,L_8)$ as indicated in Fig.~\ref{fig:doucomlie}:  
	\begin{align} &L_{ \overline{D}_- }=(G_+,G_+,G_+,G_+,G_+,G_-,G_+,G_-),\qquad &L_{W}=(G_+,G_-,G_+,G_+,G_+,G_+,G_+,G_-), \notag \\ 
  &L_{D_{+}}=(G_+,G_-,G_+,G_-,G_+,G_+,G_+,G_+),\qquad &L_{\overline{Z} }=(G_-,G_+,G_+,G_+,G_-,G_-,G_+,G_-),\notag \\ 
  &L_{C}=(G_-,G_-,G_+,G_+,G_-,G_+,G_+,G_-), \qquad
		&L_{Z}=(G_-,G_-,G_+,G_-,G_-,G_+,G_+,G_+),\label{eq:dec} \\
		&L_{\overline{D}_+ }=(G_-,G_+,G_-,G_+,G_-,G_-,G_-,G_-),\qquad &L_{\overline{W} }=(G_-,G_-,G_-,G_+,G_-,G_+,G_-,G_-)\notag\\
		&L_{D_-}=(G_-,G_-,G_-,G_-,G_-,G_+,G_-,G_+)\notag. \end{align}
        \begin{theorem}\label{pro:doumorequ} The integration of the generalized K\"ahler structure on $K$ defined in \eqref{eq:gkcomlie} is obtained by quasi-Hamiltonian reduction of the trivial $(1,1)$-morphism determined by the double groupoid \eqref{eq:dou gro modspa}, with respect to the boundary conditions in \eqref{eq:dec}, yielding the following holomorphic symplectic $(1,1)$-morphism:
	\begin{equation} 
    \begin{tikzcd}[scale cd=.8]
			\mathfrak{M}_{L_{\overline{D}_- }}(G)                    &  & \mathfrak{M}_{L_{W}}(G) \arrow[rr] \arrow[ll]               &  & \mathfrak{M}_{L_{D_{+}}}(G)                    \\
			\mathfrak{M}_{L_{\overline{Z} }}(G) \arrow[u] \arrow[d] &  & \mathfrak{M}_{L_C}(G) \arrow[ll] \arrow[rr] \arrow[d] \arrow[u] &  & \mathfrak{M}_{L_{Z}}(G) \arrow[d] \arrow[u] \\
			\mathfrak{M}_{L_{\overline{D}_+ }}(G)                    &  & \mathfrak{M}_{L_{\overline{W} }}(G) \arrow[rr] \arrow[ll]               &  & \mathfrak{M}_{L_{D_-}}(G)                  
		\end{tikzcd} \label{eq:doumormodspa} \end{equation} 
	\end{theorem}
	\begin{proof} 
    We only analyze the last column in \eqref{eq:doumormodspa} as the other Morita equivalences are analogous. Here the holomorphic Poisson groupoids are the following: the first one, denoted by $\mathcal{G}_{\mathcal{A}_-}$, is the quotient of the action groupoid 
		\[ (G_+ \times G_-) \ltimes G \rightrightarrows G, \quad (u,x)\cdot g \mapsto ugx^{-1}, \quad \forall (u,x,g)\in( G_+ \times G_- )\ltimes G \]
		by the left action of $G_-^2$ determined by $  (y,z)\cdot(u,x,g) =(u,yxz^{-1},gz^{-1})$ for all $(y,z)\in G_-^2$ and all $(u,x,g)\in (G_+ \times G_-) \ltimes G$. One can easily check that the quotient of such an action is the action groupoid $G_+ \ltimes (G/G_-) \rightrightarrows G/G_-$. On the other hand, we have $\mathcal{G}_{\mathcal{B}_- }$ which is the quotient of the action groupoid
		\[ (G_- \times G_-) \ltimes G \rightrightarrows G, \quad (x,y)\cdot g \mapsto xgy^{-1}, \quad \forall (x,y,g)\in (G_- \times G_- )\ltimes G \]
		by the left $G_-^2$-action $(z_1,z_2)\cdot (x,y,g)=(x,z_1yz_2^{-1},gz_2^{-1})$ for all $(z_1,z_2)\in G_-^2$ and all $(x,y,g)\in (G_- \times G_-) \ltimes G$. This quotient is the action groupoid $G_- \ltimes(G/G_-)\rightrightarrows G/G_-$. So the Lu-Weinstein construction of \autoref{rem:genluwei} applied to $\mathcal{G}_{\mathcal{A}_- }$ and $\mathcal{G}_{\mathcal{B}_- }$ gives us a holomorphic double groupoid 
		\[ \begin{tikzcd}
			{D_-} \arrow[r, shift right] \arrow[r, shift left] \arrow[d, shift right] \arrow[d, shift left] & \mathcal{G}_{\mathcal{A}_-} \arrow[d, shift right] \arrow[d, shift left] \\
			\mathcal{G}_{\mathcal{B}_-} \arrow[r, shift right] \arrow[r, shift left]                                                                 & G/G_-;                                                     
		\end{tikzcd}\]
		the symplectic structure on $D_-$ comes from the identification $D_- \cong \mathfrak{M}_{L_{D_-}}(G)$ which is given as follows. An element $\rho\in D_-$ as below 
		\[ \rho=\left((u',ygG_-),(y,gG_-),(y',ugG_-),(u,gG_-)\right)\in D_-\subset \mathcal{G}_{\mathcal{A}_- } \times_{G/G_-} \mathcal{G}_{\mathcal{B}_- } \times \mathcal{G}_{\mathcal{B}_- } \times_{G/G_-} \mathcal{G}_{\mathcal{A}_-}    \]
		satisfies that $y'ug=u'ygz^{-1}$ for a unique $z\in G_-$. So we may identify $\rho$ with the representation whose values on the generators of $\Pi_1(\Sigma,V)$ indicated in Fig.~\ref{fig:doucomlie} are displayed in Fig.~\ref{fig:rho}, where $g'=y'ug=u'ygz^{-1}$. Conversely, any element in $\mathfrak{M}_{L_{D_-}}(G)$ is of the form of $\rho$ above. 
  
\noindent\begin{minipage}{\linewidth}
\begin{wrapfigure}{r}{0pt}\begin{tikzpicture}[line cap=round,line join=round,>=stealth,x=1cm,y=1cm, scale=0.5]
	\clip(-2.2,6.22534829668545) rectangle (9.1,15.23302026597934);
  \coordinate (a) at (0,7);
 \coordinate (b) at (7,14);
 \coordinate (c) at (2,9);
 \coordinate (d) at (5,12);
   \draw[fill=black!10] (a) rectangle (b);
   \draw[fill=white] (c) rectangle (d);
				\draw [line width=1pt, <-] (0,7)-- (7,7);
				\draw [line width=1pt, ->] (7,7)-- (7,14);
				\draw [line width=1pt,->] (7,14)-- (0,14);
				\draw [line width=1pt, <-] (0,14)-- (0,7);
				\draw [line width=1pt,<-] (2,9)-- (5,9);
				\draw [line width=1pt,->] (5,9)-- (5,12);
				\draw [line width=1pt,->] (5,12)-- (2,12);
				\draw [line width=1pt,<-] (2,12)-- (2,9);
				\draw [dashed,line width=1pt,<-] (5,12)-- (7,14);
				\draw [dashed,line width=1pt,->] (7,7)-- (5,9);
				\draw [dashed,line width=1pt,->] (0,7)-- (2,9);
				\draw [dashed,line width=1pt,->] (0,14)-- (2,12);
				\begin{scriptsize}
					\draw [fill=black] (0,7) circle (1.7pt);
					\draw [fill=black] (7,7) circle (1.7pt);
					\draw[color=black] (3.548843996618713,6.654390523370536) node {$z\in G_-$};
					\draw [fill=black] (7,14) circle (1.7pt);
					\draw[color=black] (7.4,10.64464319092457) node {$1$};
					\draw [fill=black] (0,14) circle (1.7pt);
					\draw[color=black] (3.548843996618713,14.347115238172494) node {$1$};
					\draw[color=black] (-0.89,10.64464319092457) node {$1$};
					\draw [fill=black] (2,9) circle (1.7pt);
					\draw [fill=black] (5,9) circle (1.7pt);
					\draw[color=black] (3.548843996618713,8.6) node {$y\in G_-$};
					\draw [fill=black] (5,12) circle (1.7pt);
					\draw[color=black] (6.0,10.64464319092457) node {$u\in G_+$};
					\draw [fill=black] (2,12) circle (1.7pt);
					\draw[color=black] (3.548843996618713,12.343136968374939) node {$y'\in G_-$};
					\draw[color=black] (1.1,10.64464319092457) node {$u'\in G_+$};
					\draw[color=black] (6.3,12.5) node {$ug$};
					\draw[color=black] (6.1760106308045435,8.286302910004277) node {$g$};
					\draw[color=black] (1.9,8.00) node {$ygz^{-1}$};
					\draw[color=black] (0.9094579827389956,12.5) node {$g'$};
				\end{scriptsize}    
\end{tikzpicture} \caption{The values of $\rho$}\label{fig:rho}\end{wrapfigure}
By definition, the symplectic form on $D_-$ coincides with the pullback of \eqref{eq:luweicomlie} by the identification above, it is straightforward to verify it is doubly multiplicative. The same considerations applied to $\mathfrak{M}_{L_{D_{+}}}(G)$ give us now a holomorphic symplectic double groupoid $D_+$ in which the side groupoids are the action groupoids $\mathcal{G}_{\mathcal{A}_+ }=((G_+\backslash G) \rtimes G_- \rightrightarrows G_+\backslash G)$ and $\mathcal{G}_{\mathcal{B}_+ }=((G_+\backslash G) \rtimes G_+ \rightrightarrows G_+\backslash G)$ corresponding to the right actions, respectively, of $G_-$ and $G_+$ on the quotient by left translations $G_+\backslash G$. Now we will describe the Morita equivalence between $D_+$ and $D_-$ determined by $\mathfrak{M}_{L_{Z}}(G)$. The holomorphic biprincipal bibundle relating $\mathcal{G}_{\mathcal{A}_+ }$ and $\mathcal{G}_{\mathcal{A}_- }$ is the quotient $Z$ of $G_+ \times G_- \times G$ by the left $G_+ \times G_-$-action defined in the equation below:
		\[ (u,x)\cdot (v,y,g) =(uv,yx^{-1},gx^{-1}), \quad \forall (u,x)\in G_+ \times G_-, \,  (v,y,g)\in G_+ \times G_- \times G. \]
		So we have that $Z$ is identified with $G$ and the two commuting actions that define the Morita equivalence are given by left and right translation\end{minipage}
  
		\begin{equation*}\label{eq:z morita} \begin{tikzcd}[scale cd=0.9]
(G_+\backslash G) \rtimes G_- \arrow[d, shift right] \arrow[d, shift left] &  & Z\cong G \arrow[lld, "p_+"'] \arrow[rrd, "p_-"] &  & G_+ \ltimes(G/G_-) \arrow[d, shift right] \arrow[d, shift left] &  & {p_+(g)=G_+g,}  \\
G_+\backslash G                                                            &  &                                                 &  & G/G_-                                                           &  & {p_-(g)=gG_-;}
\end{tikzcd}\end{equation*} 
		the actions $((G_+\backslash G) \rtimes G_- ) \times_{G_+\backslash G} G \rightarrow G$ and $G \times_{G/G_-}( G_+ \ltimes (G/G_-) )\rightarrow G$ are given, respectively, by the maps
		\begin{align*} &((G_+y,b),g)\mapsto gb^{-1},\quad \forall ((G_+y,b),g) \in ((G_+\backslash G) \rtimes G_- ) \times_{G_+\backslash G} G \\
			& (g,(a,xG_-)) \mapsto a^{-1}g, \quad \forall (g,(a,xG_-))\in G \times_{G/G_-}( G_+ \ltimes G/G_- ). 
		\end{align*} 
		By considering an analogous diagram to the one in Fig.~\ref{fig:rho}, we get an identification 
		\begin{equation}
	\begin{aligned}
		&\mathfrak{M}_{L_{Z}}(G) \cong  Z \times_{(G/G_-)^2} \mathcal{G}_{\mathcal{B}_-} \times \mathcal{G}_{\mathcal{B}_+} \times_{(G_+\backslash G)^2} Z,  \\
  &Z \times_{(G/G_-)^2} \mathcal{G}_{\mathcal{B}_-} \times \mathcal{G}_{\mathcal{B}_+} \times_{(G_+\backslash G)^2} Z:=\\
  &=\left\{(z,x,x',z')\in Z\times \mathcal{G}_{\B_-}\times \mathcal{G}_{\B_+}\times Z\left| \, \begin{aligned}
        p_-(z)=\ttt(x),\, p_+(z)=\ttt(x'),\\
        p_-(z')=\sss(x),\, p_+(z')=\sss(x')
    \end{aligned}\right.\right\}
    \label{eq:lu weinstein bimodule}    
	\end{aligned}
		\end{equation}
		that induces a suitable symplectic form on $Z \times_{(G/G_-)^2} \mathcal{G}_{\mathcal{B}_-} \times \mathcal{G}_{\mathcal{B}_+} \times_{(G_+\backslash G)^2} Z$. Then the action of $D_-$ on $\mathcal{Z}$ is given by \eqref{eq:luweiact} and a direct computation using \eqref{eq:luweicomlie} shows that the action is symplectic. We can proceed similarly to describe explicitly all the Morita equivalences in \eqref{eq:doumormodspa}. \end{proof}
  
  \begin{example}[Hopf surface]\label{ex:hopf surf} Consider $K=SU(2)\times \mathbb{R} $ and ${G} =SL_2(\mathbb{C}) \times \mathbb{C} $ as in \autoref{ex:hopf surf 0}. Note that the quotient map $p_-:Z\cong {G} \rightarrow {G}/G_-$ can be identified with the projection ${G} \rightarrow X_-:=\mathbb{C}^2-\{(0,0)\}$ defined by $(A,z)\mapsto A\left[\begin{smallmatrix} e^{-iz} \\ 0\end{smallmatrix}\right]$. The Hopf surface is the quotient by a $\Z$-action $X_-/(x\sim 2x)\cong SU(2)\times U(1) $. 
  
  We have that $G_-$ and $G_+$ as in \autoref{ex:hopf surf 0} are isomorphic to the universal cover of the affine group of the complex line: $G_-\cong G_+\cong \mathbb{C} \ltimes \mathbb{C}$ with multiplication 
\[  (a,b)(c,d)=(a+b,e^cb+d), \quad (a,b),(c,d)\in \mathbb{C} \ltimes \mathbb{C}.  \]
In fact, we can define the respective isomorphisms $G_-\rightarrow  \mathbb{C} \ltimes \mathbb{C}$ and $G_+\rightarrow  \mathbb{C} \ltimes \mathbb{C}$ as follows: 
\[ \left(\begin{bmatrix} e^{iz} & w \\ 0 & e^{-iz} \end{bmatrix},z\right) \mapsto ({-2iz},e^{-iz}w), \quad \left(\begin{bmatrix} e^{-iz} & 0 \\ w & e^{iz} \end{bmatrix},z\right) \mapsto ({-2iz},e^{-iz}w). \]
Using the previous identifications, the action groupoids $\mathcal{G}_{\mathcal{A}_-}=( G_+ \ltimes X_- \rightrightarrows X_-)$ and $\mathcal{G}_{\mathcal{B}_-}=( G_- \ltimes X_- \rightrightarrows X_-)$ admit the following simplified description: they respectively correspond to the actions of $\mathbb{C} \ltimes \mathbb{C}$ defined as below: 
\begin{align}    (a,b)\cdot (z_1,z_2) \mapsto (e^{a}z_1,bz_1+z_2), \quad(a,b)\cdot (z_1,z_2) \mapsto (z_1+bz_2,e^{a}z_2). \label{eq:acthop} \end{align}  
Using the duality pairing between $\mathfrak{g}_+$ and $\mathfrak{g}_-$ determined by $s_{\mathbb{C} }$, we obtain the Hitchin Poisson structure on $X_-$:  
\[ \sigma_-^\sharp=\rho_{A_-}\circ \rho_{B_-}^*:T^*_{1,0}X_- \rightarrow T_{1,0}X_-, \quad \sigma_-=2z_1z_2 \frac{\partial}{\partial z_1} \wedge \frac{\partial}{\partial z_2} .    \]
By applying the same considerations to $X_+=G_+\backslash {G}\cong \mathbb{C}^2-\{(0,0)\}  $, we derive that the quotient $p_+:Z\cong {G} \rightarrow X_+$ is given by $(A,u) \mapsto \left[\begin{smallmatrix} e^{iu} & 0\end{smallmatrix}\right]A $ and the corresponding action groupoids are $\mathcal{G}_{\mathcal{A}_+ }\cong (X_+ \rtimes (\mathbb{C} \ltimes \mathbb{C})^{op}\rightrightarrows X_+ )$ and $\mathcal{G}_{\mathcal{B}_+ }=(X_+ \rtimes (\mathbb{C} \ltimes \mathbb{C})^{op}\rightrightarrows X_+ )$
with corresponding actions defined respectively by the same formulae as in \eqref{eq:acthop} but now using the opposite group structure on $(\mathbb{C} \ltimes \mathbb{C})$ which is denoted by $(\mathbb{C} \ltimes \mathbb{C})^{op}$. Similarly, we deduce that $\sigma_+=-\sigma_-$. \end{example}
	\subsubsection{The real symplectic core and a real trivialization of \eqref{eq:doumormodspa}} Since $G$ is simply-connected, complex conjugation on $\mathfrak{g} $ integrates to an automorphism $\Theta:G \rightarrow \overline{G} $ such that $\Theta(G_+)=G_-$. For instance, in \autoref{ex:hopf surf}, the real structure $\Theta$ on ${G}$ given by $(A,z)\mapsto ((A^\dagger)^{-1},-\overline{z})  $ integrates the complex conjugation corresponding to $\mathfrak{g}=\mathfrak{k}_{\mathbb{C}}$ as in \autoref{ex:hopf surf 0}. 
    
    In this situation, we get a symplectic real structure $\tau:\mathcal{C} \rightarrow \overline{\mathcal{C}}$ on the double bimodule $\mathcal{C}=\mathfrak{M}_{L_{C}}(G)$ of \eqref{eq:doumormodspa} by combining $\Theta$ with the natural reflection symmetry of the marked annulus $(\Sigma,V)$.
	\begin{proposition}\label{pro:real str modspa} The $(1,1)$-morphism in \eqref{eq:doumormodspa} becomes self-adjoint via $\tau:=\Theta\circ\tau_{0,0}:\mathcal{C} \to \overline{\mathcal{C}}^\top$, where $\tau_{0,0}$ is as in \eqref{eq:real triv}. In particular, the map $\tau_{\mathcal{C}}:\mathcal{C}  \rightarrow \overline{\mathcal{C}} $ determined by 
		\[ [g_1,\dots,g_4,k_1,\dots,k_4]\mapsto [\Theta(g_3^{-1}),\Theta(g_4^{-1}),\Theta(g_1^{-1}),\Theta(g_2^{-1}),\Theta(k_3),\Theta(k_4),\Theta(k_1),\Theta(k_2)], \]
		using the generators displayed in Fig.~\eqref{fig:doucomlie}, defines a real structure on $\mathcal{C}$. 
	\end{proposition}
	\begin{proof} To make sense of $\tau_{0,0}$, note that \eqref{eq:doumormodspa} is obtained by quasi-Hamiltonian reduction from the trivial $(1,1)$-morphism determined by the double Lie groupoid \eqref{eq:dou gro modspa}. We only verify the required property for $\mathcal{C}$ as the other cases are similar. The key fact needed in this argument is the following: $\tau_{\mathcal{C}}$ may be viewed as $\tau_{\mathcal{C}}=\Theta_*\circ R^!$, where $\Theta_*$ denotes post composition with $\Theta$ and $R^!$ is the pullback map of a representation along the rotation $R$ of the diagram in Figure \eqref{fig:doucomlie} by an angle of $\pi$ with respect to its center of symmetry. Since $R$ preserves the orientation of the surface, we have that $(R^!)^*\Omega =\Omega $ and hence $\tau_{\mathcal{C}}^*\Omega =\overline{\Omega  } $. \end{proof}  

 
	\begin{proposition} The fixed point locus $\mathcal{S}$ of $\tau_{\mathcal{C}}$ is nonempty and contains a distinguished Lagrangian bisection given by 
    \[ \lambda =\{(1,\dots,1,k,\dots,k) |\ k\in K\}\subset \mathcal{S}, \]
    in terms of the generators displayed in Fig.~\eqref{fig:doucomlie}.  
	\end{proposition}  
	\begin{proof} This follows immediately from the formula for $\tau_{\mathcal{C}}$ in \autoref{pro:real str modspa} and \eqref{eq:luweicomlie}. \end{proof}
Graphically, we may view the fixed point locus $\mathcal{S} $ of $\tau_{\mathcal{C}}:\mathcal{C} \rightarrow \overline{\mathcal{C}}$ as the set of representations which satisfy the constraints indicated in Fig.~\ref{fig:canlag}. 
	\begin{figure}[H]
	    \centering\begin{tikzpicture}[line cap=round,line join=round,>=stealth,x=1cm,y=1cm, scale=0.45]
			\clip(-5.171814835019457,6.22534829668545) rectangle (11.18595630529845,15.23302026597934);
    \coordinate (a) at (0,7);
 \coordinate (b) at (7,14);
 \coordinate (c) at (2,9);
 \coordinate (d) at (5,12);
   \draw[fill=black!10] (a) rectangle (b);
   \draw[fill=white] (c) rectangle (d);
			\draw [line width=1pt, <-] (0,7)-- (7,7);
			\draw [line width=1pt, ->] (7,7)-- (7,14);
			\draw [line width=1pt,->] (7,14)-- (0,14);
			\draw [line width=1pt, <-] (0,14)-- (0,7);
			\draw [line width=1pt,<-] (2,9)-- (5,9);
			\draw [line width=1pt,->] (5,9)-- (5,12);
			\draw [line width=1pt,->] (5,12)-- (2,12);
			\draw [line width=1pt,<-] (2,12)-- (2,9);
			\draw [dashed,line width=1pt,<-] (5,12)-- (7,14);
			\draw [dashed,line width=1pt,->] (7,7)-- (5,9);
			\draw [dashed,line width=1pt,->] (0,7)-- (2,9);
			\draw [dashed,line width=1pt,->] (0,14)-- (2,12);
			\begin{scriptsize}
				\draw [fill=black] (0,7) circle (1.7pt);
				\draw [fill=black] (7,7) circle (1.7pt);
				\draw[color=black] (3.548843996618713,6.554390523370536) node {$g_2$};
				\draw [fill=black] (7,14) circle (1.7pt);
				\draw[color=black] (8.339096664172763,10.64464319092457) node {$\Theta(g_1)^{-1}$};
				\draw [fill=black] (0,14) circle (1.7pt);
				\draw[color=black] (3.548843996618713,14.547115238172494) node {$\Theta(g_2)^{-1}$};
				\draw[color=black] (-0.59,10.64464319092457) node {$g_1$};
				\draw [fill=black] (2,9) circle (1.7pt);
				\draw [fill=black] (5,9) circle (1.7pt);
				\draw [fill=black] (5,12) circle (1.7pt);
				\draw [fill=black] (2,12) circle (1.7pt);
				\draw[color=black] (5.1,13.1) node {$\Theta(k_1)$};
				\draw[color=black] (6.2760106308045435,8.4) node {$k_2$};
				\draw[color=black] (0.8,8.4) node {$k_1$};
				\draw[color=black] (1.9,13.1) node {$\Theta(k_2)$};
			\end{scriptsize}
		\end{tikzpicture} \begin{tikzpicture}[line cap=round,line join=round,>=stealth,x=1cm,y=1cm, scale=0.45]
			\clip(-3.171814835019457,6.22534829668545) rectangle (11.18595630529845,15.23302026597934);
    \coordinate (a) at (0,7);
 \coordinate (b) at (7,14);
 \coordinate (c) at (2,9);
 \coordinate (d) at (5,12);
   \draw[fill=black!10] (a) rectangle (b);
   \draw[fill=white] (c) rectangle (d);
			\draw [line width=1pt, <-] (0,7)-- (7,7);
			\draw [line width=1pt, ->] (7,7)-- (7,14);
			\draw [line width=1pt,->] (7,14)-- (0,14);
			\draw [line width=1pt, <-] (0,14)-- (0,7);
			\draw [line width=1pt,<-] (2,9)-- (5,9);
			\draw [line width=1pt,->] (5,9)-- (5,12);
			\draw [line width=1pt,->] (5,12)-- (2,12);
			\draw [line width=1pt,<-] (2,12)-- (2,9);
			\draw [dashed,line width=1pt,<-] (5,12)-- (7,14);
			\draw [dashed,line width=1pt,->] (7,7)-- (5,9);
			\draw [dashed,line width=1pt,->] (0,7)-- (2,9);
			\draw [dashed,line width=1pt,->] (0,14)-- (2,12);
			\begin{scriptsize}
				\draw [fill=black] (0,7) circle (1.7pt);
				\draw [fill=black] (7,7) circle (1.7pt);
				\draw[color=black] (3.548843996618713,6.554390523370536) node {$1$};
				\draw [fill=black] (7,14) circle (1.7pt);
				\draw[color=black] (7.5,10.64464319092457) node {$1$};
				\draw [fill=black] (0,14) circle (1.7pt);
				\draw[color=black] (3.548843996618713,14.547115238172494) node {$1$};
				\draw[color=black] (-0.59,10.64464319092457) node {$1$};
				\draw [fill=black] (2,9) circle (1.7pt);
				\draw [fill=black] (5,9) circle (1.7pt);
				\draw [fill=black] (5,12) circle (1.7pt);
				\draw [fill=black] (2,12) circle (1.7pt);
				\draw[color=black] (5.4,13.1) node {$k$};
				\draw[color=black] (6.2760106308045435,8.4) node {$k$};
				\draw[color=black] (0.8,8.4) node {$k$};
				\draw[color=black] (1.4,13.1) node {$k$};
			\end{scriptsize}
		\end{tikzpicture}
		\caption{The real symplectic core of \eqref{eq:doumormodspa} and a distinguished Lagrangian inside it}\label{fig:canlag}
	\end{figure}   
	\begin{theorem}\label{thm:poisson str on bisections K example} The Lagrangian bisection $\lambda\subset\mathcal{S}$ determines a real trivialization of \eqref{eq:doumormodspa}, which includes the following multiplicative Lagrangian bisections, given by the action groupoids corresponding to the dressing actions \eqref{eq:dres1}:
         \begin{enumerate}
             \item inside the multiplicative bimodule $\mathcal{Z}:= \mathfrak{M}_{L_{{Z} }}(G)$ we have $\Lambda_Z$, which is isomorphic to $G_- \ltimes K \rightrightarrows K$ equipped with the following symplectic form and Poisson structure on its base \begin{align*} &\Omega_Z|_{(b,k)}:=\left(s_I(({}^bk)^*\theta^l,(b^k)^*\theta^r)- s_I(b^*\theta^l,k^*\theta^r)\right), \qquad \pi_Z=\frac{(I^l-I^r)g^{-1}}{2},
			\end{align*}
            \item inside the multiplicative bimodule $\mathcal{W}:= \mathfrak{M}_{L_{{W} }}(G)$ we have $\Lambda_W$, which is isomorphic to $G_+ \ltimes K \rightrightarrows K$ equipped with the following symplectic form and Poisson structure on its base
            \begin{align*} 
&\Omega_{\overline{W}}|_{(a,k)}:=-\left(s_I(({}^ak)^*\theta^l,(a^k)^*\theta^r)- s_I(a^*\theta^l,k^*\theta^r)\right), \qquad \pi_{\overline{W}} =-\frac{(I^r+I^l)g^{-1}}{2}.
			\end{align*}
         \end{enumerate}
			 
	\end{theorem} 
	\begin{proof} Since $\mathcal{Z}$ admits the description given by \eqref{eq:lu weinstein bimodule} and a similar description holds for $\overline{\mathcal{W}}$ and $\mathcal{C}$, these Morita bimodules are of Lu-Weinstein type as in \autoref{rem: lu-wei bimodules} and hence \autoref{thm: key Lagrangian in the central brane} implies that there is a real trivialization of \eqref{eq:doumormodspa} determined by $\lambda$. In particular, we obtain the following descriptions:
    \begin{align*}
        &\Lambda_Z=\{[g_1,\dots,g_4,k_1,\dots,k_4]=[1,b^{k},1,\Theta(b^k),{}^bk,k,k,{}^bk]|\ b\in G_-, k\in K\}, \\
        &\Lambda_{\overline{W}}=\{[g_1,\dots,g_4,k_1,\dots,k_4]=[\Theta(a^k),1,,a^k,1,k,k,{}^ak,{}^ak]|\ a\in G_+,k\in K\},
    \end{align*}
    which are given in terms of the generators of Fig.~\ref{fig:doucomlie} and illustrated in Fig.~\ref{fig:bra}. So the result follows from pulling back \eqref{eq:luweicomlie} to the corresponding Lagrangian bisections and observing that the inclusions $K \hookrightarrow Z$, $K \hookrightarrow \overline{W}$ are the infinitesimal counterparts of $\Lambda_Z$ and $\Lambda_{\overline{W}}$. In fact, the infinitesimal dressing action $\mathfrak{g}_-  \rightarrow \mathfrak{X}( K)$ corresponding to \eqref{eq:dres1} is given by 
	\begin{align}  u\mapsto \rho(u), \quad \rho(u)|_k=Tl_k \text{pr}_{\mathfrak{k} }(\text{Ad}_{k^{-1}}(u)), \quad \forall (u,k)\in \mathfrak{g}_+ \times K; \label{eq:infdresact}\end{align}  
	where $l_k:K \rightarrow K$ is the left translation by $k$ and $\text{pr}_{\mathfrak{k} }:\mathfrak{g} = \mathfrak{g}_+ \oplus \mathfrak{k} \rightarrow \mathfrak{k} $ is the projection. A similar formula holds for the infinitesimal dressing action $\mathfrak{g}_+  \rightarrow \mathfrak{X} (K)$. Therefore, we get the corresponding Poisson tensors by differentiating the symplectic forms above, combining \eqref{eq:infdresact} with \eqref{eq:IM 2 form}.  \end{proof}
    \begin{figure}[H] 
 \centering
 \begin{tikzpicture}[line cap=round,line join=round,>=stealth,x=1cm,y=1cm, scale=0.45]
			\clip(-5.171814835019457,6.22534829668545) rectangle (11.18595630529845,14.9);
    \coordinate (a) at (0,7);
 \coordinate (b) at (7,14);
 \coordinate (c) at (2,9);
 \coordinate (d) at (5,12);
   \draw[fill=black!10] (a) rectangle (b);
   \draw[fill=white] (c) rectangle (d);
			\draw [line width=1pt, <-] (0,7)-- (7,7);
			\draw [line width=1pt, ->] (7,7)-- (7,14);
			\draw [line width=1pt,->] (7,14)-- (0,14);
			\draw [line width=1pt, <-] (0,14)-- (0,7);
			\draw [line width=1pt,<-] (2,9)-- (5,9);
			\draw [line width=1pt,->] (5,9)-- (5,12);
			\draw [line width=1pt,->] (5,12)-- (2,12);
			\draw [line width=1pt,<-] (2,12)-- (2,9);
			\draw [dashed,line width=1pt,<-] (5,12)-- (7,14);
			\draw [dashed,line width=1pt,->] (7,7)-- (5,9);
			\draw [dashed,line width=1pt,->] (0,7)-- (2,9);
			\draw [dashed,line width=1pt,->] (0,14)-- (2,12);
			\begin{scriptsize}
				\draw [fill=black] (0,7) circle (1.7pt);
				\draw [fill=black] (7,7) circle (1.7pt);
				\draw[color=black] (3.548843996618713,6.654390523370536) node {$b^k$};
				\draw [fill=black] (7,14) circle (1.7pt);
				\draw[color=black] (7.239096664172763,10.64464319092457) node {$1$};
				\draw [fill=black] (0,14) circle (1.7pt);
				\draw[color=black] (3.548843996618713,14.447115238172494) node {$\Theta(b^k)$};
				\draw[color=black] (-0.39,10.64464319092457) node {$1$};
				\draw [fill=black] (2,9) circle (1.7pt);
				\draw [fill=black] (5,9) circle (1.7pt);
				\draw[color=black] (3.548843996618713,8.7) node {$b$};
				\draw [fill=black] (5,12) circle (1.7pt);
				\draw[color=black] (4.6,10.64464319092457) node {$1$};
				\draw [fill=black] (2,12) circle (1.7pt);
				\draw[color=black] (3.548843996618713,12.343136968374939) node {$\Theta(b)$};
				\draw[color=black] (2.5,10.64464319092457) node {$1$};
				\draw[color=black] (6.1760106308045435,12.7) node {$k$};
				\draw[color=black] (6.1760106308045435,8.286302910004277) node {$k$};
				\draw[color=black] (1.6,8.005257177044864) node {${}^bk$};
				\draw[color=black] (0.9094579827389956,12.7) node {${}^bk$};
			\end{scriptsize}
		\end{tikzpicture}
		\begin{tikzpicture}[line cap=round,line join=round,>=stealth,x=1cm,y=1cm, scale=0.45]
			\clip(-3.171814835019457,6.22534829668545) rectangle (11.18595630529845,14.9);
    \coordinate (a) at (0,7);
 \coordinate (b) at (7,14);
 \coordinate (c) at (2,9);
 \coordinate (d) at (5,12);
   \draw[fill=black!10] (a) rectangle (b);
   \draw[fill=white] (c) rectangle (d);
			\draw [line width=1pt, <-] (0,7)-- (7,7);
			\draw [line width=1pt, ->] (7,7)-- (7,14);
			\draw [line width=1pt,->] (7,14)-- (0,14);
			\draw [line width=1pt, <-] (0,14)-- (0,7);
			\draw [line width=1pt,<-] (2,9)-- (5,9);
			\draw [line width=1pt,->] (5,9)-- (5,12);
			\draw [line width=1pt,->] (5,12)-- (2,12);
			\draw [line width=1pt,<-] (2,12)-- (2,9);
			\draw [dashed,line width=1pt,<-] (5,12)-- (7,14);
			\draw [dashed,line width=1pt,->] (7,7)-- (5,9);
			\draw [dashed,line width=1pt,->] (0,7)-- (2,9);
			\draw [dashed,line width=1pt,->] (0,14)-- (2,12);
			\begin{scriptsize}
				\draw [fill=black] (0,7) circle (1.7pt);
				\draw [fill=black] (7,7) circle (1.7pt);
				\draw[color=black] (3.548843996618713,6.654390523370536) node {$1$};
				\draw [fill=black] (7,14) circle (1.7pt);
				\draw[color=black] (7.539096664172763,10.64464319092457) node {$a^k$};
				\draw [fill=black] (0,14) circle (1.7pt);
				\draw[color=black] (3.548843996618713,14.447115238172494) node {$1$};
				\draw[color=black] (-1.19,10.64464319092457) node {$\Theta(a^k)$};
				\draw [fill=black] (2,9) circle (1.7pt);
				\draw [fill=black] (5,9) circle (1.7pt);
				\draw[color=black] (3.548843996618713,8.7) node {$1$};
				\draw [fill=black] (5,12) circle (1.7pt);
				\draw[color=black] (4.6,10.64464319092457) node {$a$};
				\draw [fill=black] (2,12) circle (1.7pt);
				\draw[color=black] (3.548843996618713,12.343136968374939) node {$1$};
				\draw[color=black] (2.8,10.64464319092457) node {$\Theta(a)$};
				\draw[color=black] (6.2760106308045435,12.6) node {${}^ak$};
				\draw[color=black] (6.1760106308045435,8.286302910004277) node {$k$};
				\draw[color=black] (1.5,8.005257177044864) node {$k$};
				\draw[color=black] (0.8094579827389956,12.6) node {${}^ak$};
			\end{scriptsize}
		\end{tikzpicture} \caption{Elements of the multiplicative bisections $\Lambda_Z$ and $\Lambda_{\overline{W}}$, respectively.} \label{fig:bra}
	\end{figure}
It thus follows that $\pi_Z$ is a multiplicative Poisson structure (corresponding to the Manin triple $(\mathfrak{g} ,\mathfrak{k} , \mathfrak{g}_- )$), whereas $\pi_{\overline{W}}$ is an affine Poisson structure with respect to $\pi_Z$, see \cite[Ch. 5]{luphd}. Note how, as expected by our general results, the choice of the Lagrangian bisection $\lambda:K \hookrightarrow \mathcal{S}$ completely determines a pair of real Poisson structures on $K$ which together determine the metric $g$ as in the proof of \autoref{thm: reconstruction}.   
 \subsubsection{A generalized K\"ahler potential for the Hopf surface} Implementing the trivialization of \eqref{eq:doumormodspa} determined by the Lagrangian bisection $\lambda:K=SU(2)\times \mathbb{R}\hookrightarrow \mathcal{S}$ as above in the case of \autoref{ex:hopf surf}, we get the following explicit formulae. If we identify $(z_1,z_2)\in X_-=X_+ =\mathbb{C}^2-\{(0,0)\}$ with the matrix $A=\left[\begin{smallmatrix} z_1 & -\overline{z_2} \\ z_2 & \overline{z_2}\end{smallmatrix}\right]$, then the diffeomorphisms determined by $\lambda$ and the corresponding bisections of the bimodules $Z$ and $W$ are indicated below 
\[ \begin{tikzcd}[scale cd=.85]
\overline{X_-} \arrow[d, "\lambda_Z", bend left=49] & W \arrow[r] \arrow[l]                     & X_+ \arrow[l, "\lambda_W", bend left=49] \\
\overline{Z} \arrow[u] \arrow[d]                    & \mathcal{C} \arrow[r] \arrow[l] \arrow[d] \arrow[u] & Z \arrow[d] \arrow[u]                    \\
\overline{X_+} \arrow[r, "\lambda_W", bend left=49] & \overline{W} \arrow[r] \arrow[l]          & X_- \arrow[u, "\lambda_Z", bend left=49]
\end{tikzcd} \qquad \begin{tikzcd}[scale cd=.85]
\overline{X_-}                                                     &  & X_+ \arrow[ll, "{A^{-1}} \mapsfrom A"']                                                    \\
\overline{X_+} \arrow[u, "A\mapsto \overline{A^{-1}}" description] &  & X_-. \arrow[ll, "{A^{-1}} \mapsfrom A"] \arrow[u, "A\mapsto \overline{A^{-1}}" description]
\end{tikzcd}\] 
Moreover, $X_-$ is covered by two coordinate patches where the restricted generalized K\"ahler  structure is of symplectic type: $\mathcal{O}_A=\{(z_1,z_2)|\ z_1\neq0\}$ and $\mathcal{O}_B=\{(z_1,z_2)|\ z_2\neq0\}$. In fact, $\mathcal{O}_A$ is the $\A$-leaf of $(1,1)$, whereas $\mathcal{O}_B$ is the $\B$-leaf of $(1,1)$, see \cite[Remark 3.10]{MR3232003}. 
So, for the purpose of studying the generalized K\"ahler  structure induced by $\lambda$ on $X_-= X_+$, it suffices to use the two restricted Morita equivalences:
\[ \begin{tikzcd}[scale cd=.8]
\mathcal{G}_{\mathcal{A}_+}|_{\mathcal{O}_B} \arrow[d, shift right] \arrow[d, shift left]  &  & Z|_{\mathcal{O}_B} \arrow[lld, "p_+"'] \arrow[rrd, "p_-"] &  & \mathcal{G}_{\mathcal{A}_-}|_{\mathcal{O}_B} \arrow[d, shift right] \arrow[d, shift left]  \\
\mathcal{O}_B                                                                                                           &  &                                                 &  & \mathcal{O}_B \arrow[llu, "\lambda_Z|", bend left]                                                                       
\end{tikzcd}\begin{tikzcd}[scale cd=.8]
\overline{\mathcal{G}_{\mathcal{B}_+}|_{\mathcal{O}_A}} \arrow[d, shift right] \arrow[d, shift left]  &  & \overline{W}|_{\mathcal{O}_A} \arrow[lld, "q_+"'] \arrow[rrd, "q_-"] &  & \mathcal{G}_{\mathcal{B}_-}|_{\mathcal{O}_A} \arrow[d, shift right] \arrow[d, shift left]  \\
\mathcal{O}_A \arrow[rru, "\lambda_W|"', bend right]                                                                                &  &                                                            &  & \mathcal{O}_A                                                                                                           
\end{tikzcd} \]
which are now holomorphic symplectic Morita equivalences according to \autoref{rem:symplectic type A and B leaves}. As we have seen in \S\ref{subsec:examples}, the real parts of the symplectic structures on $Z|_{\mathcal{O}_B}$ and $\overline{W}|_{\mathcal{O}_A}$ agree with the ones on the respective real symplectic cores $ (Z|_{\mathcal{O}_B})_\Delta\hookrightarrow Z|_{\mathcal{O}_B} \times \overline{Z|_{\mathcal{O}_B}}$ and $(\overline{W}|_{\mathcal{O}_A})_\Delta\hookrightarrow  \overline{W|_{\mathcal{O}_A}}\times W|_{\mathcal{O}_A}$ up to a factor of 2. To obtain simple formulae, we will focus on $\mathcal{O}_A\cap \mathcal{O}_B$. Since the map
\[ \begin{tikzcd}
Z|_{\mathcal{O}_B} \arrow[rr, "{(p_+,p_-)}"] &  & {(\mathcal{O}_B\times \mathcal{O}_B,\sigma_+\times(-\sigma_-))}
\end{tikzcd}\] 
is Poisson and a local diffeomorphism over $(\mathcal{O}_A\cap \mathcal{O}_B)\times( \mathcal{O}_A\cap \mathcal{O}_B)$, we can use it to endow $Z|_{\mathcal{O}_B}$ with Darboux coordinates ($\sigma_\pm$ is nondegenerate on $\mathcal{O}_B\cap \mathcal{O}_B$). If $z_1,z_2$ and $w_1,w_2$ are the natural respective coordinates on $X_-,X_+$, we can obtain Darboux coordinates $v_i=\log z_i$, $u_i=\log w_i$ for $i=1,2$: in fact,  $\omega:=\left(\sigma_+\times(-\sigma_-)\right)^{-1}=dv_2\wedge dv_1 +du_2\wedge du_1$ (we drop a factor of 2 for simplicity). Since $(p_+,p_-)$ takes $\lambda_Z$ to the graph of the diffeomorphism $\psi:X_-\rightarrow X_+$ it determines, we just have to describe such a diffeomorphism
\[ (v_1,v_2) \mapsto \left( \log (R^{-2}e^{v_1}),\log (-R^{-2}e^{\overline{v_2}})\right), \quad R^2=e^{v_1}\overline{e^{v_1}}+e^{v_2}\overline{e^{v_2}}\] 
in terms of a generating function. Note that $\lambda_Z(\mathcal{O}_B)\hookrightarrow Z$ is a real coisotropic submanifold but then, by dimensional reasons, it is Re-Lagrangian in $Z|_{\mathcal{O}_B}$ which is symplectic. Implementing the further change of coordinates $(x_1,x_2)=(u_2-u_1-\log(-1),u_2-\log(-1))$, we get that $\omega=dv_2\wedge dv_1-dx_2\wedge dx_1$ and
\[ \Gr(\psi)=\left\{(v_1,v_2,x_1,x_2)=\left(v_1,\overline{x_1}+\overline{v_1},x_1,x_1-\overline{v_1}-\log(1+e^{x_1+\overline{x_1}}) \right)\right\}. \]
Now one can consider the 1-form potential $\alpha=v_2dv_1 -x_2dx_1$ for $\omega$ and integrate $\Re (\alpha) $ over $\Gr (\psi)$, obtaining the generating function 
\begin{align*}
    f(v_1,\overline{v_1}, x_1 ,\overline{x_1})=\frac{1}{2}\left(v_1\overline{v_1}+\overline{v_1}x_1+\overline{x_1}v_1-\frac{1}{2}(x_1^2+\overline{x_1}^2)+\int^{x_1+\overline{x_1}} \log (1+e^t)dt\right).
\end{align*}
Similar expressions for such generalized K\"ahler potentials appear in the physics literature, see \cite{Ang2018} for example.


\begin{remark} One may also obtain Darboux coordinates for the real part of the bimodules $Z|_{\mathcal{O}_B}$ and $\overline{W}|_{\mathcal{O}_A}$ using the respective symplectomorphisms to (the real parts of) $\mathcal{G}_{\mathcal{A}_-}|_{\mathcal{O}_B}$ and $\mathcal{G}_{\mathcal{B}_-}|_{\mathcal{O}_A}$ determined by the bisections $\lambda_Z$ and $\lambda_W$. In fact, note that $\mp \sigma_\pm|_{\mathcal{O}_B}$ is isomorphic to $\pi:=x_1\frac{\partial}{\partial x_1} \wedge \frac{\partial}{\partial x_2}$ via the change of coordinates $(x_1,x_2)=(z_1,\log z_2)$ and a similar change of coordinates holds for $\mp \sigma_\pm|_{\mathcal{O}_A}$. As discussed in \cite[Example 3.8]{MR4466669}, $\mathcal{G}_{\mathcal{A}_-}|_{\mathcal{O}_B}$ integrates $\pi$ (in fact, this holds over $X_-$). The corresponding symplectic structure on $\mathcal{G}_{\mathcal{A}_-}|_{\mathcal{O}_B}$ is
\[ \omega:= da\wedge d(z_2+z_1b)-db\wedge dz_1 \]
and one can put it Darboux form using $(p_1,p_2,q_1,q_2)=(a,-b,z_2+bz_1,z_1)$.  
\end{remark} 

\printbibliography

\end{document}